\documentclass[reqno]{amsart}
\usepackage{mathrsfs}
\usepackage{color}
\usepackage{amsmath}
\usepackage{amsfonts}
\usepackage{amssymb}
\usepackage{graphicx}
\usepackage{hyperref}
\usepackage[numbers,sort&compress]{natbib}
\allowdisplaybreaks[4]

 \newtheorem{Theorem}{Theorem}[section]
 \newtheorem{Corollary}[Theorem]{Corollary}
 \newtheorem{Lemma}[Theorem]{Lemma}
 \newtheorem{Proposition}[Theorem]{Proposition}
 \newtheorem{Question}[Theorem]{Question}
 \newtheorem{Definition}[Theorem]{Definition}

 \newtheorem{Remark}[Theorem]{Remark}

 \numberwithin{equation}{section}

\begin{document}

\title[Optimal $L^2$ extension for singular metric on vector bundle]
 {Optimal $L^2$ extension for holomorphic vector bundles with singular hermitian metrics}

\author{Qi'an Guan}
\address{Qi'an Guan: School of
Mathematical Sciences, Peking University, Beijing 100871, China.}
\email{guanqian@math.pku.edu.cn}

\author{Zhitong Mi}
\address{Zhitong Mi: Institute of Mathematics, Academy of Mathematics
and Systems Science, Chinese Academy of Sciences, Beijing, China
}
\email{zhitongmi@amss.ac.cn}
\author{Zheng Yuan}
\address{Zheng Yuan: School of
Mathematical Sciences, Peking University, Beijing 100871, China.}
\email{zyuan@pku.edu.cn}

\thanks{}



\date{\today}

\dedicatory{}

\commby{}


\begin{abstract}
In the present paper, we study the properties of singular Nakano positivity of singular hermitian metrics on holomorphic vector bundles, and establish an optimal $L^2$ extension theorem for holomorphic vector bundles with singular hermitian metrics on weakly pseudoconvex K\"{a}hler manifolds, which is a unified version of the optimal $L^2$ extension theorems for holomorphic line bundles with singular hermitian metrics of Guan-Zhou  and Zhou-Zhu. As applications, we give a necessary condition for the holding of the equality in optimal $L^2$ extension theorem, and present singular hermitian holomorphic vector bundle  versions of some $L^2$ extension theorems with optimal estimate.
\end{abstract}

\maketitle

\section{Introduction}

We recall the $L^2$ extension problem (see \cite{DemaillyAG}, see also \cite{guan-zhou13ap}) as follows: let $Y$ be a complex subvariety of a complex manifold $M$; given a holomorphic object $f$ on $Y$ satisfying certain $L^2$ estimate on $Y$, finding a holomorphic extension $F$ of $f$ from $Y$ to $M$, together with a good or even optimal $L^2$ estimate of $F$ on $M$.

The existence part of $L^2$ extension problem was firstly solved by Ohsawa-Takegoshi \cite{OT87} and their result is called Ohsawa-Takegoshi $L^2$ extension theorem now. Since then, many mathematicians made contributions to generalizations of $L^2$ extension theorem and applications of the theorem in the study of several complex variables and complex geometry, e.g.  Berndtsson, Demailly, Ohsawa,  Siu, et al (see \cite{berndtsson1996,berndtsson annals,Demaillyshm,DemaillyManivel, D2016,Ohsawa2,Ohsawa3,Ohsawa4,Ohsawa5,berndtsson paun,DHP,HPS}).

The second part of $L^2$ extension problem was called the $L^2$ extension problem with optimal estimate or sharp $L^2$ extension problem (see \cite{zhou-abel}). Guan-Zhou-Zhu (see \cite{ZGZ}, see also \cite{GZZCRMATH}) firstly introduced a method of undetermined functions to study the sharp $L^2$ extension problem. For bounded
pseudoconvex domains in $\mathbb{C}^n$, Blocki \cite{Blocki-inv} developed the equation of undetermined functions, and obtained the optimal version of Ohsawa-Takegoshi's $L^2$ extension theorem in \cite{OT87}. As an application, Blocki \cite{Blocki-inv} got the inequality part of Suita conjecture for planar domains. Using undetermined functions method, Guan-Zhou (see \cite{GZsci}, see also \cite{guan-zhou CRMATH2012}) proved the optimal $L^2$ extension theorem with negligible weight on Stein manifolds, and obtained the inequality part of Suita conjecture for open Riemann surfaces, which is the original form of the inequality part of Suita conjecture in \cite{Suita1972}. In \cite{guan-zhou13ap}, Guan-Zhou established an $L^2$ extension theorem with optimal estimate in a general setting on Stein manifolds, which gave various optimal versions of $L^2$ extension theorem. As an application, Guan-Zhou \cite{guan-zhou13ap} proved the equality part of Suita conjecture, which finished the proof of Suita conjecture. Guan-Zhou \cite{guan-zhou13ap} also found a relation between the optimal $L^2$ extension theorem  and Berndtsson's log-plurisubharmonicity of fiberwise Bergman kernels, which was called Guan-Zhou method in Ohsawa's book \cite{Ohsawabook}. In \cite{ZhouZhu-jdg} and \cite{ZZ2019}, Zhou-Zhu proved the optimal $L^2$ extension theorem on weakly pseudoconvex K\"{a}hler  manifolds.
Using optimal $L^2$ extension theorem and Guan-Zhou method, Bao-Guan \cite{BG} generalized Berndtsson's log-plurisubharmonicity of fiberwise Bergman kernels, and gave a new approach to the sharp effectiveness result of strong openness property (for further research, see \cite{BG2},\cite{BGBoundary1} and \cite{BGBoundary2}).

Recall that $L^2$ extension theorems for  holomorphic line bundles with singular hermitian metrics  were mainly studied in the previous work and the results in \cite{ZZ2019} (see also \cite{ZhouZhu-jdg}) did not fully generalize the results in \cite{guan-zhou13ap}. It is natural to ask
\begin{Question}
\label{que1.1}Can one give a unified version of the optimal $L^2$ extension theorems (for holomorphic line bundles with singular hermitian metrics) of Guan-Zhou \cite{guan-zhou13ap} and Zhou-Zhu \cite{ZZ2019} (see also \cite{ZhouZhu-jdg}).
\end{Question}

To do this, we would like to consider holomorphic vector bundles  with singular hermitian metrics  in the present paper.
For vector bundles, the notation of  singular hermitian metrics and its corresponding positivity were introduced in many different ways
(see \cite{berndtsson paun}, \cite{CA},\cite{paun takayama}, \cite{Raufi1}, \cite{DNWZ}, \cite{inayama}).

In \cite{GMY-boundary5}, we modified the definition of singular hermitian metric on holomorphic vector bundles in \cite{CA}, and established the concavity property of minimal $L^2$ integrals of holomorphic vector bundles with singular hermitian metrics on weakly pseudoconvex K\"{a}hler manifolds.

In the present paper, following the notations in \cite{GMY-boundary5}, we present some properties of singular Nakano positivity of singular hermitian metrics on holomorphic vector bundles, and establish an optimal $L^2$ extension theorem for holomorphic vector bundles with singular hermitian metrics on weakly pseudoconvex K\"{a}hler manifolds which is a unified version of the optimal $L^2$ extension theorems (for holomorphic line bundles with singular hermitian metrics) of Guan-Zhou \cite{guan-zhou13ap} and Zhou-Zhu \cite{ZZ2019}, i.e., a positive answer to the Question \ref{que1.1}.

\subsection {Singular hermitian metrics on vector bundles}
Let $M$ be an $n-$dimensional complex manifold. Let $E$ be a rank $r$ holomorphic vector bundle over $M$ and $\bar{E}$ be the conjugate of $E$.

\begin{Definition}[see \cite{CA}, see also \cite{Raufi1}]\label{measurable metric}
Let $h$ be a section of the vector bundle $E^*\otimes {\bar{E}}^*$ with measurable coefficients, such that $h$ is an almost everywhere positive definite hermitian form on $E$; we call such an $h$ a \textit{measurable metric} on $E$.
\end{Definition}

We would like to use the following definition for singular hermitian metrics on vector bundles in this article which is a modified version of the definition in \cite{CA}.
\begin{Definition} Let $M$, $E$ and $h$ be as in Definition \ref{measurable metric} and $\Sigma\subset M$ be a closed set of measure zero. Let $\{M_j\}_{j=1}^{+\infty}$ be a sequence of relatively compact subsets of $M$ such that $M_1 \Subset  M_2\Subset  ...\Subset
M_j\Subset  M_{j+1}\Subset  ...$ and $\cup_{j=1}^{+\infty} M_j=M$. Assume that for each $M_j$, there exists a sequence of hermitian metrics $\{h_{j,s}\}_{s=1}^{+\infty}$ on $M_j$ of class $C^2$ such that

\centerline{$\lim\limits_{s\to+\infty}h_{j,s}=h$\ \ \ point-wisely on $M_j\backslash \Sigma$.}

We call the collection of data $(M,E,\Sigma,M_j,h,h_{j,s})$ a singular hermitian metric (s.h.m. for short) on $E$.

\label{singular metric}
\end{Definition}

 We use the following definition of singular version of Griffiths positivity in this article.

\begin{Definition}[see \cite{Raufi1}]\label{singular gri}Let $h$ be  a \textit{measurable metric} on $E$ satisfying that $h$ is everywhere positive definite hermitian form on $E$.

(1) $h$ is called singular Griffiths semi-negative if $|u|^2_h$ is plurisubharmonic for any local holomorphic section $u$ of $E$.

(2) A singular hermitian metric $h$ is Griffiths semi-positive if the dual metric $h^*$
is singular Griffiths semi-negative on $E^*$.
\end{Definition}

We use the following definition of singular version of Nakano positivity in this article. Let $\omega$ be a hermitian metric on $M$, $\theta$ be a hermitian form on $TM$ with continuous coefficients.
\begin{Definition}[see \cite{GMY-boundary5}]Let $(M,E,\Sigma,M_j,h,h_{j,s})$ be a s.h.m on $E$. We write:
$$\Theta_h(E)\ge^s_{Nak} \theta \otimes Id_E$$
if the following requirements are met.

For each $M_j$, there exist a sequence of continuous functions $\lambda_{j,s}$ on $\overline{M_j}$ and a continuous function $\lambda_j$ on $\overline{M_j}$ subject to the following requirements:
\par
(1.2.1) for any $x\in M_j:\ |e_x|_{h_{j,s}}\le |e_x|_{h_{j,s+1}},$ for any $s\in \mathbb{N}$ and any $e_x\in E_x$;
\par
(1.2.2) $\Theta_{h_{j,s}}(E)\ge_{Nak} \theta-\lambda_{j,s}\omega\otimes Id_E$ on $M_j$;
\par
(1.2.3)  $\lambda_{j,s}\to 0$ a.e. on $M_j$;
\par
(1.2.4) $0\le \lambda_{j,s}\le \lambda_j$ on $M_j$, for any $s$.

Especially, when $\theta=0$, i.e. $\Theta_h(E)\ge^s_{Nak} 0$, we call that $h$ is singular Nakano semi-positive.
\label{singular nak}
\end{Definition}

\begin{Remark}We prove that if $h$ is singular Nakano semi-positive in the sense of Definition \ref{singular nak}, then $h$ is Griffiths semi-positive in the sense of Definition \ref{singular gri}, see Proposition \ref{singular nak implies singular gri}.
\end{Remark}

\subsection {Optimal $L^2$ extension theorem on holomorphic vector bundles with singular hermitian metrics }

\begin{Definition}
	\label{def:neat}A function $\psi:M\rightarrow[-\infty,+\infty)$ on a complex manifold $M$ is said to be quasi-plurisubharmonic if $\psi$ is locally the sum of a plurisubharmonic function and a smooth function (or equivalently, if $i\partial\bar\partial\psi$ is locally bounded from below). In addition, we say that $\psi$ has neat analytic singularities if every point $z\in M$ possesses an open neighborhood $U$ on which $\psi$ can be written as
	$$\psi=c\log\sum_{1\le j\le N}|g_j|^2+v,$$
	where $c\ge0$ is a constant, $g_j\in\mathcal{O}(U)$ and $v\in C^{\infty}(U)$.
\end{Definition}

\begin{Definition}\label{locally lower bound}Let $M$ be a complex manifold and $E$ be a holomorphic vector bundle on $M$. Let $(M,E,\Sigma,M_j,h,h_{j,s})$ be a singular hermitian metric on $E$. We call $h$ is locally lower bounded if for every point $x\in M$, there exists an open set $\Omega_x\subset M$ such that $h$ can be written as $$h=h_x\eta_x$$
on $\Omega_x$, where
 $h_x$ is a singular hermitian metric on $E|_{\Omega_x}$ which is Nakano semi-positive in the sense of Definition \ref{singular nak} and $\eta_x$ is a smooth function on $\Omega_x$
\end{Definition}

\begin{Definition}
	If $\psi$ is a quasi-plurisubharmonic function on an $n$-dimensional complex manifold $M$, the multiplier ideal sheaf $\mathcal{I}(\psi)$ is the coherent analytic subsheaf of $\mathcal{O}_M$ defined by
	\begin{displaymath}
		\mathcal{I}(\psi)_z=\left\{f\in\mathcal{O}_{M,z}:\exists U\ni z,\,\int_U|f|^2e^{-\psi}d\lambda<+\infty \right\},
	\end{displaymath}
			where $U$ is an open coordinate neighborhood of $z$ and $d\lambda$ is the Lebesgue measure in the corresponding open chart of $\mathbb{C}^n$.
			
			We say that the singularities of $\psi$ are log canonical along the zero variety $Y=V(I(\psi))$ if $\mathcal{I}((1-\epsilon)\psi)|_Y=\mathcal{O}_M|_Y$ for any $\epsilon>0$.
\end{Definition}

Let $(M,\omega)$ be an $n$-dimensional K\"ahler manifold, and let $dV_{M,\omega}=\frac{1}{n!}\omega^n$ be the corresponding K\"ahler volume element.

\begin{Definition}
	\label{def:oshawa measure}Let $\psi$ be a quasi-plurisubharmonic function on $M$ with neat analytic singularities. Assume that the singularities of $\psi$ are log canonical along the zero variety $Y=V(I(\psi))$. Denote $Y^0=Y_{\rm{reg}}$ the regular point set of $Y$. If $g\in C_c(Y^0)$  and $\hat g\in C_c(M)$ satisfy $\hat g|_{Y^0}=g$ and $(supp\,\hat{g})\cap Y=Y^0$, we set
	\begin{equation}
\label{eq:0627e}
		\int_{Y^0}gdV_{M,\omega}[\psi]=\limsup_{t\rightarrow+\infty}\int_{\{-t-1<\psi<-t\}}\hat ge^{-\psi}dV_{M,\omega}.
	\end{equation}
\end{Definition}
\begin{Remark}[see \cite{D2016}]\label{r:measure}
	By Hironaka's desingularization theorem, it is not hard to see that the limit in the right of equality \eqref{eq:0627e} does not depend on the continuous extension $\hat g$ and $dV_{M,\omega}[\psi]$ is well defined on $Y^0$ .
\end{Remark}

We would like to introduce a class of functions before introducing our main result.

\begin{Definition}\label{class gTdelta}Let $T\in (-\infty,+\infty)$ and $\delta\in (0,+\infty)$. Let $\mathcal{G}_{T,\delta}$ be the class of functions $c(t)$ which satisfies the following statements,\\
(1) $c(t)$ is a continuous positive function on $[T,+\infty)$,\\
(2)  $\int_{T}^{+\infty}c(t)e^{-t}dt<+\infty$,\\
(3) for any $t> T$, the following equality holds,
\begin{equation}\label{integarl condition 1}
\begin{split}
    &{\left(\frac{1}{\delta}c(T)e^{-T}+\int_{T}^{t}c(t_1)e^{-t_1}dt_1\right)}^2 > \\
     &c(t)e^{-t}\bigg(\int_{T}^{t}(\frac{1}{\delta}c(T)e^{-T}+\int_{T}^{t_2}c(t_1)e^{-t_1}dt_1)dt_2
     +\frac{1}{{\delta}^2}c(T)e^{-T}\bigg).
\end{split}
\end{equation}
\end{Definition}

\begin{Theorem} [Main theorem] \label{main result}
Let $c(t)\in \mathcal{G}_{T,\delta}$, where $\delta<+\infty$. Let $(M,\omega)$ be a weakly pseudoconvex K\"ahler manifold. Let $\psi<-T$ be a quasi-plurisubharmonic function on $M$ with neat analytic singularities. Let $Y:=V(\mathcal{I}(\psi))$ and assume that $\psi$ has log canonical singularities along $Y$. Let $\varphi$ be a Lebesgue measurable function on $M$ such that $\varphi+\psi$ is a quasi-plurisubharmonic function. Let $E$ be a holomorphic vector bundle on $M$ with rank $r$. Let $(M,E,\Sigma,M_k,h,h_{k,s})$ be a singular metric on $E$. Assume that \\
(1) $\Theta_{h}(E)\ge^s_{Nak} 0$ on $M$ in the sense of Definition \ref{singular nak} and $he^{-\varphi}$ is locally lower bounded;\\
(2) $\sqrt{-1}\partial\bar{\partial}\varphi+\sqrt{-1}\partial\bar{\partial}\psi\ge 0$ on $M\backslash\{\psi=-\infty\}$ in the sense of currents;\\
(3)  $s(-\psi)\big(\sqrt{-1}\partial\bar{\partial}\varphi+\sqrt{-1}\partial\bar{\partial}\psi\big)
+\sqrt{-1}\partial\bar{\partial}\psi\ge 0$ on $M\backslash\{\psi=-\infty\}$ in the sense of currents, where
$$s(t):=\frac{\int^t_T\bigg(\frac{1}{\delta}c(T)e^{-T}+\int^{t_2}_T c(t_1)e^{-t_1}dt_1\bigg)dt_2+\frac{1}{\delta^2}c(T)e^{-T}}{\frac{1}{\delta}c(T)e^{-T}+\int^t_T
c(t_1)e^{-t_1}dt_1}.$$

Then for every section $f \in H^0(Y^0,(K_M\otimes E)|_{Y^0})$ on $Y^0=Y_{\text{reg}}$ such that
\begin{equation}\label{mainth:ohsawa measure finite}
  \int_{Y^0}|f|^2_{\omega,h}e^{-\varphi}dV_{M,\omega}[\psi]<+\infty,
\end{equation}
there exists a section $F\in H^0(M,K_M\otimes E)$ such that $F|_{Y_0}=f$ and
\begin{equation}\label{mainth:l2 estimate}
 \int_M c(-\psi)|F|^2_{\omega,h}e^{-\varphi}dV_{M,\omega}\le \left(\frac{1}{\delta}c(T)e^{-T}+\int_{T}^{+\infty}c(t_1)e^{-t_1}dt_1\right)\int_{Y^0}|f|^2_{\omega,h}e^{-\varphi}dV_{M,\omega}[\psi].
\end{equation}
\end{Theorem}
\begin{Remark}
Note that for any section $f \in H^0(Y^0,(K_M\otimes E)|_{Y^0})$ on $Y^0=Y_{\text{reg}}$, the integral
$$\int_{Y^0}|f|^2_{\omega,h}e^{-\varphi}dV_{M,\omega}[\psi]$$ is independent of the choice of $\omega$.
\end{Remark}

\begin{Remark}In \cite{D2016}, Demailly obtained an $L^2$ extension theorem for holomorphic line bundles with singular hermitian metrics on weakly pseudoconvex K\"ahler manifolds. In \cite{ZZ2019} (see also \cite{ZhouZhu-jdg}), Zhou-Zhu proved the optimal version of Demailly's result. Theorem \ref{main result} gives an optimal $L^2$ extension theorem for holomorphic vector bundles with singular hermitian metrics on weakly pseudoconvex K\"ahler manifolds.
\end{Remark}

We would like to  introduce a  class of functions as follows.
\begin{Definition}\label{class gT}Let $T\in [-\infty,+\infty)$. Let $\mathcal{G}_{T}$ be the class of functions $c(t)$ which satisfies the following statements,\\
(1) $c(t)$ is a continuous positive function on $(T,+\infty)$,\\
(2)  $\int_{T}^{+\infty}c(t)e^{-t}dt<+\infty,$\\
(3) for any $t>T$, the following equality holds,
\begin{equation}\nonumber
\begin{split}
    {\left(\int_{T}^{t}c(t_1)e^{-t_1}dt_1\right)}^2 >
     c(t)e^{-t}\bigg(\int_{T}^{t}(\int_{T}^{t_2}c(t_1)e^{-t_1}dt_1)dt_2
     \bigg).
\end{split}
\end{equation}
\end{Definition}
We deduce the following optimal $L^2$ extension Theorem from Theorem \ref{main result}.
\begin{Theorem}\label{main result 1}
Let $c(t)\in \mathcal{G}_{T}$. Let $(M,\omega)$ be a weakly pseudoconvex K\"ahler manifold. Let $\psi<-T$ be a plurisubharmonic function on $X$ with neat analytic singularities. Let $Y:=V(\mathcal{I}(\psi))$ and assume that $\psi$ has log canonical singularities along $Y$. Let $\varphi$ be a Lebesgue measurable function on $M$ such that $\varphi+\psi$ is a plurisubharmonic function on $M$. Let $E$ be a holomorphic vector bundle on $M$. Let $(M,E,\Sigma,M_k,h,h_{k,s})$ be a singular metric on $E$. Assume that $\Theta_{h}(E)\ge^s_{Nak} 0$ on $M$ in the sense of Definition \ref{singular nak} and $he^{-\varphi}$ is locally lower bounded.

Then for every section $f \in H^0(Y^0,(K_M\otimes E)|_{Y^0})$ on $Y^0=Y_{\text{reg}}$ such that
\begin{equation}\label{mainth:ohsawa measure finite 2}
  \int_{Y^0}|f|^2_{\omega,h}e^{-\varphi}dV_{M,\omega}[\psi]<+\infty,
\end{equation}
there exists a section $F\in H^0(M,K_M\otimes E)$ such that $F|_{Y_0}=f$ and
\begin{equation}\label{mainth:l2 estimate 2}
 \int_M c(-\psi)|F|^2_{\omega,h}e^{-\varphi}dV_{M,\omega}\le \left(\int_{T}^{+\infty}c(t_1)e^{-t_1}dt_1\right)\int_{Y^0}|f|^2_{\omega,h}e^{-\varphi}dV_{M,\omega}[\psi].
\end{equation}
\end{Theorem}

\subsection{Applications }
In this section, we introduce some applications of Theorem \ref{main result} and   Theorem \ref{main result 1}.
\subsubsection{Equality in optimal $L^2$ extension problem: a necessary condition}
As an application of  Theorem \ref{main result 1}, we give a necessary condition for the holding of the equality in optimal $L^2$ extension theorem.

Let $(M,\omega)$ be a weakly pseudoconvex K\"ahler manifold. Let $\psi<0$ be a plurisubharmonic function on $M$ with neat analytic singularities. Let $Y:=V(\mathcal{I}(\psi))$ and assume that $\psi$ has log canonical singularities along $Y$. Let $\varphi$ be a Lebesgue measurable function on $M$ such that $\varphi+\psi$ is a plurisubharmonic function on $M$. Let $E$ be a holomorphic vector bundle on $M$. Let $(M,E,\Sigma,M_k,h,h_{k,s})$ be a singular metric on $E$.
Assume that
$\Theta_{h}(E)\ge^s_{Nak} 0$ on $M$ in the sense of Definition \ref{singular nak} and $he^{-\varphi}$ is locally lower bounded.

Let $c(t)$ be a function on $(0,+\infty)$ which satisfies the following statements,\\
(1) $c(t)$ is a continuous positive function on $(0,+\infty)$,\\
(2) $c(t)e^{-t}$ is decreasing with respect to $t$,\\
(3) $\int_{0}^{+\infty}c(t)e^{-t}dt<+\infty$.

Let $f \in H^0(Y^0,(K_M\otimes E)|_{Y^0})$ on $Y^0=Y_{\text{reg}}$ satisfies that
\begin{equation}\label{ch1,sec3,for1}
 \int_{Y^0}|f|^2_{\omega,h}e^{-\varphi}dV_{M,\omega}[\psi]<+\infty.
\end{equation}
Denote that
$$\|f\|_{Y_0,L^2}:=\bigg(\int_0^{+\infty}c(t)e^{-t}dt\bigg) \int_{Y^0}|f|^2_{\omega,h}e^{-\varphi}dV_{M,\omega}[\psi],$$
 and $$\|F\|_{M,L^2}:=\int_M c(-\psi)|F|^2_{\omega,h}e^{-\varphi}dV_{M,\omega},$$
for any $F\in H^0\big(M,K_M\otimes E\big)$.

When $E$ is a trivial line bundle and $h$ is a singular hermitian metric on $E$, Guan-Mi \cite{GMPEK} considered the following question:\par
\emph{Under which condition, does the equality in optimal $L^2$ extension theorem
$\|f\|_{Y_0,L^2}=\inf\{\|F\|_{M,L^2}: F$ is a holomorphic extension of $f$ from $Y_0$ to $M\}$ hold?}

When $M$ is an open Riemann surface which admits a nontrivial Green function and
$E$ is a trivial line bundle with trivial (or harmonic) weight $\varphi$, the characterization of the holding of the equality is equivalent to  Suita conjecture (or extended Suita conjecture) which were proved by Guan-Zhou in \cite{guan-zhou13ap}.
When $E$ is a trivial line bundle with a singular weight $\varphi$, using the concavity property of minimal $L^2$ integrals, Guan-Mi \cite{GMPEK} gave a necessary condition for the holding of the equality and  established a characterization to the question when $M$ is an open Riemann surface which admits a nontrivial Green function
(for recent progress, see
\cite{GY-concavity1,GMY-boundary2,GY-concavity3,GY-concavity4,BGY-concavity5,BGY-concavity6,BGMY7}).

For the case $E$ is a holomorphic vector bundle with a singular hermitian metric, using Theorem \ref{main result 1} and the concavity property of minimal $L^2$ integrals in \cite{GMY-boundary5} (see also Theorem \ref{concavity of min l2 int}), we have the following necessary condition for the holding of the equality
$\|f\|_{Y_0,L^2}=\inf\{\|F\|_{M,L^2}: F$ is a holomorphic extension of $f$ from $Y_0$ to $M\}$.

\begin{Theorem}\label{nec in equ of L2 ext}  Let $M,\psi,Y,\varphi,E,h$ be as above.

Let $f \in H^0(Y^0,(K_M\otimes E)|_{Y^0})$ on $Y^0=Y_{\text{reg}}$ such that
\begin{equation}\label{ch1,sec3,for1}
 \int_{Y^0}|f|^2_{\omega,h}e^{-\varphi}dV_{M,\omega}[\psi]<+\infty.
\end{equation}
Assume that the equality $\|f\|_{Y_0,L^2}=\inf\{\|F\|_{M,L^2}: F$ is a holomorphic extension of $f$ from $Y_0$ to $M\}$ holds.

Then there exists a unique $E$-valued holomorphic $(n,0)$-form $F$ on
$M$ such that $F|_{Y_0}=f$ and for any $t\ge 0$, the norm $\|F\|_{\{\psi<-t\},L^2}$ of $F$ is minimal along all holomorphic  extension of $f$ from $Y_0$ to $\{\psi<-t\}$. Moreover, we have
\begin{equation}\nonumber
\int_{\{\psi<-t\}} c(-\psi)|F|^2_{\omega,h}e^{-\varphi}dV_{M,\omega}=\bigg(\int_t^{+\infty}c(t_1)e^{-t_1}dt_1\bigg) \int_{Y^0}|f|^2_{\omega,h}e^{-\varphi}dV_{M,\omega}[\psi].
\end{equation}

\end{Theorem}

When $M$ is a Stein manifold and $(E,h)$ is a trivial line bundles with singular hermitian metric, Theorem \ref{nec in equ of L2 ext} can be referred to \cite{GMPEK}.
\subsubsection{Optimal estimate of $L^2$ extension theorem of Ohsawa for holomorphic vector bundles with singular hermitian metrics.}
\

Let $X$ be a complex manifold. Let $(E,h)$ be a holomorphic vector bundle over $X$ with a smooth hermitian metric $h$. For any local section $f$ of $K_X\otimes E$, denote
$$\{f,f\}_h:=<e,e>_h\sqrt{-1}^{n^2}f_1\wedge \bar{f_1},$$
where $f=f_1\otimes e$ locally.

In \cite{Ohsawa2}, Ohsawa proved following $L^2$ extension theorem.
\begin{Theorem}[see \cite{Ohsawa2}] \label{ohsawa2}Let $X$ be an $n$-dimensional Stein manifold, $Y\subset X$ a closed complex submanifold of codimension $m$, and $(E,h)$ be a Nakano-semipositive vector bundle over $X$, where $h$ is smooth. Let $\varphi$ be any plurisubharmonic function on $X$ and let $s_1,\cdots,s_m$ be holomorphic functions on $X$ vanishing on $Y$ and $ds_1\wedge\cdots\wedge ds_m\neq 0$ on $Y$. Then given a holomorphic $E$-valued $(n-m,0)$-form $g$ on $Y$ with
$$\int_Y \{g,g\}_he^{-\varphi}<+\infty,$$
there exists for any $\epsilon>0$, a holomorphic $E$-valued $(n,0)$-form $G_{\epsilon}$ on $X$ which coincides with $g\wedge ds_1\wedge\cdots\wedge ds_m$ on $Y$ and satisfies
$$\int_{X}e^{-\varphi}(1+|s|^2)^{-m-\epsilon}\{G_{\epsilon},G_{\epsilon}\}_h\le \epsilon^{-1}C_m\int_Y \{g,g\}_he^{-\varphi},$$
where $|s|^2=\sum_{i=1}^{m}|s_i|^2$ and $C_m$ is a positive number only depends on $m$.
\end{Theorem}

Let $c_{\infty}(t):=(1+e^{\frac{-t}{m}})^{-m-\epsilon}$. Note that $c_{\infty}(t)$ belongs to class $\tilde{\mathcal{G}}_{+\infty}$ and $\int_{-\infty}^{+\infty}c_{\infty}(t)e^{-t}dt=m\sum_{j=0}^{m-1}C_{m-1}^j(-1)^{m-1-j}\frac{1}{m-1-j+\epsilon}<+\infty$.
Using Theorem \ref{main result 1} (take $\psi=m\log|s|^2$), we have optimal estimate of Theorem \ref{ohsawa2} for holomorphic vector bundles with singular hermitian metrics as follows.
\begin{Corollary}Let $X,Y,E,\varphi$ be as in Theorem \ref{ohsawa2}. Let $h_E$ be a singular hermitian metric on a holomorphic vector bundle $E$ with rank $r$ such that $\Theta_{h_E}(E)\ge^s_{Nak} 0$ on $M$ in the sense of Definition \ref{singular nak}. Then given any holomorphic $E$-valued $(n-m,0)$-form $g$ on $Y$ with
$$\int_Y \{g,g\}_{h_E}e^{-\varphi}<+\infty,$$
there exists for any $\epsilon>0$, a holomorphic $E$-valued $(n,0)$-form $G_{\epsilon}$ on $X$ which coincides with $g\wedge ds_1\wedge\cdots\wedge ds_m$ on $Y$ and satisfies
\begin{equation}
\begin{split}
&\int_{X}e^{-\varphi}(1+|s|^2)^{-m-\epsilon}\{G_{\epsilon},G_{\epsilon}\}_{h_E}\\
\le &\big( m\sum_{j=0}^{m-1}C_{m-1}^j(-1)^{m-1-j}\frac{1}{m-1-j+\epsilon}\big) \frac{(2\pi)^m}{m!}\int_Y \{g,g\}_{h_E}e^{-\varphi}.
\end{split}
\end{equation}
\end{Corollary}

Let $M$ be a Stein manifold, and $S$ be an analytic hypersurface on $M$. Let $S=s^{-1}(0)$, where $s$ is a holomorphic function on $M$ such that $ds$ does not vanish identically on $S$. Denote $S_{reg}:=\{x\in S: ds(x)\neq 0\}$. Let  $\psi$ be a plurisubharmonic function on $M$. Assume that $\Psi:=\log|s|^2+\psi$ is a plurisubharmonic function on $M$. Denote $-T:=\sup_{M}\Psi$ and assume that $T$ is a real number.

For any $k\ge 1$, let $c_k:=(\frac{i}{2})^{k^2}$. When $M=D$ is a bounded pseudovoncex domain in $\mathbb{C}^n$, Ohsawa \cite{Ohsawa3} proved an $L^2$ extension theorem with
negligible weights as follows.

\begin{Theorem}[see \cite{Ohsawa3}] \label{negligible weights}Let $M=D$ be  a pseudovoncex domain in $\mathbb{C}^n$. Let $S,\Psi$ be above. Let $\varphi$ be a plurisubharmonic function on $M$. Then there exists a  constant $C_T>0$ (only depends on $T$) such that, for any holomorphic section $f$ of $K_{{S_{reg}}}$ on $S_{reg}$ satisfying
$$c_{n-1}\int_{S_{reg}}f\wedge \bar{f}e^{-\varphi-\psi}<+\infty,$$
there exists a holomorphic section $F$ of $K_M$ on $M$ satisfying $F=f\wedge ds$ on $S_{reg}$ and
$$c_{n}\int_{M}F\wedge \bar{F} e^{-\varphi}
\le 2\pi C_T c_{n-1}\int_{S_{reg}}f\wedge \bar{f}e^{-\varphi-\psi}$$
\end{Theorem}
When $M$ is a general Stein manifold and $\varphi$ is a Lebesgue measurable function such that $\varphi+\psi$ is plurisubharmonic, assume $T=0$, Guan-Zhou \cite{guan-zhou CRMATH2012} proved $C_0=1$ in Theorem \ref{negligible weights}, which is optimal (for related research, see \cite{GZZCRMATH} and \cite{ZGZ}).

Note that when $\varphi$ is plurisubharmonic, $e^{-\varphi}$ in Theorem \ref{negligible weights} can be viewed as a singular hermitian metric on trivial line bundle $M\times \mathbb{C}$. Using Theorem \ref{main result 1}, we show that Theorem \ref{negligible weights} holds for holomorphic vector bundles with singular hermitian metrics.
\begin{Corollary} \label{ohsawa3 sin ver}Let $M,S,\Psi<0,\psi$ be as in  Theorem \ref{negligible weights}. Let $E$ be a holomorphic vector bundle $E$ with rank $r$ with singular hermitian metric $h_{E}$. Assume that $h_Ee^{-\psi}$  is singular Nakano semi-positive in the sense of Definition \ref{singular nak}.

Then  for any holomorphic section $f$ of $K_{{S_{reg}}}\otimes E|_{S_{reg}}$ on $S_{reg}$ satisfying
$$c_{n-1}\int_{S_{reg}}\{f,f\}_{h_E}e^{-\psi}<+\infty,$$
there exists a holomorphic section $F$ of $K_M\otimes E$ on $M$ satisfying $F=f\wedge ds$ on $S_{reg}$ and
$$c_{n}\int_{M}\{F,F\}_{h_E}
\le 2\pi  c_{n-1}\int_{S_{reg}}\{f,f\}_{h_E}e^{-\psi}.$$

\end{Corollary}

\subsubsection{Optimal $L^2$ extension theorem for holomorphic vector bundles with singular hermitian metrics  on projective families.}
\

Siu established an $L^2$
extension theorem on projective families in \cite{siu02}. P\u{a}un \cite{paun07} reformulated Siu's theorem as below.
 In \cite{berndtsson annals}, Berndtsson proved a related result for K\"{a}hler families.
\begin{Theorem}[see \cite{siu02}, \cite{paun07} or \cite{berndtsson annals}]\label{siu's extension}Let $M$ be a projective family (or K\"{a}hler family due to \cite{berndtsson annals}) fibred over the unit ball in $(\mathbb{C}^m,z)$, with compact fibers $M_t$. Let $(L,h_L)$ be a holomorphic line bundle on $M$ with a  singular hermitian metric $h$ of semipositive curvature. Let $u$ be a holomorphic section of $K_{M_0}\otimes L$ over $M_0$ such that
$$\int_{M_0}\{u,u\}_h< +\infty.$$
Then there is a holomorphic section $\tilde{u}$ of $K_M\otimes L$ over $M$ such that $\tilde{u}|_{M_0}=u\wedge dz$, and
$$\int_{M}\{u,u\}_h\le C_b\int_{M_0}\{u,u\}_h,$$
where the constant $C_b>0$ is universal.
\end{Theorem}
In \cite{paun07}, P\u{a}un takes $C_b$ around $200$.

Replace $(L,h_L)$ by $(E,h_E)$ where $h_E$ is a singular metric on a holomorphic vector bundle $E$ with rank $r$ such that $\Theta_{h}(E)\ge^s_{Nak} 0$ on $M$ in the sense of Definition \ref{singular nak}. Then let $c(t)\equiv 1$, $\psi=2m\log|z|$ and $\varphi=0$ in Theorem \ref{main result 1}, we obtain an optimal $L^2$ extension theorem for holomorphic vector bundles with singular hermitian metrics  on projective families.
\begin{Corollary}Theorem \ref{siu's extension} holds for holomorphic vector bundles $(E,h_E)$, where $h_{E}$ is a singular Nakano semi-positive metric on $E$ in the sense of Definition \ref{singular nak}, with optimal estimate $C_b=\frac{2^m\pi^m}{m!}$.
\end{Corollary}

\subsubsection{Optimal $L^2$ extension theorem for holomorphic vector bundles with singular hermitian metrics  on weakly pseudoconvex K\"{a}hler manifolds.}
\

In \cite{ZZ2019} (see also \cite{ZhouZhu-jdg}), Zhou-Zhu proved optimal $L^2$ extension theorem for holomorphic line bundles with singular hermitian metrics  on weakly pseudoconvex K\"{a}hler manifolds.

 Let $T\in (-\infty,+\infty)$ and $\delta\in (0,+\infty)$. Recall that $\mathcal{G}_{T,\delta}$ is the class of functions $c(t)$ which satisfies the following statements,\\
(1) $c(t)$ is a continuous positive function on $[T,+\infty)$,\\
(2)  $\int_{T}^{+\infty}c(t)e^{-t}dt<+\infty$,\\
(3) for any $t> T$, the following equality holds,
\begin{equation}\nonumber
\begin{split}
    &{\left(\frac{1}{\delta}c(T)e^{-T}+\int_{T}^{t}c(t_1)e^{-t_1}dt_1\right)}^2 > \\
     &c(t)e^{-t}\bigg(\int_{T}^{t}(\frac{1}{\delta}c(T)e^{-T}+\int_{T}^{t_2}c(t_1)e^{-t_1}dt_1)dt_2
     +\frac{1}{{\delta}^2}c(T)e^{-T}\bigg).
\end{split}
\end{equation}
The number $-T$, $\frac{1}{\delta}$ and function $c(t)$ are equal to the number $\alpha_0$, $\alpha_1$ and function $\frac{1}{R(-t)e^{-t}}$ in \cite{ZZ2019}. We use $-T$, $\frac{1}{\delta}$ and $c(t)$ here for the simplicity of notations.

Zhou-Zhu's optimal $L^2$ extension theorem for holomorphic line bundles with singular hermitian metrics is as follows.
\begin{Theorem}[see \cite{ZZ2019}]\label{zhou-zhu result} Let $c(t)\in \mathcal{G}_{T,\delta}$ for some $\delta<+\infty$ be a smooth function on $[T,+\infty)$ such that $c(t)e^{-t}$ is decreasing with respect to $t$ near $+\infty$ and $\liminf_{t\to +\infty}c(t)>0$. Let $(M,\omega)$ be a weakly pseudoconvex K\"ahler manifold. Let $\psi$ be a quasi-plurisubharmonic function on $M$ with neat analytic singularities. Let $Y:=V(\mathcal{I}(\psi))$ and assume that $\psi$ has log canonical singularities along $Y$. Let $L$ be a holomorphic vector bundle on $M$ with a singular hermitian metric $h_L$, which is locally written as $e^{-\phi_L}$ for some quasi-plurisubharmonic function.  Assume that \\
(I) $\sqrt{-1}\Theta_{L}+\sqrt{-1}\partial\bar{\partial}\psi\ge 0$ on $M\backslash\{\psi=-\infty\}$ in the sense of currents;\\
 and
there exists a continuous function $\alpha< -T$ on $M$ such that the following two hold:\\
(II)  $\big(\sqrt{-1}\Theta_{L}+\sqrt{-1}\partial\bar{\partial}\psi\big)
+\frac{1}{s(-\alpha)}\sqrt{-1}\partial\bar{\partial}\psi\ge 0$ on $M\backslash\{\psi=-\infty\}$ in the sense of currents;\\
(III) $\psi\le \alpha$,
where
$$s(t):=\frac{\int^t_T\bigg(\frac{1}{\delta}c(T)e^{-T}+\int^{t_2}_T c(t_1)e^{-t_1}dt_1\bigg)dt_2+\frac{1}{\delta^2}c(T)e^{-T}}{\frac{1}{\delta}c(T)e^{-T}+\int^t_T
c(t_1)e^{-t_1}dt_1}.$$

Then for every section $f \in H^0(Y^0,(K_M\otimes L)|_{Y^0})$ on $Y^0=Y_{\text{reg}}$ such that
\begin{equation}\label{mainth:ohsawa measure finite}
  \int_{Y^0}|f|^2_{\omega,h_L}e^{-\varphi}dV_{M,\omega}[\psi]<+\infty,
\end{equation}
there exists a section $F\in H^0(M,K_M\otimes L)$ such that $F|_{Y_0}=f$ and
\begin{equation}\label{mainth:l2 estimate}
 \int_M c(-\psi)|F|^2_{\omega,h_L}e^{-\varphi}dV_{M,\omega}\le \left(\frac{1}{\delta}c(T)e^{-T}+\int_{T}^{+\infty}c(t_1)e^{-t_1}dt_1\right)\int_{Y^0}|f|^2_{\omega,h_L}e^{-\varphi}dV_{M,\omega}[\psi].
\end{equation}
\end{Theorem}

Replace $(L,h_L)$ by $(E\otimes L, h_E\otimes h_L)$, where $h_{E}$ is a singular Nakano semi-positive metric on $E$ in the sense of Definition \ref{singular nak} and $h_L$ satisfies condition (I) and (II) in Theorem \ref{zhou-zhu result}.
As $s(-\psi)\ge s(-\alpha)$, then condition (II) in Theorem \ref{zhou-zhu result} implies condition (2) in Theorem \ref{main result}.

It follows from Theorem \ref{main result} and Remark \ref{rem:main result for line bdl} that we have following result
\begin{Corollary}Let $M,Y,\psi$ be as Theorem \ref{zhou-zhu result}. Let $c(t)\in \mathcal{G}_{T,\delta}$. Theorem \ref{zhou-zhu result} holds for $(E\otimes L, h_E\otimes h_L)$, where $h_{E}$ is a singular Nakano semi-positive metric on $E$ in the sense of Definition \ref{singular nak} and $h_L$ satisfies conditions (I),(II) and (III).
\end{Corollary}

\section{Preparations}
\subsection{Some properties of singular hermitian metrics on vector bundles}
Let $M$ be an $n-$dimensional complex manifold.
Let $h$ be a $C^2$ smooth hermitian metric on a holomorphic vector bundle $E$ over $M$. We recall the following proposition of Griffiths semi-positivity in the smooth case.

\begin{Proposition} [see \cite{Raufi1}]  \label{equivalent of Gri in smooth case}The following four statements are equivalent:\\
(1) $h$ is Griffiths semi-positive;\\
(2) $h^*$ is Griffiths semi-negative where $h^*$ is the dual metric of $h$ on $E^*$;\\
(3) $|u|^2_{h^*}$ is a plurisubharmonic function for any local holomorphic section $u$ of $E^*$;\\
(4) $\log |u|^2_{h^*}$ is plurisubharmonic function for any local holomorphic section $u$ of $E^*$.
\end{Proposition}

The following proposition will be used in the proof of Proposition \ref{singular nak implies singular gri}.
\begin{Proposition}[see Theorem 3.8 in \cite{Boucksom note}] \label{approx of qpsh} Let $X$ be a complex manifold. Let $u$ be a quasi-plurisubharmonic function on $X$. Given finitely many closed, real $(1,1)$-forms $\theta_\alpha$ such that $\theta_{\alpha}+dd^c u\ge 0$ for all $\alpha$. Suppose either that $X$ is strongly pseudoconvex, or that $\theta_\alpha>0$ for all $\alpha$. Then we can find a sequence $u_j\in C^{+\infty}(X)$ with the following properties:\\
(1) $u_j$ converges pointwisely to $u$;\\
(2) for each relatively compact open subset $U\Subset X$, there exists $j_U\ge 1$ such that the sequence $(u_j)$ is decreasingly convergent to $u$ with $\theta_\alpha+dd^c u_j>0$ for $j\ge j_U$.
\end{Proposition}

 Let $(M,E,\Sigma,M_j,h,h_{j,s})$ be a s.h.m in the sense of Definition \ref{singular metric}. We have the following property of singular Nakano semi-positivity.
\begin{Proposition}\label{singular nak implies singular gri} Assume that $h$ is singular Nakano semi-positive, i.e. $\Theta_h(E)\ge^s_{Nak} 0$ in the sense of Definition \ref{singular nak}. Then $h$ is Griffiths semi-positive in the sense of Definition \ref{singular gri}.
\end{Proposition}
\begin{proof} Let $(U,z) \Subset M$ be a local coordinate open subset of $M$. Let $u\in H^0(U,E^*)$. By definition, it suffices to prove that $|u|^2_{h^*}$ is a plurisubharmonic function on $U$.

Since the case is local, it follows from $\Theta_h(E)\ge^s_{Nak} 0$ that there exists a sequence of $C^2$ smooth metric $h_s$ ($s\ge 1$) convergent point-wisely to $h$ on a neighborhood of $\overline{U}$ which satisfies
\par
(1) for any $x\in U:\ |e_x|_{h_{s}}\le |e_x|_{h_{s+1}},$ for any $s\ge 1$ and any $e_x\in E_x$;
\par
(2) $\Theta_{h_{s}}(E)\ge_{Nak} -\lambda_{s}\omega\otimes Id_E$ on $U$;
\par
(3)  $\lambda_{s}\to 0$ a.e. on $U$, where $\lambda_{s}$ is a sequence of continuous functions on $\overline{U}$;
\par
(4)  $0\le \lambda_{s}\le \lambda_0$ on $U$, for any $s\ge 1$, where $\lambda_0$ is a continuous function $\overline{U}$.

We now prove that $|u|^2_{h^*}$ is upper semi-continuous. Let $h^*$ be the dual metric of $h$ and $h^*_s$ be the dual metric of $h_s$ on $E^*$. Note that $h^*_s$ is also $C^2$ smooth, and it follows from $h_s$ increasing converges to $h$ that we know $h^*_s$ is decreasing convergent to $h^*$. Then $|u|^2_{h^*_s}$ is a $C^2$ smooth function on $U$ and decreasingly convergent to $|u|^2_{h^*}$ as $s \to +\infty$. Then we have
\begin{equation}\label{fm:semi-continuous 1}
\begin{split}
   \limsup_{z\to z_0} |u|^2_{h^*}(z) &\le  \limsup_{z\to z_0} |u|^2_{h^*_s}(z)
     = |u|^2_{h^*_s}(z_0).
\end{split}
\end{equation}
Letting $s\to +\infty$ in inequality \eqref{fm:semi-continuous 1}, we have $\limsup_{z\to z_0} |u|^2_{h^*}(z) \le |u|^2_{h^*}(z_0)$. Hence  $|u|^2_{h^*}$ is upper semi-continuous.

We may assume that under the local coordinate, $U \cong B$, where $B\subset \mathbb{C}^n$ is the unit ball. Now we prove that for any complex line $L$ in $\mathbb{C}^n$, $|u|^2_{h^*}|_{B\cap L}$ is subharmonic. Since the case is local, we can assume that $L=\tau z_1$, where $\tau\in \mathbb{C}$ is a complex number and $z_1\in \mathbb{C}^n$ is a unit vector. Then $(B\cap L)\cong \Delta:=\{\tau:|\tau|<1\}\subset \mathbb{C}.$

We firstly assume that $\lambda_s|_{B\cap L}$ converges to $0$ a.e. as $s\to +\infty$. It follows from $\Theta_{h_{s}}(E)\ge_{Nak} -\lambda_{s}\omega\otimes Id_E$ on $B$ that we know that $$\Theta_{h_{s}|_{B\cap L}}(E|_{B\cap L})\ge_{Nak} -(\lambda_{s}|_{B\cap L}\omega\otimes Id_E)|_{B\cap L}$$
 on $B\cap L$. Let $G_{\Delta}(w,t)$ be the Green function on $\Delta$ with pole $t\in \Delta$. Let $z\in \Delta$, and for any $s\ge 0$ denote
 $$\varphi_s(z):=\frac{i}{\pi}\int_{t\in \Delta}2G_{\Delta}(z,t)(\lambda_{s}\omega)|_{B\cap L}(t).$$
 It follows from $\frac{i}{\pi}\partial_z\bar{\partial}_z 2G_{\Delta}(z,t)=2[t]$ for any fixed $t\in \Delta$, where $[t]$ is the (1,1)-current of integration over a point $t$ that we know that $i\partial_z\bar{\partial}_z \varphi_s=(\lambda_{s}\omega)|_{B\cap L}$. It follows from $\lambda_s|_{B\cap L}$ converges to $0$ as $s\to +\infty$, $0\le \lambda_{s}\le \lambda_0$ on $U$ for any $s\ge 1$, and dominant convergence theorem that we know that $\varphi_0\le\varphi_s\le 0$ and $\varphi_s\to 0$ as $s\to +\infty$ on $\Delta$. For each $s$, it follows from proposition \ref{approx of qpsh} (shrink $U$ if necessary) that there exists a sequence of function $\varphi_{s,m}\in C^{\infty}(B\cap L)$ decreasing convergent to $\varphi_s$ on $B\cap L$ as $m\to +\infty$ and $i\partial_z\bar{\partial}_z \varphi_{s,m} \ge (\lambda_{s}\omega)|_{B\cap L}$ for any $m$.

  Denote $\tilde{h}_{s,m}:=h_s|_{B\cap L}e^{-\varphi_{s,m}}$ for any $s\ge 1$ and $m\ge 1$. Then we know that $\tilde{h}_{s,m}$ is a $C^2$ smooth Nakano semi-positive hermitian metric on $E|_{B\cap L}$. Denote $\tilde{h}^*_{s,m}:=h^*_s|_{B\cap L}e^{\varphi_{s,m}}$ for any $s\ge 1$ and $m\ge 1$, where $h^*_{s}$ is the dual metric of $h_{s}$ on $E^*|_{B\cap L}$.

  Note that for fixed $s$, $\tilde{h}^*_{s,m}\le \tilde{h}^*_{s,1}$ for any $m\ge 1$. We also note that  $\tilde{h}^*_se^{\varphi_s}\le h^*_1$ for any $s\ge 1$. As Nakano semi-positivity implies Griffiths semi-positivity in the smooth case, it follows from Proposition \ref{equivalent of Gri in smooth case} that we know $\tilde{h}^*_{s,m}$ is  Griffiths semi-negative and $|u|^2_{\tilde{h}^*_{s,m}}$ is subharmonic on $B\cap L$  for any $s\ge 1$ and $m\ge 1$. Using dominant convergence theorem twice, we have
 \begin{equation}\label{fm:mean-value inequality 1}
\begin{split}
  |u(0)|^2_{h^*}&=\lim_{s\to +\infty} |u(0)|^2_{\tilde{h}^*_s}\\
  &=\lim_{s\to +\infty} \lim_{m\to +\infty} |u(0)|^2_{\tilde{h}^*_{s,m}}\\
  &\le \lim_{s\to +\infty} \lim_{m\to +\infty} \frac{1}{\pi}\int_{\Delta} |u(\tau z_1)|^2_{\tilde{h}^*_{s,m}}d\lambda_{\tau}\\
  &=\lim_{s\to +\infty} \frac{1}{\pi}\int_{\Delta} |u(\tau z_1)|^2_{h^*_s e^{\varphi_s}}d\lambda_{\tau }\\
  &=\frac{1}{\pi}\int_{\Delta} |u(\tau z_1)|^2_{h^*}d\lambda_{\tau }.
\end{split}
\end{equation}
Now we have proved that for any complex line $L$ in $\mathbb{C}^n$, $|u|^2_{h^*}|_{B\cap L}$ is subharmonic under the assumption $\lambda_s|_{B\cap L}$ converges to $0$ a.e. as $s\to +\infty$.

If $\lambda_s|_{B\cap L}$ does not converge to $0$ a.e. as $s\to +\infty$. Assume that $L=\tau z$, where $\tau\in \mathbb{C}$ is a complex number and $z\in \mathbb{C}^n$ is a unit vector. Since $\lambda_{s}\to 0$ a.e. on $U$, we can find a sequence of complex lines $L_i=\tau z_i$ such that $z_i$ converges to $z$ as $i\to+\infty$ such that $\lambda_s|_{B\cap L_i}\to 0$ as $s\to +\infty$ on each $L_i$. It follows from $\lambda_s|_{B\cap L_i}\to 0$ as $s\to +\infty$ on each $L_i$ that we have $|u|^2_{h^*}|_{B\cap L_i}$ is subharmonic. Hence we know that for each $i$,
 \begin{equation}\label{fm:mean-value inequality 2}
\begin{split}
  |u(0)|^2_{h^*}&\le \frac{1}{\pi}\int_{\Delta} |u(\tau z_i)|^2_{h^*}d\lambda_{\tau}
\end{split}
\end{equation}
Note that $|u(\tau z_i)|^2_{h^*}\le |u(\tau z_i)|^2_{h^*_1}\le \sup_{z\in U} |u(z)|^2_{h^*_1}<+\infty$. It follows from Fatou's lemma that we have
 \begin{equation}\label{fm:mean-value inequality 3}
\begin{split}
  |u(0)|^2_{h^*}&\le \limsup_{i\to +\infty}\frac{1}{\pi}\int_{\Delta} |u(\tau z_i)|^2_{h^*}d\lambda_{\tau}\\
  &\le  \frac{1}{\pi}\int_{\Delta} \limsup_{i\to +\infty}|u(\tau z_i)|^2_{h^*}d\lambda_{\tau}\\
  &\le \frac{1}{\pi}\int_{\Delta} |u(\tau z)|^2_{h^*}d\lambda_{\tau},
\end{split}
\end{equation}
which implies that $|u|^2_{h^*}|_{B\cap L}$ is also subharmonic.

Proposition \ref{singular nak implies singular gri} is proved.
\end{proof}

Let $h$ be any hermitian metric on a holomorphic
vector bundle $E$, then $h$ induces a hermitian metric $\rm{det}\,h$ on $\rm{det}\,E$. We recall the following proposition of singular Griffiths negative metric.
\begin{Proposition} [see \cite{Raufi1}]\label{p:det}  Let $h$ be  a measurable metric on $E$ satisfying that $h$ is everywhere positive definite hermitian form on $E$, and assume that $h$ is singular Griffiths semi-negative as in Definition \ref{singular gri}. Then $\log \rm{det}\,h$  is a plurisubharmonic function.
If $M$ is a polydisc and $E=M\times \mathbb{C}^r$, then there exists a sequence of smooth hermitian metrics $\{h_v\}_{v=1}^{+\infty}$ with negative Griffiths curvature, decreasingly convergent to $h$ point-wisely on any smaller polydisc.
 \end{Proposition}
We recall the following definition which can be referred to \cite{CA}.

\begin{Definition}[see \cite{CA}]Let $h$ be a measurable metric on $E$. Define an analytic sheaf $\mathcal{E}(h)$ by setting:
\center{$\mathcal{E}(h)_x:=\{e_x\in \mathcal{O}(E)_x: |e_x|^2_h$ is integrable in some neighborhood of $x\}$.}
\end{Definition}

\begin{Theorem}
	[see \cite{GMY-boundary5}]
	\label{thm:soc1}
	Let $M=\mathbb{B}^n\subset\mathbb{C}^n$, and let $E=M\times \mathbb{C}^r$ be the trivial vector bundle on $M$. Let $(M,E,\Sigma,M_j,h,h_{j,s})$ be a s.h.m in the sense of Definition \ref{singular metric}, and assume that $\Theta_{h}(E)\ge_{Nak}^s0$. Let $\psi$ be a plurisubharmonic function on $M$. Then
	$$\mathcal{E}(he^{-a\psi})_o=\cup_{s>a}\mathcal{E}(he^{-s\psi})_o$$
	 holds for any $a\ge0$.
\end{Theorem}

Using Proposition \ref{singular nak implies singular gri} and Proposition \ref{p:det}, Theorem \ref{thm:soc1} implies the following result.
\begin{Proposition}
	\label{p:soc2}Let $M=\mathbb{B}^n\subset\mathbb{C}^n$, and let $E=M\times \mathbb{C}^r$ be the trivial vector bundle on $M$. Let $(M,E,\Sigma,M_j,h,h_{j,s})$ be a s.h.m in the sense of Definition \ref{singular metric}, and assume that $\Theta_{h}(E)\ge_{Nak}^s0$. Then
	$$\mathcal{E}(h(\rm{det}\,h)^a)_o=\cup_{s>a}\mathcal{E}(h(\rm{det}\,h)^s)_o$$
	 holds for any $a\ge0$.
\end{Proposition}

\begin{Lemma}
	\label{l:2}Let $M$ be a domain in $\mathbb{C}^n$, and let $E=M\times \mathbb{C}^r$ be the trivial vector bundle on $M$. Let $h$ be  a measurable metric on $E$ satisfying that $h$ is everywhere positive definite hermitian form on $E$, and assume that $h$ is singular Griffiths semi-positive as in Definition \ref{singular gri}.  For any $v\in \mathbb{C}^r$, there exist two plurisubharmonic functions $\varphi_1$ and $\varphi_2$ on $M$ such that $|v|_h^2=e^{\varphi_1-\varphi_2}$.
	
	Moreover if there exists $s>0$ such that $sI_r\le h(z)\le s^{-1}I_r$ for any $z\in M$,  there exist two bounded plurisubharmonic functions $\varphi_1$ and $\varphi_2$ on $M$ such that $|v|_h^2=e^{\varphi_1-\varphi_2}$.
\end{Lemma}
\begin{proof}Let $\{e_i\}_{1\le i\le r}$ be a basis for $E$ on $M$.
	Without loss of generality, we assume that $v=e_1$. Note that $h^*(=\bar{h}^{-1})$ is the dual metric of $h$ on $E^*$ and $h^*$ is singular Griffiths semi-negative.  Let $F=<e_2^*,\ldots,e_r^*>$ be a vector subbundle of $E^*$ on $M$. Following from the Definition \ref{singular gri}, we know that the induced metric $h^*|_F$ of $F$ is  singular Griffiths semi-negative on $M$.
	
	As $h=\overline{h^*}^{-1}$ and $v=e_1$, we have
	\begin{equation}
		\label{eq:0626c}|v|_h^2=\overline{|v|_h^2}=\frac{\rm{det}\, h^*|_{F}}{\rm{det}\,h^*}.
	\end{equation}
	It follows from Definition \ref{singular gri} and Proposition \ref{p:det} that there exist two plurisubharmonic fucntions $\varphi_1$ and $\varphi_2$ such that $\rm{det}\, h^*|_{F}=e^{\varphi_1}$ and $\rm{det}\,h^*=e^{\varphi_2}.$
	 Thus, equality \eqref{eq:0626c} becomes
	$$|v|_h^2=e^{\varphi_1-\varphi_2}.$$
	
	$sI_r\le h(z)\le s^{-1}I_r$ implies that $sI_r\le h^*(z)\le s^{-1}I_r$ for any $z\in M$, which shows that $\varphi_1=\log(\rm{det}\, h^*|_{F})$ and $\varphi_2=\log(\rm{det}\,h^*)$ are bounded on $M$.
\end{proof}

\begin{Lemma}
	\label{l:1}Let $A$ be an $n\times n$ positive definite hermitian matrix, and let $\beta\in(0,1)$. Assume that all eigenvalues of $A$ are greater than $1$. Then we have $$A-(A^{-1}+sI_n)^{-1}\le (s\rm{det}\,A)^{\beta}A$$
	for any $s>0$.
\end{Lemma}
\begin{proof}
	As $A$ is an $n\times n$ positive definite hermitian matrix, there exists a unitary matrix $U$ such that $UAU^{-1}$ and $U(A^{-1}+sI_n)^{-1}U^{-1}$ are diagonal matrices. Note that all eigenvalues of $A$ are greater than $1$, then we have all eigenvalues of $A$ are smaller than $\rm{det}\,A$. Thus, it suffices to prove the case $n=1$.

	Now we prove Lemma \ref{l:1} for the case $n=1$. Note that
	\begin{equation}
		\label{eq:0626a}A-(A^{-1}+sI_n)^{-1}=\frac{sA}{A^{-1}+s}=\frac{A^2}{s^{-1}+A}
	\end{equation}
	 By H\"older inequality, it follows that
	 \begin{equation}
	 	\label{eq:0626b}s^{-\beta}A^{1-\beta}\le \beta s^{-1}+(1-\beta)A\le s^{-1}+A.
	 \end{equation}
Combining equality \eqref{eq:0626a} and inequality \eqref{eq:0626b}, we obtain that
$$A-(A^{-1}+sI_n)^{-1}=\frac{A^2}{s^{-1}+A}\le s^{\beta}A^{1+\beta}.$$
	
	Thus, Lemma \ref{l:1} holds.	
	\end{proof}

\begin{Lemma}
	[\cite{ZhouZhu-jdg}]\label{l:3}
	Let $\delta$ be any positive real number. Let $\varphi$ be a negative plurisubharmonic function on $\mathbb{B}_r^n:=\{z\in\mathbb{C}^n:|z|<r\}$ such that $\varphi(0)>-\infty$. Put
	$$S_{\delta,t}=\{z\in\mathbb{B}_r^n:\varphi(e^{-t}z)<(1+\delta)\varphi(0)\},$$
	where $t>1$.
	Then $$\lim_{t\rightarrow+\infty}\mu(S_{\delta,t})=0,$$
	where $\mu$ is the Lebesgue measure on $\mathbb{C}^n$.
\end{Lemma}

Let $M=\Delta \subset\mathbb{C}$, and let $E=M\times \mathbb{C}^r$ be the trivial vector bundle on $M$. Assume that $h$ is singular Griffths semi-positive on $E$, and there exists $s>0$ such that $sI_r\le h(z)\le s^{-1}I_r$ for any $z\in M$.
\begin{Lemma}
\label{l:4}
Let $\alpha$, $c_1$ and $c_2$ be positive real numbers. Denote that $I_t=\{w\in \mathbb{C}:e^{-\alpha t}c_1\le |w|\le e^{-\alpha t}c_2\}$. Let $v_t$ be a measurable section of $E$ on $M$ for any $t>0$, and let $v_0\in\mathbb{C}^r$.
	Assume that
	$$\lim_{t\rightarrow+\infty}\sup_{w\in I_t}{|v_t(w)-v_0|_{I_r}^2}=0.$$
	Then
	$$\limsup_{t\rightarrow+\infty}\int_{I_t}\frac{|v_t(w)|_{h}^2}{|w|^2}d\lambda(w) \le 2\pi|v_0|^2_{h(0)}\log\frac{c_2}{c_1},$$
	where $d\lambda(w)$ is the Lebesgue measure on $\mathbb{C}$.
\end{Lemma}
\begin{proof}
	For any $c>0$, we have
	\begin{equation}
		\label{eq:0626d}
		\begin{split}
					&\int_{I_t}\frac{|v_t(w)|_{h}^2}{|w|^2}d\lambda(w)\\
					\le& (1+c)\int_{I_t}\frac{|v_0|_{h}^2}{|w|^2}d\lambda(w)+\frac{1+c}{c}\int_{I_t}\frac{|v_t(w)-v_0|_{h}^2}{|w|^2}d\lambda(w).
		\end{split}
	\end{equation}
	It follows from $\lim_{t\rightarrow+\infty}\sup_{w\in I_t}{|v_t(w)-v_0|_{I_r}^2}=0$ that
\begin{equation}
	\label{eq:0626e}
	\begin{split}
		&\limsup_{t\rightarrow+\infty}\int_{I_t}\frac{|v_t(w)-v_0|_{h}^2}{|w|^2}d\lambda(w)\\
		\le&\limsup_{t\rightarrow+\infty}\left(\sup_{w\in I_t}{|v_t(w)-v_0|_{h}^2}\right)\int_{I_t}\frac{1}{|w|^2}d\lambda(w)\\
		\le&s^{-1}\limsup_{t\rightarrow+\infty}\left(\sup_{w\in I_t}{|v_t(w)-v_0|_{I_r}^2}\right)\times(2\pi\log\frac{c_2}{c_1})\\
		=&0.
	\end{split}
\end{equation}	
	
Lemma \ref{l:2}	tell us that there exist two bounded plurisubharmonic functions $\varphi_1$ and $\varphi_2$ on $\Delta$ such that $|v_0|_h^2=e^{\varphi_1-\varphi_2}$. As $\varphi_1$ is upper semi-continuous on $\Delta$, we have
\begin{equation}
	\label{eq:0626f}\limsup_{t\rightarrow+\infty}\int_{I_t}\frac{|v_0|_{h}^2}{|w|^2}d\lambda(w)\le e^{\varphi_1(0)}\limsup_{t\rightarrow+\infty}\int_{I_t}\frac{e^{-\varphi_2}}{|w|^2}d\lambda(w).
\end{equation}
Let $\delta>0$. Denote that $$S_{\delta,t}=\{|z|<c_2+1 :\varphi_2(e^{-\alpha t}z)<(1+\delta)\varphi_2(0)\},$$
and Lemma \ref{l:3} shows that $\mu(S_{\delta,t})=0$.
Denote that $N=\sup_{\Delta}e^{-\varphi_2}<+\infty$, then
\begin{equation}
	\label{eq:0626g}\begin{split}
		&\limsup_{t\rightarrow+\infty}\int_{I_t}\frac{e^{-\varphi_2}}{|w|^2}d\lambda(w)\\
		=&\limsup_{t\rightarrow+\infty}\int_{\{z\in\mathbb{C}:c_1\le|z|\le c_2\}}\frac{e^{-\varphi_2(e^{-\alpha t}z)}}{|z|^2}d\lambda(z)\\
		\le&\limsup_{t\rightarrow+\infty}\mu(S_{\delta,t})\frac{N}{c_1^2}+2\pi e^{-(1+\delta)\varphi_2(0)}\log\frac{c_2}{c_1}\\
		=&2\pi e^{-(1+\delta)\varphi_2(0)}\log\frac{c_2}{c_1}.
	\end{split}
\end{equation}
 Taking $\delta\rightarrow0$, inequality \eqref{eq:0626g} becomes that
 $$\limsup_{t\rightarrow+\infty}\int_{I_t}\frac{e^{-\varphi_2}}{|w|^2}d\lambda(w)\le 2\pi e^{-\varphi_2(0)}\log\frac{c_2}{c_1}.$$
Combining ineqaulity \eqref{eq:0626d}, inequality \eqref{eq:0626e} and inequality \eqref{eq:0626f}, we get
\begin{displaymath}
	\begin{split}
		&\limsup_{t\rightarrow+\infty}\int_{I_t}\frac{|v_t(w)|_{h}^2}{|w|^2}d\lambda(w)\\
					\le& (1+c)\limsup_{t\rightarrow+\infty}\int_{I_t}\frac{|v_0|_{h}^2}{|w|^2}d\lambda(w)+\frac{1+c}{c}\limsup_{t\rightarrow+\infty}\int_{I_t}\frac{|v_t(w)-v_0|_{h}^2}{|w|^2}d\lambda(w)\\
					\le&(1+c) 2\pi e^{\varphi_1(0)-\varphi_2(0)}\log\frac{c_2}{c_1}.
	\end{split}
\end{displaymath}
 	Note that $e^{\varphi_1(0)-\varphi_2(0)}=|v_0|^2_{h(0)}$, taking $c\rightarrow0$, we obtain $\limsup_{t\rightarrow+\infty}\int_{I_t}\frac{|v_t(w)|_{h}^2}{|w|^2}d\lambda(w) \le 2\pi|v_0|^2_{h(0)}\log\frac{c_2}{c_1}$.
\end{proof}

We recall the following desingularization theorem due to Hironaka.
\begin{Theorem}[\cite{Hironaka}, see also \cite{BM-1991}]\label{thm:desing}
	Let $X$ be a complex manifold, and $M$ be an analytic sub-variety in $X$.  Then there is a local finite sequence of blow-ups $\mu_j:X_{j+1}\rightarrow X_j$ $(X_1:=X,j=1,2,...)$ with smooth centers $S_j$ such that:
	
	$(1)$ Each component of $S_j$ lies either in $(M_j)_{\rm{sing}}$ or in $M_j\cap E_j$, where $M_1:=M$, $M_{j+1}$ denotes the strict transform of $M_j$ by $\mu_j$, $(M_j)_{\rm{sing}}$ denotes the singular set of $M_j$, and $E_{j+1}$ denotes the exceptional divisor $\mu^{-1}_j(S_j\cup E_j)$;
	
	$(2)$ Let $M'$ and $E'$ denote the final strict transform of $M$ and the exceptional divisor respectively. Then:
	
	\quad$(a)$ The underlying point-set $|M'|$ is smooth;
	
	\quad$(b)$ $|M'|$and $E'$ simultaneously have only normal crossings.
\end{Theorem}

The $(b)$ in the above theorem means that, locally, there is a coordinate system in which $E'$ is a union of coordinate hyperplanes and $|M'|$ is a coordinate subspace.

We present the following proposition which deals with a limiting problem related to singular metrics on vector bundles.

\begin{Proposition}
	\label{p:inte} Let $\psi$ be a quasi-plurisubharmonic function on an $n$-dimensional complex manifold $M$ with neat analytic singularities. Assume that the singularities of $\psi$ are log canonical along the zero variety $Y=V(I(\psi))$.  Denote $Y^0=Y_{\rm{reg}}$ the set of regular points of $Y$. Let $E$ be a holomorphic vector bundle over $M$ with rank $r$. Let $h$ be  a measurable metric on $E$ satisfying that $h$ is everywhere positive definite hermitian form on $E$, and assume that $h$ is singular Griffiths semi-positive.
	
	Let $V\Subset \Omega$ be two local coordinate balls in $M$ such that $E|_{\Omega}=\Omega\times \mathbb{C}^r$.  Let $v$ be a nonnegative continuous function on $\Omega$ with $\rm{supp}$$v\in V$. Let $C$ and  $\beta$ be positive numbers, and let $\beta_1$ be a small enough positive number (depending only on $M$, $Y$ and $\psi$). Let $d\lambda$ be the Lebesgue measure on $\Omega$. Assume that $f=(f_1,\ldots,f_r)\in \mathcal{O}(\Omega\cap Y)^{r}$ satisfying that
	\begin{equation}
		\label{eq:0627a}
		\int_{\Omega\cap Y^0}|f|^2_{h}d\lambda[\psi]<+\infty.
	\end{equation}
For any $t>0$,	let $f_t=(f_{t,1},\ldots,f_{t,r})\in\mathcal{O}(\Omega)^r$ and $f_t|_{\Omega\cap Y}=f$, which satisfies
\begin{equation}
	\label{eq:0627b}
	\sup_{\Omega}\sum_{1\le j\le r}|f_{t,j}|^2\le Ce^{\beta_1t}
\end{equation}
and
\begin{equation}
	\label{eq:0627c}
	\int_{\Omega\cap\{\psi<-t\}}|f_t|^2_{h(\rm{det}\,h)^{\beta}}d\lambda\le Ce^{-t}.
\end{equation}
	Then we have
	\begin{equation}
		\label{eq:0627d}
		\limsup_{t\rightarrow+\infty}\int_{V\cap\{-t-1<\psi<-t\}}v|f_t|^2_he^{-\psi}d\lambda\le\int_{V\cap Y^0}v|f|^2_hd\lambda[\psi].
	\end{equation}
\end{Proposition}

\begin{proof} We prove Proposition \ref{p:inte} in three steps.

\

\emph{Step 1. Using Hironaka's desingularization theorem to deal with the measure $d\lambda[\psi]$.}

\

The idea of using Hironaka's desingularization theorem to deal with the measure $d\lambda[\psi]$ comes from \cite{D2016} (see also \cite{ZhouZhu-jdg}).
Firstly, we use Theorem \ref{thm:desing} on $M$ to resolve the singularities of $Y$, and denote the corresponding proper modification by $\mu_1$. Let $Y'$ denote the strict transform of $Y$. Secondly, we make a blow-up along $|Y'|$ denoted by $\mu_2$, where $|Y'|$ is the underlying point-set of $Y'$. Let $\sum$ denote the strict transform of $\{\psi=-\infty\}$ by $\mu_1\circ\mu_2$. Thirdly, we use  Theorem \ref{thm:desing}  to resolve the singularities of $\sum$, and denote the corresponding proper modification by $\mu_3$. Let $\sum'$ denote the strict transform of $\sum$ by $\mu_3$. Finally, we make a blow-up along $|\sum'|$ denoted by $\mu_4$. Thus, we get a proper holomorphic map $\mu(:=\mu_1\circ\mu_2\circ\mu_3\circ\mu_4):\widetilde M\rightarrow M$, which is locally a finite composition of blow-ups with smooth centers. Let $\widetilde Y$ be the strict transform of $\mu_2^{-1}(|Y'|)$ by $\mu_3\circ\mu_4$, and we have $\widetilde Y$ and the divisor $\overline{\mu^{-1}(\{\psi=-\infty\})\backslash\widetilde Y}$ simultaneously have only normal crossings on $\widetilde M$.

For any $\tilde z\in\overline{\mu^{-1}(V)}\cap\mu^{-1}(\{\psi=-\infty\})$, let $(W;w_1,\ldots,w_n)$ be a coordinate ball centered at $\tilde z$ satisfying that $W\Subset\mu^{-1}(V)$, $w^b=0$ is the zero divisor of the Jacobian $J_{\mu}$ (of $\mu$) and
$$\psi\circ\mu(w)=c\log|w^a|^2+\tilde u(w)$$
on $W$,
where $\tilde u\in\mathcal{C}^{\infty}(\overline W)$, $w^a:=\prod_{p=1}^nw_p^{a_p}$ and $w^b:=\prod_{p=1}^nw_p^{b_p}$. Then the multiplier ideal sheaf $\mathcal{I}(\psi)$ can be given as (see \cite{D2016,DemaillyAG})
$$\mathcal{I}(\psi)=\mu_*\mathcal{O}_{\widetilde M}\left(-\sum_{p=1}^n\lfloor ca_p-b_p\rfloor_+D_p\right),$$
where $D_p:=\{w_p=0\}$ and $\lfloor ca_p-b_p\rfloor_+:=\sup\{m\in\mathbb{Z}_{\ge0}:m\le ca_p-b_p\}$. Denote that $\xi:=\frac{|J_{\mu}|^2}{|w^b|^2}$ and  $\kappa:=\{p:ca_p-b_p=1\}$.

By Definition \ref{def:oshawa measure} and Remark \ref{r:measure}, the measure $d\lambda[\psi]$ can be defined as
\begin{equation}
	\label{eq:220627a}
	g\mapsto \limsup_{t\rightarrow+\infty}\int_{\{-t-1<\psi\circ\mu <-t\}}\frac{(\hat g\circ\mu)\xi e^{-\tilde u}}{|w^{ca-b}|^2}d\lambda(w),
\end{equation}
where $g\in\mathcal{C}_c(Y^0)$  and $\hat g\in\mathcal{C}_c(M)$ satisfy $\hat g|_{Y^0}=g$ and $(supp\,\hat{g})\cap Y=Y^0$ (One would take into account a partition of unity on various coordinate charts covering the fibers of $\mu$ (see \cite{D2016}), but we avoid this technicality for the simplicity of notation).
By the construction of $\mu$, we get that one of the following cases holds:

\noindent$(A)$ There exists $p_0$ such that $\widetilde Y\cap W=D_{p_0}$ (choosing $W$ small enough);

\noindent$(B)$ $\widetilde Y\cap W=\emptyset$.

\noindent As the singularities of $\psi$ are log canonical along $Y=V(\mathcal{I}(\psi))$, we have $ca_{p_0}-b_{p_0}=1$ and $ca_p-b_p\le1$ for $p\not=p_0$ in Case $(A)$, and $ca_p-b_p\le1$ for any $p$ in Case $(B)$.

 In Case $(A)$, note that
 $$\int_{\{-t-1<\psi\circ\mu <-t\}}\frac{(\hat g\circ\mu)\xi e^{-\tilde u}}{|w^{ca-b}|^2}d\lambda(w)=\int_{\{-t-1<\psi\circ\mu <-t\}}\frac{(\hat g\circ\mu)\xi e^{-\tilde u}}{|(w')^{ca'-b'}|^2|w_{p_0}|^2}d\lambda(w)$$
and
 $$\{-t-1<\psi\circ\mu<-t\}=\left\{e^{-t-1-\tilde u(w)}|(w')^{a'}|^{-2c}<|w_{p_0}|^{2ca_{p_0}}<e^{-t-\tilde u(w)}|(w')^{a'}|^{-2c}\right\},$$
 where $w=(w',w_{p_0})\in\mathbb{C}^{n-1}\times\mathbb{C}$, $a=(a',a_{p_0})$ and $b=(b',b_{p_0})$. Thus, the mapping \eqref{eq:220627a} becomes
 \begin{equation}
 	\label{eq:220627b}
 	g\mapsto \frac{\pi}{ca_{p_0}} \int_{w'\in D_{p_0}}\frac{(g\circ\mu)\xi e^{-\tilde u}}{|(w')^{ca'-b'}|^2}d\lambda(w'),
 \end{equation}
  where $d\lambda(w)=d\lambda(w')d\lambda(w_{p_0})$.  If $p_1\in\kappa\backslash\{p_0\}$,  by the construction of $\mu$, we obtain that the images of $D_{p_1}$ and $D_{p_1}\cap D_{p_0}$ coincide under $\mu$. Proposition \ref{p:det} shows that $\inf _K\frac{|f|_h^2}{|f|_{\tilde h}^2}>0$ for any compact subset $K$ of $Y$ and any smooth metric $\tilde h$ on $E|_{\Omega}$. It follows from inequality \eqref{eq:0627a} and mapping \eqref{eq:220627b} that $f\circ\mu|_{D_{p_1}\cap D_{p_0}}=0$, which implies that
  \begin{equation}
  	\label{eq:220627c}f\circ\mu|_{D_{p_1}}=0
  \end{equation}
  for any $p_1\in\kappa\backslash\{p_0\}$.

In Case $(B)$, we get that
\begin{equation}
	\label{eq:0628e}f\circ\mu|_{D_p}=0
\end{equation}
 holds for any $p\in\kappa$ by similar discussion.

\

\emph{Step 2. approximations of $h$ and uniform estimates for $f_t\circ\mu$.}

\

 Denote that $h_j=(h^*+e^{-j}I_r)^*$ is a measurable metric on $E|_{\Omega}$, where $j\in\mathbb{Z}_{\ge1}$. Following from Definition \ref{singular gri}, we know that $h_j$ is singular Griffiths semi-positive on $E|_{\Omega}$. Note that $h_j(z)$ is increasingly convergent to $h(z)$ with respect to $j$ for any $z\in \Omega$. Without loss of generality, we assume that all eigenvalues of $h_j(z)$ are greater than $1$ for any $z\in \Omega$.

Following from inequality \eqref{eq:0627c} and Lemma \ref{l:1}, we have
\begin{equation}\nonumber
\begin{split}
		&\int_{\{-t-1<\psi<-t\}\cap V}v|f_t|^2_he^{-\psi}d\lambda\\
		=&\int_{\{-t-1<\psi<-t\}\cap V}v|f_t|^2_{h_j}e^{-\psi}d\lambda+\int_{\{-t-1<\psi<-t\}\cap V}v|f_t|^2_{h-h_j}e^{-\psi}d\lambda\\
		\leq&\int_{\{-t-1<\psi<-t\}\cap V}v|f_t|^2_{h_j}e^{-\psi}d\lambda+e^{-\beta j}\int_{\{-t-1<\psi<-t\}\cap V}v|f_t|^2_{h(\rm{det}\,h)^{\beta}}e^{-\psi}d\lambda\\
		\le&\int_{\{-t-1<\psi<-t\}\cap V}v|f_t|^2_{h_j}e^{-\psi}d\lambda+e^{-\beta j+t+1}\int_{\{\psi<-t\}\cap V}v|f_t|^2_{h(\rm{det}\,h)^{\beta}}d\lambda\\
		\le&\int_{\{-t-1<\psi<-t\}\cap V}v|f_t|^2_{h_j}e^{-\psi}d\lambda+Ce^{-\beta j+1}\sup_{V}v.
	\end{split}
\end{equation}
 For any $b>0$, there exists $j_b$ such that for any $j>j_b$,
\begin{equation}
	\label{eq:220627f}\begin{split}
	\int_{\{-t-1<\psi<-t\}\cap V}v|f_t|^2_he^{-\psi}d\lambda
		\le\int_{\{-t-1<\psi<-t\}\cap V}v|f_t|^2_{h_j}e^{-\psi}d\lambda+b.
	\end{split}
\end{equation}

In Case $(A)$, it follows from inequality \eqref{eq:0627b} and equaltity \eqref{eq:220627c} that
\begin{equation}
	\label{eq:220627d}\begin{split}
		|f_t\circ\mu(w',w_{p_0})- f_t\circ\mu(w',0)|^2&=\sum_{1\le k\le r}|f_{t,k}\circ\mu(w',w_{p_0})- f_{t,k}\circ\mu(w',0)|^2\\
		&\le C_1\prod_{p\in\kappa}|w_p|^2\sup_{|\gamma|\le|\kappa|}\sup_{U}|\partial^{\gamma}f_{t,k}|^2\\
		&\le C_2\prod_{p\in\kappa}|w_p|^2\sup_{\Omega}|f_t|^2\\
		&\le CC_2e^{\beta_1t}\prod_{p\in\kappa}|w_p|^2
	\end{split}
\end{equation}
and
\begin{equation}
	\label{eq:220627e}\begin{split}
		|f_t\circ\mu(w',0)|^2&=|f\circ\mu(w',0)|^2\\
		&=\sum_{1\le k\le r}|f_{k}\circ\mu(w',0)|^2\\
		&\le C_3\prod_{p\in\kappa\backslash\{p_0\}}|w_p|^2
	\end{split}
\end{equation}
for any  $w=(w',w_{p_0})\in W$, where $C_1$, $C_2$ and $C_3$ are positive constants independent of $t$.

In Case $(B)$, if $\kappa=\emptyset$, inequality \eqref{eq:0627b} implies that
\begin{equation}
	\label{eq:220627g}|f_t\circ\mu(w)|^2\le Ce^{\beta_1t}
\end{equation}
for any $w\in W$. If $\kappa\not=\emptyset$, it follows from inequality \eqref{eq:0627b} and equality \eqref{eq:0628e} that
\begin{equation}
	\label{eq:220627h}|f_t\circ\mu(w)|^2\le C_4e^{\beta_1t}\prod_{p\in\kappa}|w_p|^2
\end{equation}
for any $w\in W$.

\

\emph{Step 3. Completing the proof.}

\

Following the notations in Step 1, denote that
$$I_j:=\limsup_{t\rightarrow+\infty}\int_{W\cap\{-t-1<\psi\circ\mu<-t\}}\frac{(v\circ\mu)|f_t\circ\mu|^2_{h_j\circ\mu}\xi e^{-\tilde u}}{|w^{ca-b}|^2}d\lambda,$$
where $\xi=\frac{|J_{\mu}|^2}{|w^b|^2}$.
Following from inequality \eqref{eq:220627f} and the mapping \eqref{eq:220627b}, it suffices to prove that
\begin{equation}
	\nonumber
	I_j\le\frac{\pi}{ca_{p_0}} \int_{W\cap D_{p_0}}\frac{(v\circ\mu)|f_t\circ\mu|_{h_j\circ\mu}\xi e^{-\tilde u}}{|(w')^{ca'-b'}|^2}d\lambda(w')
\end{equation}
in Case $(A)$ and $I_j=0$ in Case $(B)$ for large enough $j$.

In Case $(A)$, for any $t$ and $w'\in (D_{p_0}\cap W)\backslash\cup_{p\not=p_0}D_p$, denote that
$$\Phi_{t,j}(w'):=\int_{W_{t,w'}}\frac{(v\circ\mu)|f_t\circ\mu|^2_{h_j\circ\mu}\xi e^{-\tilde u}}{|(w')^{ca'-b'}|^2}\frac{1}{|w_{p_0}|^2} d\lambda(w_{p_0})$$
and
$$\Phi_j(w'):=\frac{\pi}{ca_{p_0}} \frac{\left((v\circ\mu)|f_t\circ\mu|_{h_j\circ\mu}\xi e^{-\tilde u}\right)(w',0)}{|(w')^{ca'-b'}|^2},$$
where
\begin{displaymath}
	\begin{split}
		W_{t,w'}&=\{w_{p_0}\in\mathbb{C}:(w',w_{p_0})\in W\,\&\,-t-1<\psi\circ\mu(w',w_{p_0})<-t \}\\
		&=\left\{w_{p_0}\in\mathbb{C}:(w',w_{p_0})\in W\,\&\, \frac{e^{-t-1-\tilde u(w',w_{p_0})}}{|(w')^{a'}|^{2c}}<|w_{p_0}|^{2ca_{p_0}}<\frac{e^{-t-\tilde u(w',w_{p_0})}}{|(w')^{a'}|^{2c}}\right\}.
	\end{split}
\end{displaymath}
It follows from inequality \eqref{eq:220627d} that
\begin{displaymath}
	\begin{split}
		\sup_{w_{p_0}\in W_{t,w'}}|f_t\circ\mu(w',w_{p_0})-f_t\circ\mu(w',0)|^2&\le CC_2e^{\beta_1t}\prod_{p\in\kappa}|w_p|^2\\
		&\le C_5e^{\left(\beta_1-\frac{1}{ca_{p_0}}\right)t},
	\end{split}
\end{displaymath}
	where $w'\in (D_{p_0}\cap W)\backslash\cup_{p\not=p_0}D_p$ and $C_5$ is a positive constant independent of $t$. When $\beta_1<\frac{1}{ca_{p_0}}$, by Lemma \ref{l:4}, we obtain that
	\begin{equation}
		\label{eq:0628c}\limsup_{t\rightarrow+\infty}\Phi_{t,j}(w')\le \Phi_j(w')
	\end{equation}
	holds for any $w'\in (D_{p_0}\cap W)\backslash\cup_{p\not=p_0}D_p.$
	It follows from inequality \eqref{eq:220627d} and inequality \eqref{eq:220627e} that
	\begin{equation}
		\label{eq:0628d}\begin{split}
			\Phi_{t,j}(w')\le& C_6\int_{W_{t,w'}}\frac{|f_t\circ\mu|^2}{|(w')^{ca'-b'}|^2}\frac{1}{|w_{p_0}|^2} d\lambda(w_{p_0})\\
			\le& 2C_6\int_{W_{t,w'}}\frac{|f_t\circ\mu(w',w_{p_0})-f_t\circ\mu(w',0)|^2}{|w^{ca-b}|^2}d\lambda(w_{p_0})\\
			&+2C_6\int_{W_{t,w'}}\frac{|f_t\circ\mu(w',0)|^2}{|w^{ca-b}|^2} d\lambda(w_{p_0})\\
			\le& C_7\int_{W_{t,w'}}\frac{e^{\beta_1t}\prod_{p\in\kappa}|w_p|^2}{|w^{ca-b}|^2} d\lambda(w_{p_0})+C_7\int_{W_{t,w'}}\frac{\prod_{p\in\kappa\backslash\{p_0\}}|w_p|^2}{|w^{ca-b}|^2} d\lambda(w_{p_0})\\
			\le& C_8\int_{W_{t,w'}}\frac{\prod_{p\in\kappa}|w_p|^2}{|w^{(1+\beta_1)ca-b}|^2} d\lambda(w_{p_0})+C_7\int_{W_{t,w'}}\frac{\prod_{p\in\kappa\backslash\{p_0\}}|w_p|^2}{|w^{ca-b}|^2} d\lambda(w_{p_0}),
		\end{split}
	\end{equation}
	where $C_6$, $C_7$ and $C_8$ are positive constants independent of $t$.
	Thus, $\Phi_{t,j}(w')$ is dominated by a function of $w'$ which is independent of $t$ and belongs to $L^1(W\cap D_{p_0})$ when
	$$\beta_1<\min_{\{p:a_p\not=0\}}\frac{1-(ca_p-b_p)+\lfloor ca_p-b_p\rfloor_+}{ca_p}.$$
	Note that
\begin{equation}
\nonumber
	I_j=\limsup_{t\rightarrow+\infty}\int_{W\cap D_{p_0}}\Phi_{t,j}(w')d\lambda(w').
\end{equation}
It follows from \eqref{eq:0628c} and Fatou's Lemma that
\begin{displaymath}
	\begin{split}
		I_j&\le \int_{W\cap D_{p_0}}\limsup_{t\rightarrow+\infty}\Phi_{t,j}(w')d\lambda(w')\\
		&\leq \int_{W\cap D_{p_0}}\Phi_{j}(w')d\lambda(w')\\
		&=\frac{\pi}{ca_{p_0}} \int_{W\cap D_{p_0}}\frac{(v\circ\mu)|f_t\circ\mu|_{h_j\circ\mu}\xi e^{-\tilde u}}{|(w')^{ca'-b'}|^2}d\lambda(w').
	\end{split}
\end{displaymath}

In Case $(B)$, if $\kappa=\emptyset$, it follows from inequality \eqref{eq:220627g} that
\begin{displaymath}
	\begin{split}
		I_j=&\limsup_{t\rightarrow+\infty}\int_{W\cap\{-t-1<\psi\circ\mu<-t\}}\frac{(v\circ\mu)|f_t\circ\mu|^2_{h_j\circ\mu}\xi e^{-\tilde u}}{|w^{ca-b}|^2}d\lambda\\
		\le& C_9\limsup_{t\rightarrow+\infty}\int_{W\cap\{-t-1<\psi\circ\mu<-t\}}\frac{1}{|w^{(1+\beta_1)ca-b}|^2}d\lambda\\
		=&0
	\end{split}
\end{displaymath}
when $\beta_1<\min_{\{p:a_p\not=0\}}\frac{1-(ca_p-b_p)+\lfloor ca_p-b_p\rfloor_+}{ca_p},$
where $C_9$ is a positive constent independent of $t$. If $\kappa\not=\emptyset$, it follows from
inequality \eqref{eq:220627h} that
\begin{displaymath}
	\begin{split}
		I_j=&\limsup_{t\rightarrow+\infty}\int_{W\cap\{-t-1<\psi\circ\mu<-t\}}\frac{(v\circ\mu)|f_t\circ\mu|^2_{h_j\circ\mu}\xi e^{-\tilde u}}{|w^{ca-b}|^2}d\lambda\\
		\le& C_{10}\limsup_{t\rightarrow+\infty}\int_{W\cap\{-t-1<\psi\circ\mu<-t\}}\frac{\prod_{p\in\kappa}|w_p|^2}{|w^{(1+\beta_1)ca-b}|^2}d\lambda\\
		=&0
	\end{split}
\end{displaymath}
when $\beta_1<\min_{\{p:a_p\not=0\}}\frac{1-(ca_p-b_p)+\lfloor ca_p-b_p\rfloor_+}{ca_p},$
where $C_{10}$ is a positive constant independent of $t$.

Thus, Proposition \ref{p:inte} holds.
\end{proof}

\subsection{Other preparations for the proof of main theorem}
\
In this section, we make some preparations for the proof of main theorem.

We would like to recall the following results of blow-up of K\"ahler manifolds.

Let $(X,\omega)$ be a K\"ahler manifold and $M\subset\subset X$ be a relatively compact open subset of $X$. Let $Y$ be a smooth complex submanifold of $X$ of codimension $l$. Let $\sigma:\tilde{X}\to X$ be the blow up of $X$ along $Y$. Denote $D:=\sigma^{-1}(Y)$  and $\tilde{M}:=\sigma^{-1}(M)$. Let $|D|$ be the underlying point-set of $D$.

We know that $\sigma^{-1}(Y)$ is isomorphic to the projective bundle $\mathbb{P}(N_{Y\backslash X})\xrightarrow{\sigma} Y$. Let $\mathcal{O}_{\tilde{X}}(D)$ be the line bundle associated to $D$ on $\tilde{X}$. Let $s$ be the canonical section of $\mathcal{O}_{\tilde{X}}(D)$, i.e. $D=\{s=0\}$. Denote $[D]$ be the integration current associated to $D$.

\begin{Lemma}\label{kahler metric}
There exist a metric $h_D$ on $\mathcal{O}_{\tilde{X}}(D)$ and a positive number $\tilde{a}$ big enough such that $$\tilde{\omega}:=\tilde{a}\mu^*\omega+\sqrt{-1}\partial\bar{\partial}(\log|s|^2_{h_D})-2\pi[D]$$ is a K\"ahler metric on an open neighborhood of the closure of $\tilde{M}$.
\end{Lemma}
\begin{proof} We  recall the construction of $\mathcal{O}_{\tilde{X}}(D)$ and $s$.

Let $\tilde{X}=\cup_{\alpha} U_\alpha$ be an open cover of $\tilde{X}$. If $U_{\alpha}\cap D\neq \emptyset$, we assume that $U_{\alpha}\cap D$ is defined by equation $f_{\alpha}=0$, where $f_{\alpha}$ is a holomorphic function on $U$. If $U_{\alpha}\cap D= \emptyset$, we set $f_{\alpha}=1$. On the intersections $U_{\alpha}\cap U_{\beta}$, the function $g_{\alpha\beta}=\frac{f_{\alpha}}{f_{\beta}}$ is invertible. Note that both $\{g_{\alpha\beta}\}$ and $\{g^{-1}_{\alpha\beta}\}$ satisfy cocycle condition. Actually we know that the transition functions of $\mathcal{O}_{\tilde{X}}(D)$ are $\{g_{\alpha\beta}\}$ and hence the transition functions $\{g^{-1}_{\alpha\beta}\}$ define the holomorphic line bundle $\mathcal{O}_{\tilde{X}}(-D)$. We also have $s=\{f_{\alpha}\}_{\alpha}$ is the canonical section of $\mathcal{O}_{\tilde{X}}(D)$.
The following result can be referred to \cite{voisin} (see Lemma 3.26 in \cite{voisin}).

\centerline{\emph{The restriction of $\mathcal{O}_{\tilde{X}}(-D)$ to $D$ is isomorphic to $\mathcal{O}_{\mathbb{P}(N_{Y/ X})}(1)$},}

\noindent
where $N_{Y/ X}=T_{X}|_{Y}/T_{Y}$ is the normal bundle of $Y$ in $X$.

We note that $\mathcal{O}_{\mathbb{P}(N_{Y/ X})}(1)$ has a metric $h$ which has positive curvature along  the fiber $\sigma^{-1}(y)$ for any $y\in Y$ (see chapter 3.3.2 in \cite{voisin}). By a partition of unity argument, the metric $h$ on $\mathcal{O}_{\mathbb{P}(N_{Y/ X})}(1)$ extends to a smooth metric $h_{-D}$ on $\mathcal{O}_{\tilde{X}}(-D)$ which has positive curvature on the fiber $\sigma^{-1}(y)$ for any $y\in Y$ and is flat outside a compact neighborhood of $|D|$. Hence the curvature form $\omega_{-D}$ of $h_{-D}$ is strictly positive on the fiber $\sigma^{-1}(y)$ for any $y\in Y$ and is zero outside a compact neighborhood of $|D|$.

Denote $h_D$ be the dual metric of $h_{-D}$ on $\mathcal{O}_{\tilde{X}}(D)$. Then $\log|s|^2_{h_{D}}$ is a globally defined quasi-plurisubharmonic function on $\tilde{X}$. By Lelong-Poincar\'{e} equation, we have
\begin{equation}\label{current of logsh}
\begin{split}
 i\sqrt{-1}\partial\bar{\partial}(\log|s|^2_{h_D})&=2\pi[D]-i\sqrt{-1}\Theta(\mathcal{O}_{\tilde{X}}(D))\\
     &= 2\pi[D]+i\sqrt{-1}\Theta(\mathcal{O}_{\tilde{X}}(-D))\\
     &=2\pi[D]+\omega_{-D}.
\end{split}
\end{equation}

As $\omega$ is a K\"ahler form on $X$, then $\sigma^{*}(\omega)$ is a closed real $(1,1)$-form which is positive outside $|D|$ and is semi-positive along the fiber $\sigma^{-1}(y)$ for any $y\in Y$ (the kernel of $\sigma^{*}(\omega)$ along $\sigma^{-1}(Y)$ consists of the tangent space to
the fibres of $\sigma$).
Note that the smooth curvature form $\omega_{-D}$ of $h_{-D}$ is strictly positive on the fiber $\sigma^{-1}(y)$ for any $y\in Y$ and is zero outside a compact neighborhood of $|D|$. As $\sigma$ is proper, we know that $\tilde{M}$ is relatively compact in $\tilde{X}$. Hence we can choose $\tilde{a}$ big enough such that
$$\tilde{\omega}:=\tilde{a}\sigma^*\omega+\sqrt{-1}\partial\bar{\partial}(\log|s|^2_{h_D})-2\pi[D]$$ is  positive on an open neighborhood of the closure of $\tilde{M}$ and hence a K\"ahler metric.
\end{proof}

Let $(X_0,\omega_0)$ be a K\"ahler manifold.
Let $D\subset X_0$ be a divisor. We call $D$ is normal crossing if there is a local coordinate system in which $|D|$
is a union of coordinate hyperplanes.

Let $M\subset\subset X_0$ be a relatively compact open subset of $X_0$. Let $N$ be a positive integer.
For $i=1,\ldots,N$, let $\sigma_i:X_i\to X_{i-1}$ be a blow up of $X_{i-1}$ along $Y_{i-1}$ for any $i\ge 1$.

 Denote $\tilde{M}_i:=\sigma_i^{-1}\circ\cdots\circ\sigma_1^{-1}(M)$. By Lemma \ref{kahler metric} and induction, we know that there exists $a_i>0$ such that $\omega_{i}:=a_i\sigma_{i}^{*}(\omega_{i-1})+
 \sqrt{-1}\partial\bar{\partial}(\log|s_i|^2_{h_{D_i}})-2\pi[D_i]$ is a K\"ahler metric on $\tilde{M}_i$, where $D_i:=\sigma_i^{-1}(Y_{i-1})$, $s_i$ is the canonical section of $O_{X_i}(D_i)$ and $h_{D_i}$ is a smooth metric on $X_i$. Since $D_i=\{s_i=0\}$ and $h_{D_i}$ is smooth, we know that $\log|s_i|^2_{h_{D_i}}|_{D_i}=-\infty$.

 Denote $\sigma:=\sigma_1\circ\cdots\circ\sigma_N:X_N\to X_0$. We have the following result.
 \begin{Lemma} \label{sequence of blow up}
 There exist a positive number $A>0$, a quasi-plurisubharmonic function $\Upsilon$ on $X_N$ and divisor $H$ on $X_N$ such that
 $$\omega:=A\sigma^{*}\omega_0+
 \sqrt{-1}\partial\bar{\partial}\Upsilon-2\pi [H]$$
 is a K\"ahler metric on an open neighborhood of the closure of $\tilde{M}_{N}$ and  $|H|=\{\Upsilon=-\infty\}$.
 \end{Lemma}
\begin{proof}Note that
$$\omega_{N}:=a_N\sigma_{N}^{*}(\omega_{N-1})+
 \sqrt{-1}\partial\bar{\partial}(\log|s_N|^2_{h_{D_N}})-2\pi[D_N]$$
  is a K\"ahler metric on an open neighborhood of the closure of $\tilde{M}_N$, where $\log|s_N|^2_{h_{D_N}}$ is a quasi-plurisubharmonic function on $X_N$ which satisfies $\log|s_N|^2_{h_{D_N}}$ equals $-\infty$ on $|D_N|$. Denote $\Upsilon_N=\log|s_N|^2_{h_{D_N}}$. Denote $H_N:=D_N$. Then we have $|H_N|=\{\Upsilon_N=-\infty\}$.

  Note that
$$\omega_{N-1}:=a_{N-1}\sigma_{N-1}^{*}(\omega_{N-2})+
 \sqrt{-1}\partial\bar{\partial}(\log|s_{N-1}|^2_{h_{D_{N-1}}})-2\pi[D_{N-1}]$$
  is a K\"ahler metric on an open neighborhood of the closure of $\tilde{M}_{N-1}$, where $\log|s_{N-1}|^2_{h_{D_{N-1}}}$ is a
  quasi-plurisubharmonic function on $X_{N-1}$ which equals $-\infty$ on $|D_{N-1}|$. Since quasi-plurisubharmonicity is invariant under the pull-back of holomorphic map, we know that $\sigma_{N}^{*}(\log|s_{N-1}|^2_{h_{D_{N-1}}})$ is a quasi-plurisubharmonic function on $X_N$. Denote $\Upsilon_{N-1}=\log|s_N|^2_{h_{D_N}}+\sigma_{N}^{*}(a_{N}\log|s_{N-1}|^2_{h_{D_{N-1}}})$ and $H_{N-1}=D_N+\sigma_{N}^{*}(a_{N}D_{N-1})$ on $X_N$. We still have $|H_{N-1}|=\{\Upsilon_{N-1}=-\infty\}$. By the construction of $\omega_{N}$, we have
  $$\omega_{N}=a_Na_{N-1}\sigma_{N}^{*}\sigma_{N-1}^{*}(\omega_{N-2})+
 \sqrt{-1}\partial\bar{\partial}\Upsilon_{N-1}-2\pi[H_{N-1}].$$

  Inductively, denote $$\Upsilon_{k}=\log|s_N|^2_{h_{D_N}}+\sigma_{N}^{*}(a_N\log|s_{N-1}|^2_{h_{D_{N-1}}})+\cdots+
  \sigma_{N}^{*}\circ\cdots\circ  \sigma_{k+1}^{*}(a_N\cdots a_{k+1}\log|s_{k}|^2_{h_{D_{k}}})$$
  and
  $$H_{k}=D_N+\sigma_{N}^{*}(a_ND_{N-1})+\cdots+\sigma_{N}^{*}\circ\cdots\circ  \sigma_{k+1}^{*}(a_N\cdots a_{k+1} D_k)$$
  on $X_N$, for any $k\ge 0$. We know that $\Upsilon_{k}$ is a quasi-plurisubharmonic function on $X_N$ and $\Upsilon_{k}|_{|H_{k}|}=-\infty$. By the construction of all  $\omega_{i}$ ($i=1,\ldots,N$), we have
  $$\omega_{N}=a_Na_{N-1}\cdots a_1\big(\sigma_{N}^{*}\sigma_{N-1}^{*}\cdots \sigma_{1}^{*}\big)(\omega_{0})+
 \sqrt{-1}\partial\bar{\partial}\Upsilon_{0}-2\pi[H_{0}]$$

 Let $A=a_Na_{N-1}\cdots a_1$, $\Upsilon=\Upsilon_{0}$ and $H=H_0$. Then we know that
 $$\omega:=A\sigma^{*}\omega_0+
 \sqrt{-1}\partial\bar{\partial}\Upsilon-2\pi[H]$$
 is a K\"ahler metric on an open neighborhood of the closure of $M_{N}$ and $\Upsilon$ is a quasi-plurisubharmonic function on $X_N$ which has analytic singularity and $|H|=\{\Upsilon=-\infty\}$.
\end{proof}

\begin{Remark}\label{Hsnc case}
Assume that the divisor $H$ we get in Lemma \ref{sequence of blow up} is a normal crossing divisor. By the construction of $\Upsilon$, for any given point $p\in |H|$, we can find local coordinate neighborhood $(W;w_1,\ldots,w_n)$ of $p$ such that
 \begin{equation}\label{local description of upsilon}
\Upsilon=\log(\prod_{l=1}^{n}|w_l|^{2d_l})+v(w),
\end{equation}
where $d_l$ is nonnegative real number and $v(w)$ is a smooth function on $W$.
\end{Remark}
Remmert's proper mapping theorem shows that
$$\hat{H}:=\sigma(H)$$
is an analytic set in $X_0$.
Note that the map $\sigma:X_N\to X_0$ is biholomorphic from $X_N\backslash H\to X_0\backslash \hat{H}$. We have the following result when we go down to $X_0$.
\begin{Lemma}\label{property of upsilon}$\sigma_*(\Upsilon)$ is a well-defined upper-semicontinuous function on $X_0$. $\sigma_*(\Upsilon)$ is a smooth function on $X_0\backslash \hat{H}$.
\end{Lemma}
\begin{proof}
  As $\sigma:X_N\backslash H\to X_0\backslash \hat{H}$ is biholomorphic, $\sigma_*(\Upsilon)$ is a well defined quasi-plurisubharmonic function on $X_0\backslash \hat{H}$. Note that  $\Upsilon|_{|H|}=-\infty$ and $\hat{H}:=\sigma(H)$. We define $\sigma_*(\Upsilon)\equiv -\infty$ on $\hat{H}$. Let $p\in \hat{H}$ be a point. Now we prove
  $$\limsup_{z\to p} \sigma_*(\Upsilon)=-\infty.$$
  Let $B_{r}:=B(p,r)\backslash \{p\}$ and denote $\tilde{B}_r:=\sigma^{-1}(B_r)$. When $r$ is small, we know that $\tilde{B}_r$ is contained in some relatively compact open neighborhood $U$ of $H$. Note that $\Upsilon$ can be locally written as $\log|s|^2_{h}$ on $U$, where $s$ is a holomorphic function on $U$ such that $H=\{s=0\}$ and $h$ is smooth. Hence we know that $\limsup_{r\to 0 }\big(\sup_{\tilde{B}_r}\Upsilon\big)=-\infty$, which implies that $\limsup_{z\to p} \sigma_*(\Upsilon)=-\infty$. Hence $\sigma_*(\Upsilon)$ is upper-semicontinuous at any point of $\hat{H}$.

By the construction of $\Upsilon$, we know that $\Upsilon$ has analytic singularities. As $H=\{\Upsilon=-\infty\}$, we know that $\Upsilon$ is smooth on $X_N\backslash H$, hence $\sigma_*(\Upsilon)$ is a smooth function on $X_0\backslash \hat{H}$.
\end{proof}

We would like to recall some lemmas which will be used in this section.
\begin{Lemma}
[Theorem 1.5 in \cite{Demailly82}]
\label{completeness}
Let M be a K\"ahler manifold, and Z be an analytic subset of M. Assume that
$\Omega$ is a relatively compact open subset of M possessing a complete K\"ahler
metric. Then $\Omega\backslash Z $ carries a complete K\"ahler metric.

\end{Lemma}

\begin{Lemma}
[Lemma 6.9 in \cite{Demailly82}]
\label{extension of equality}
Let $\Omega$ be an open subset of $\mathbb{C}^n$ and Z be a complex analytic subset of
$\Omega$. Assume that $v$ is a (p,q-1)-form with $L^2_{\text{loc}}$ coefficients and h is
a (p,q)-form with $L^1_{\text{loc}}$ coefficients such that $\bar{\partial}v=h$ on
$\Omega\backslash Z$ (in the sense of distribution theory). Then
$\bar{\partial}v=h$ on $\Omega$.
\end{Lemma}

\begin{Lemma}
[see Lemma 9.10 in \cite{GMY-boundary5}]
\label{d-bar equation with error term}
Let ($M,\omega$) be a complete K\"ahler manifold equipped with a (non-necessarily
complete) K\"ahler metric $\omega$, and let $(Q,h)$ be a holomorphic vector bundle over $M$ with hermitian metric $h$.
Assume that $\eta$ and $g$ are smooth bounded positive functions on $M$ such that $\eta+g^{-1}$ is a smooth bounded positive function on $M$ and let
$B:=[\eta \sqrt{-1}\Theta_Q-\sqrt{-1}\partial \bar{\partial} \eta-\sqrt{-1}g
\partial\eta \wedge\bar{\partial}\eta, \Lambda_{\omega}]$. Assume that $\lambda \ge
0$ is a bounded continuous functions on $M$ such that $B+\lambda I$ is positive definite everywhere on
$\wedge^{n,q}T^*M \otimes Q$ for some $q \ge 1$. Then given a form $v \in
L^2(M,\wedge^{n,q}T^*M \otimes Q)$ such that $D^{''}v=0$ and $\int_M \langle
(B+\lambda I)^{-1}v,v\rangle_{Q,\omega} dV_{\omega} < +\infty$, there exists an approximate solution
$u \in L^2(M,\wedge^{n,q-1}T^*M \otimes Q)$ and a correcting term $\tau\in
L^2(M,\wedge^{n,q}T^*M \otimes Q)$ such that $D^{''}u+P_h(\sqrt{\lambda}\tau)=v$, where $P_h:L^2(M,\wedge^{n,q}T^*M \otimes Q)\to \text{Ker}{D''}$ is the orthogonal projection and
\begin{equation}\nonumber
\int_M(\eta+g^{-1})^{-1}|u|^2_{Q,\omega}dV_{\omega}+\int_M|\tau|^2_{Q,\omega}dV_{\omega} \le \int_M \langle (B+\lambda
I)^{-1}v,v\rangle_{Q,\omega} dV_{\omega}.
\end{equation}
\end{Lemma}
Let $M$ be a complex manifold. Let $\omega$ be a continuous hermitian metric on $M$. Let $dV_M$ be a continuous volume form on $M$. We denote by $L^2_{p,q}(M,\omega,dV_M)$ the spaces of $L^2$ integrable $(p,q)$-forms over $M$ with respect to $\omega$ and $dV_M$. It is known that $L^2_{p,q}(M,\omega,dV_M)$ is a Hilbert space.
\begin{Lemma}[see Lemma 9.1 in \cite{GMY-boundary5}]
\label{weakly convergence}
Let $\{u_n\}_{n=1}^{+\infty}$ be a sequence of $(p,q)$-forms in $L^2_{p,q}(M,\omega,dV_M)$ which is weakly convergent to $u$. Let $\{v_n\}_{n=1}^{+\infty}$ be a sequence of Lebesgue measurable real functions on $M$ which converges pointwisely to $v$. We assume that there exists a constant $C>0$ such that $|v_n|\le C$ for any $n$. Then $\{v_nu_n\}_{n=1}^{+\infty}$ weakly converges to $vu$ in $L^2_{p,q}(M,\omega,dV_M)$.
\end{Lemma}

Let $X$ be an $n-$dimensional complex manifold and $\omega$ be a hermitian metric on $X$. Let $Q$ be a holomorphic vector bundle on $X$ with rank $r$. Let $\{h_i\}_{i=1}^{+\infty}$ be a family of $C^2$ smooth hermitian metric on $Q$ and $h$ be a measurable metric on $Q$ such that $\lim_{i\to+\infty}h_i=h$ almost everywhere on $X$.  We assume that
$h_i$ is increasingly convergent to $h$ as $i\to+\infty$.

Denote $\mathcal{H}_i:=L^2(X,\wedge^{p,q}T^*X\otimes Q,h_i,dV_{\omega})$ and $\mathcal{H}:=L^2(X,\wedge^{p,q}T^*X\otimes Q,h,dV_{\omega})$. Note that $\mathcal{H}\subset \mathcal{H}_i\subset \mathcal{H}_1$ for any $i\in\mathbb{Z}_{>0}$. Denote $P_i:=\mathcal{H}_i\to \text{Ker}D''$ and $P:=\mathcal{H}\to \text{Ker}D''$ be the orthogonal projections with respect to $h_i$ and $h$ respectively.
\begin{Lemma}[see Lemma 9.9 in \cite{GMY-boundary5}]
\label{weakly converge lemma} For any sequence of $Q$-valued $(n,q)$-forms $\{f_i\}_{i=1}^{+\infty}$ which satisfies $f_i\in\mathcal{H}_i$ and $||f_i||_{h_i}\le C_1$ for some constant $C_1>0$, there exists a $Q$-valued $(n,q)$-form $f_0\in \mathcal{H}$ such that there exists a subsequence of $\{f_i\}_{i=1}^{+\infty}$ (also denoted by $\{f_i\}_{i=1}^{+\infty}$) weakly converges to $f_0$ in $\mathcal{H}_1$ and $P_i(f_i)$ weakly converges to $P(f_0)$ in $\mathcal{H}_1$.
\end{Lemma}

The following proposition will be used to deal with the singular metric $h$ on $E$.
\begin{Proposition}
\label{key pro l2 method} Let $c(t)$ be a positive continuous function on $[0,+\infty)$ such that $\alpha_1:=\inf\limits_{t\ge 0}c(t)>0$ and $\alpha_2:=\sup\limits_{t\ge 0}c(t)e^{-t}<+\infty.$ Let $\Omega\Subset\tilde{\Omega}\Subset \mathbb{C}^n$ be two bounded pseudoconvex domains. Let $E:=\Omega_1\times \mathbb{C}^r$ and $h$ be a singular hermitian metric on $E$ in the sense of Definition \ref{singular metric}. We assume that $\Theta_h(E)\ge^s_{Nak} 0$ in the sense of Definition \ref{singular nak}. Let $\psi$ be a quasi-plurisubharmonic function on $\Omega_1$. Assume that $\psi$ has neat analytic singularities and the singularities of $\psi$ are log canonical along the zero variety $Y=V(\mathcal{I}(\psi))$. Set
$$U:=\{x\in \Omega:\psi(x)<0\}.$$
We assume that
$$\sqrt{-1}\partial\bar{\partial}\psi\ge -\beta \sqrt{-1}\partial\bar{\partial}|z|^2,$$
on $\Omega$ for some $\beta\ge 0$, where $z:=(z_1,\ldots,z_n)$ is the coordinate on $\mathbb{C}^n$. Then for any $\beta_1\in(0,1)$ and every $E$-valued holomorphic $(n,0)$-form $f$ on $U$ satisfying
$$\int_U c(-\psi)|f|^2_h<+\infty,$$
there exists an $E$-valued holomorphic $(n,0)$-form $F$ on $\Omega$ satisfying $F=f$ on $Y$,
\begin{equation}
\label{eq:l2 method estimate 1}
\int_U c(-\psi)|F|^2_h\le e^{2\beta\sup\limits_{\Omega}|z|^2}\big(2+\frac{72\alpha_2}{\alpha_1\beta_1}\big)\int_U c(-\psi)|f|^2_h<+\infty,
\end{equation}
and
\begin{equation}
\label{eq:l2 method estimate 2}
\int_U \frac{|F|^2_h}{(1+e^{\psi})^{1+\beta_1}}\le e^{2\beta\sup\limits_{\Omega}|z|^2}\big(\frac{1}{\alpha_1}+\frac{36}{\alpha_1\beta_12^{\beta_1}}\big)\int_U c(-\psi)|f|^2_h<+\infty.
\end{equation}
\end{Proposition}
\begin{proof}The proof is a modification of a proposition in \cite{Demaillybook} (see also \cite{ZZ2019}).

As $\Omega$ is pseudoconvex domain, there exists a sequence of pseudoconvex domains
$\Omega_k$ satisfying $\Omega_1 \Subset  \Omega_2\Subset  ...\Subset
\Omega_k\Subset  \Omega_{k+1}\Subset  ...$ and $\cup_{k=1}^n \Omega_k=\Omega$.

It follows from $\Omega\Subset\Omega_1\Subset \mathbb{C}^n$ and $\Theta_h(E)\ge^s_{Nak} 0$ in the sense of Definition \ref{singular nak} that we know that there exists a sequence of $C^2$ smooth metrics $h_j$ ($j\ge 1$) convergent point-wisely to $h$ on a neighborhood of $\overline{\Omega}$ which satisfies
\par
(1) for any $x\in \Omega:\ |e_x|_{h_{j}}\le |e_x|_{h_{j+1}},$ for any $j\ge 1$ and any $e_x\in E_x$;
\par
(2) $\Theta_{h_{j}}(E)\ge_{Nak} -\lambda_{j}\omega\otimes Id_E$ on $\Omega$;
\par
(3)  $\lambda_{j}\to 0$ a.e. on $\Omega$, where $\lambda_{j}$ is a sequence of continuous functions on $\overline{\Omega}$;
\par
(4)  $0\le \lambda_{j}\le \lambda_0$ on $\Omega$, for any $s\ge 1$, where $\lambda_0$ is a continuous function $\overline{\Omega}$.

Let $\theta:\mathbb{R}\to [0,1]$ be a smooth function such that $\theta=1$ on $(-\infty,\frac{1}{4})$, $\theta=0$ on $(\frac{3}{4},+\infty)$ and $|\theta'|\le 3$ on $\mathbb{R}$. Denote $\tilde{f}=\theta(e^{\psi})f$. Then $\tilde{f}$ is smooth on $\Omega$ and $\tilde{f}=f$ on $Y\cap\Omega$. Hence $g:=\bar{\partial}\tilde{f}$ is well defined.

Fix $k$ and $j$ temporarily.
Let $\sum :=\{\psi=-\infty\}$. It follows from Lemma \ref{completeness} that $\Omega_k\backslash \sum$ is a complete K\"ahler manifold. Let $\omega$ be the Euclidean metric on $\Omega_k\backslash \sum$. We equip $E$ with the metric
$$\tilde{h}:=h_je^{-\psi-\beta_1\log(1+e^\psi)-2\beta|z|^2}$$
on $\Omega_k\backslash \sum$. Denote $B:=[\sqrt{-1}\Theta_{\tilde{h}},\Lambda_{\omega}]$. Direct calculation shows
\begin{equation*}
  \begin{split}
     \sqrt{-1}\Theta_{\tilde{h}} =& \sqrt{-1}\Theta_{\tilde{h}_j}+\sqrt{-1}\partial\bar{\partial}\psi+\beta_1 \sqrt{-1}\partial\bar{\partial}\log(1+e^{\psi})+2\beta\sqrt{-1}\partial\bar{\partial}|z|^2\\
     =& (\sqrt{-1}\Theta_{\tilde{h}_j}+\lambda_j\omega\otimes \text{Id}_E)-\lambda_j\omega\otimes \text{Id}_E+(1+\beta_1\frac{e^{\psi}}{1+e^{\psi}})\sqrt{-1}\partial\bar{\partial}\psi\\
       &+2\beta\sqrt{-1}\partial\bar{\partial}|z|^2+\beta_1\frac{e^{\psi}}{(1+e^{\psi})^2}\partial{\psi}\wedge\bar{\partial}\psi\\
       \ge&-\lambda_j\omega\otimes \text{Id}_E+\beta_1\frac{e^{\psi}}{(1+e^{\psi})^2}\partial{\psi}\wedge\bar{\partial}\psi.
  \end{split}
\end{equation*}
Denote $\tilde{\lambda}_j=\lambda_j+\frac{1}{j}$, then we know that the operator $B+\tilde{\lambda}_j \text{Id}_E$ is positive. Hence, for any $E$-valued $(n,1)$-form $\alpha$, we have
\begin{equation}
\label{curvature inequality}
\begin{split}
&\langle(B+\tilde{\lambda}_j \text{Id}_E)\alpha,\alpha\rangle_{\tilde h}\\
\ge&\langle[(\beta_1\frac{e^{\psi}}{(1+e^{\psi})^2}\partial\psi\wedge\bar{\partial}\psi)\otimes\text{Id}_E,
\Lambda_{\omega}]\alpha,\alpha\rangle_{\tilde h}
\end{split}
\end{equation}
Note that $g=\bar{\partial}\tilde{f}=\theta'(e^{\psi})e^{\psi}\bar{\partial}\psi\wedge f$. Hence we have $\langle(B+\tilde{\lambda}_j \text{Id}_E)^{-1}g,g\rangle_{\tilde{h}}|_{\Omega_k\backslash U}=0$ and
\begin{equation*}
  \begin{split}
     \langle(B+\tilde{\lambda}_j \text{Id}_E)^{-1}g,g\rangle_{\tilde{h}}|_{(\Omega_k\cap U)\backslash \sum} & \le \frac{(1+e^\psi)^2}{\beta_1e^{\psi}}|\theta'(e^\psi)e^\psi f|^2_{h_j}e^{-\psi-\beta_1\log(1+e^\psi)-2\beta|z|^2}\\
       &= \frac{(1+e^\psi)^{2-\beta_1}}{\beta_1}|\theta'(e^\psi)f|^2_{h_j}e^{-2\beta|z|^2}\\
       &\le\frac{36}{2^{\beta_1}\beta_1}|f|^2_{h_j}e^{-2\beta|z|^2}
  \end{split}
\end{equation*}

It follows from Lemma \ref{d-bar equation with error term} that there exists an approximate solution
$u_{k,j} \in L^2(\Omega_k\backslash \sum,K_{\Omega} \otimes E,\tilde{h})$ and a correcting term $\tau\in
L^2(\Omega_k\backslash \sum,\wedge^{n,1}T^*\Omega\otimes E,\tilde{h})$ such that $D^{''}u_{k,j}+P_{\tilde{h}}(\sqrt{\tilde{\lambda}_j}\tau_{k,j})=g$ holds on $\Omega_k\backslash \sum$, where $P_{\tilde{h}}:L^2(\Omega_k\backslash \sum,\wedge^{n,1}T^*\Omega \otimes E,\tilde{h})\to \text{Ker}{D''}$ is the orthogonal projection and
\begin{equation}
  \begin{split}
\int_{\Omega_k\backslash \sum}|u_{k,j}|^2_{\tilde{h}}+\int_{\Omega_k\backslash \sum}|\tau_{k,j}|^2_{\tilde{h}} &\le \int_{\Omega_k\backslash \sum} \langle (B+\tilde{\lambda}_j
\text{Id}_E)^{-1}g,g\rangle_{\tilde{h}}\\
&\le \frac{36}{2^{\beta_1}\beta_1}\int_{U}|f|^2_{h_j}e^{-2\beta|z|^2}.
  \end{split}
\end{equation}
Thus, we have
\begin{equation}\label{eq:estimate for ukj}
\begin{split}
  \int_{\Omega_k\backslash \sum}\frac{|u_{k,j}|^2_{h_j}e^{-2\beta|z|^2}}{e^\psi(1+e^\psi)^{\beta_1}}
\le& \frac{36}{2^{\beta_1}\beta_1}\int_{U}|f|^2_{h_j}e^{-2\beta|z|^2}\\
\le& \frac{36}{2^{\beta_1}\beta_1\alpha_1}\int_{U}|f|^2_{h_j}e^{-2\beta|z|^2}c(-\psi)<+\infty.
  \end{split}
\end{equation}
Similarly, we also have
\begin{equation}\label{eq:estimate for taukj}
\begin{split}
  \int_{\Omega_k\backslash \sum}\frac{|\tau_{k,j}|^2_{h_j}e^{-2\beta|z|^2}}{e^\psi(1+e^\psi)^{\beta_1}}
\le \frac{36}{2^{\beta_1}\beta_1\alpha_1}\int_{U}|f|^2_{h_j}e^{-2\beta|z|^2}c(-\psi)<+\infty.
  \end{split}
\end{equation}
It follows from $e^{-\psi}$, $(1+e^\psi)^{-\beta_1}$ and $e^{-2\beta|z|^2}$ have positive lower bound on $\Omega_k$, $h_j\ge h_1$ for any $j\ge 1$, inequalities \eqref{eq:estimate for ukj} and \eqref{eq:estimate for taukj} that we know  both $u_{k,j}$ and $\tau_{k,j}$ are $L^2_{\text{loc}}$. Then by Lemma \ref{extension of equality}, we know that
\begin{equation}\label{eq:d-bar equation}
D^{''}u_{k,j}+P_{\tilde{h}}(\sqrt{\tilde{\lambda}_j}\tau_{k,j})=g
\end{equation}
holds on $\Omega_k$. As $\Omega$ is pseudoconvex domain, there exists a sequence of smooth quasi-plurisubharmonic functions $\psi_m$ decreasingly converges  to $\psi$ on $\Omega$ as $m\to +\infty$. By inequalities \eqref{eq:estimate for ukj} and \eqref{eq:estimate for taukj}, we have
\begin{equation}\label{eq:estimate for ukj on omega}
\begin{split}
  \int_{\Omega_k}\frac{|u_{k,j}|^2_{h_j}e^{-2\beta|z|^2}}{e^{\psi_m}(1+e^\psi)^{\beta_1}}
  \le&\int_{\Omega_k}\frac{|u_{k,j}|^2_{h_j}e^{-2\beta|z|^2}}{e^{\psi}(1+e^\psi)^{\beta_1}}\\
\le& \frac{36}{2^{\beta_1}\beta_1\alpha_1}\int_{U}|f|^2_{h_j}e^{-2\beta|z|^2}c(-\psi)\\
 \le &\frac{36}{2^{\beta_1}\beta_1\alpha_1}\int_{U}|f|^2_{h}e^{-2\beta|z|^2}c(-\psi)<+\infty,
  \end{split}
\end{equation}
and
\begin{equation}\label{eq:estimate for taukj on omega}
\begin{split}
  \int_{\Omega_k}\frac{|\tau_{k,j}|^2_{h_j}e^{-2\beta|z|^2}}{e^\psi(1+e^\psi)^{\beta_1}}
\le& \frac{36}{2^{\beta_1}\beta_1\alpha_1}\int_{U}|f|^2_{h_j}e^{-2\beta|z|^2}c(-\psi)\\
 \le &\frac{36}{2^{\beta_1}\beta_1\alpha_1}\int_{U}|f|^2_{h}e^{-2\beta|z|^2}c(-\psi)<+\infty.
  \end{split}
\end{equation}
It follows from $e^{-\psi}$, $(1+e^\psi)^{-\beta_1}$  and $e^{-2\beta|z|^2}$ have positive lower bound on $\Omega_k$, $h_j\ge h_1$ for any $j\ge 1$ and inequality \eqref{eq:estimate for ukj on omega} that we have
$$\sup_{j}\int_{\Omega_k}|u_{k,j}|^2_{h_1}<+\infty. $$
Since the closed unit ball of the Hilbert space is
weakly compact, we can extract a subsequence $u_{k,j'}$ weakly convergent to $u_{k}$ in $L^2(\Omega_k,K_{\Omega} \otimes E,h_1)$. It follows from $e^{-2\beta|z|^2}$ have positive upper bound on $\Omega_k$, $\psi_m$ is smooth on $\Omega$, $\Omega_k\Subset \Omega$ and Lemma \ref{weakly convergence} that we know $\sqrt{e^{-\psi_m}(1+e^\psi)^{-\beta_1}e^{-2\beta|z|^2}}u_{k,j'}$ weakly converges to $\sqrt{e^{-\psi_m}(1+e^\psi)^{-\beta_1}e^{-2\beta|z|^2}}u_{k}$ in $L^2(\Omega_k,K_{\Omega} \otimes E,h_1)$ as $j'\to+\infty$.

For fixed $j\in\mathbb{Z}_{\ge 1}$, as $h_1$ and $h_j$ are both $C^2$ smooth hermitian metrics on $\Omega_k$ and $\Omega_k\subset\subset X$, we know that the two norms in $L^2(\Omega_k, K_\Omega\otimes E, h_1)$ and $ L^2(\Omega_k, K_\Omega\otimes E, h_j)$ are equivalent. Hence we know that $\sqrt{e^{-\psi_m}(1+e^\psi)^{-\beta_1}e^{-2\beta|z|^2}}u_{k,j'}$ weakly converges to $\sqrt{e^{-\psi_m}(1+e^\psi)^{-\beta_1}e^{-2\beta|z|^2}}u_{k}$ in $L^2(\Omega_k,K_{\Omega} \otimes E,h_j)$ as $j'\to+\infty$.
Then it follows from \eqref{eq:estimate for ukj on omega} that we have
\begin{equation*}
\begin{split}
 &\int_{\Omega_k}\frac{|u_{k}|^2_{h_j}e^{-2\beta|z|^2}}{e^{\psi_m}(1+e^\psi)^{\beta_1}}\\
  \le&\liminf_{j'\to +\infty}\int_{\Omega_k}\frac{|u_{k,j'}|^2_{h_j}e^{-2\beta|z|^2}}{e^{\psi_m}(1+e^\psi)^{\beta_1}}\\
  \le&\liminf_{j'\to +\infty}\int_{\Omega_k}\frac{|u_{k,j'}|^2_{h_{j'}}e^{-2\beta|z|^2}}{e^{\psi_m}(1+e^\psi)^{\beta_1}}\\
\le&\liminf_{j'\to +\infty}\int_{\Omega_k}\frac{|u_{k,j'}|^2_{h_{j'}}e^{-2\beta|z|^2}}{e^{\psi}(1+e^\psi)^{\beta_1}}\\
\le& \frac{36}{2^{\beta_1}\beta_1\alpha_1}\int_{U}|f|^2_{h_{j'}}e^{-2\beta|z|^2}c(-\psi)\\
 \le &\frac{36}{2^{\beta_1}\beta_1\alpha_1}\int_{U}|f|^2_{h}e^{-2\beta|z|^2}c(-\psi)<+\infty.
\end{split}
\end{equation*}
Letting $j\to+\infty$ and $m\to+\infty$, by monotone convergence theorem, we have
\begin{equation}
\begin{split}
\label{eq:estimate fo uk}
 \int_{\Omega_k}\frac{|u_{k}|^2_{h}e^{-2\beta|z|^2}}{e^{\psi}(1+e^\psi)^{\beta_1}}
 \le\frac{36}{2^{\beta_1}\beta_1\alpha_1}\int_{U}|f|^2_{h}e^{-2\beta|z|^2}c(-\psi)<+\infty.
\end{split}
\end{equation}

It follows from $e^{-\psi}$, $(1+e^\psi)^{-\beta_1}$, $e^{-2\beta|z|^2}$ have positive lower bound on $\Omega_k$ and inequality \eqref{eq:estimate for taukj on omega} that we have
$$\sup_{j'}\int_{\Omega_k}|\tau_{k,j'}|^2_{h_{j'}}<+\infty. $$
As $h_j\ge h_1$ for any $j\ge 1$, we also have
$$\sup_{j'}\int_{\Omega_k}|\tau_{k,j'}|^2_{h_1}<+\infty. $$
Since the closed unit ball of the Hilbert space is
weakly compact, we can extract a subsequence of $\tau_{k,j'}$ (also denoted by $\tau_{k,j'}$) weakly converges to $\tau_{k}$ in $L^2(\Omega_k,\wedge^{n,1}T^*\Omega \otimes E,h_1)$ as $j' \to +\infty$. As $\Omega_k\Subset \Omega$, $0\le \tilde{\lambda}_{j'}\le\lambda+1 $, we know that

$$\sup_{j'}\int_{\Omega_k}\tilde{\lambda}_{j'}|\tau_{k,j'}|^2_{h_{j'}}<+\infty. $$
It follows from Lemma \ref{weakly converge lemma} that we have a subsequence of $\{\sqrt{\tilde{\lambda}_j}\tau_{k,j'}\}_{j'}^{+\infty}$ (also denoted by $\{\sqrt{\tilde{\lambda}_j}\tau_{k,j'}\}_{j'}^{+\infty}$) weakly convergent to some $\tilde{\tau}_k$ in $L^2(\Omega_k,\wedge^{n,1}T^*\Omega \otimes E,h_1)$ and $P_{j'}(\sqrt{\tilde{\lambda}_j}\tau_{k,j'})$ weakly converges to $P(\tilde{\tau}_k)$ as $j' \to +\infty$.

It follows from  $0\le \tilde{\lambda}_{j'}\le\lambda+1$, $\Omega_k$ is relatively compact in $\Omega$ and Lemma \ref{weakly convergence} that we know $\sqrt{\tilde{\lambda}_j}\tau_{k,j'}$ weakly
converges to $0$ in $L^2(\Omega_k,\wedge^{n,1}T^*\Omega \otimes E,h_1)$. It follows from the uniqueness of weak limit that we know $\tilde{\tau}_k=0$. Then we have $P_{j'}(\sqrt{\tilde{\lambda}_j}\tau_{k,j'})$ weakly converges to $0$ in $L^2(\Omega_k,\wedge^{n,1}T^*\Omega \otimes E,h_1)$.

Replacing $j$ by $j'$ in equality \eqref{eq:d-bar equation}, we have
\begin{equation}\nonumber
D^{''}u_{k}=g=\bar{\partial}\tilde{f}=\theta'(e^{\psi})e^{\psi}\bar{\partial}\psi\wedge f.
\end{equation}

Denote $F_k:=\tilde{f}-u_k$. Then $\bar{\partial} F_k=0$ on $\Omega_k$. Then $F_k$ is holomorphic on $\Omega_k$ and $u_k$ is smooth on $\Omega_k$. Then it follows from inequality \eqref{eq:estimate fo uk} and $e^{-\psi}$ is not integrable along $Y$ that $u_k=0$ on $Y\cap \Omega_k$. Hence $F_k=f$ on $Y\cap \Omega_k$.

It follows from inequality \eqref{eq:estimate fo uk} that we have
\begin{equation}\nonumber
\begin{split}
\label{eq:estimate fo uk 2}
\int_{U\cap\Omega_k}|u_{k}|^2_{h}e^{-2\beta|z|^2}c(-\psi)
&\le2^{\beta_1}\alpha_2\int_{U\cap\Omega_k}\frac{|u_{k}|^2_{h}e^{-2\beta|z|^2}}{e^{\psi}(1+e^\psi)^{\beta_1}}\\
 &\le\frac{36\alpha_2}{\beta_1\alpha_1}\int_{U}|f|^2_{h}e^{-2\beta|z|^2}c(-\psi)<+\infty.
\end{split}
\end{equation}
Note that $|F_k|^2_{h}|_{U\cap \Omega_k}\le 2|\tilde{f}|^2_h+2|u_k|^2_h\le 2|f|^2_h+2|u_k|^2_h$ and $|z|^2\ge 0$, then we have
\begin{equation}
\begin{split}
\label{eq:estimate fo Fk 1}
\int_{U\cap\Omega_k}|F_{k}|^2_{h}e^{-2\beta|z|^2}c(-\psi)
&\le 2 \int_{U\cap\Omega_k}(|f|^2_{h}+|u_{k}|^2_{h})e^{-2\beta|z|^2}c(-\psi)\\
 &\le\big(2+\frac{72\alpha_2}{\beta_1\alpha_1}\big)\int_{U}|f|^2_{h}e^{-2\beta|z|^2}c(-\psi)\\
 &\le\big(2+\frac{72\alpha_2}{\beta_1\alpha_1}\big)\int_{U}|f|^2_{h}c(-\psi)<+\infty.
\end{split}
\end{equation}

Since $$\langle a_1+a_2,a_1+a_2\rangle\le (1+c)\langle a_1,a_1\rangle+(1+\frac{1}{c})\langle a_2,a_2\rangle$$
holds for any $a_1,a_2$ in an inner product space $(H,\langle \cdot,\cdot \rangle)$, we have
$$|F_k|^2_{h}|_{U\cap \Omega_k}\le (1+e^{\psi})|f|^2_h+(1+e^{-\psi})|u_k|^2_h.$$
Then we know
$$\frac{|F_k|^2_{h}}{(1+e^\psi)^{1+\beta_1}}|_{U\cap \Omega_k}\le |f|^2_h+\frac{|u_k|^2_h}{e^{\psi}(1+e^{\psi})^{\beta_1}}.$$
Note that $|F_k|^2_{h}|_{ \Omega_k\backslash U}=|u_k|^2_h$, then we get
$$\frac{|F_k|^2_{h}}{(1+e^\psi)^{1+\beta_1}}|_{ \Omega_k\backslash U}\le \frac{|u_k|^2_{h}}{e^{\psi}(1+e^\psi)^{\beta_1}}|_{ \Omega_k\backslash U}.$$
Combining with inequality \eqref{eq:estimate fo uk}, we have
\begin{equation}
\begin{split}
\label{eq:estimate fo Fk 2}
\int_{\Omega_k}\frac{|F_k|^2_{h}e^{-2\beta|z|^2}}{(1+e^\psi)^{1+\beta_1}}
&\le  \int_{U}|f|^2_{h}e^{-2\beta|z|^2}+\int_{\Omega_k}\frac{|u_k|^2_{h}e^{-2\beta|z|^2}}{e^{\psi}(1+e^\psi)^{\beta_1}}\\
 &\le\big(\frac{1}{\alpha_1}+\frac{36}{2^{\beta_1}\beta_1\alpha_1}\big)\int_{U}|f|^2_{h}e^{-2\beta|z|^2}c(-\psi)\\
 &\le\big(\frac{1}{\alpha_1}+\frac{36}{2^{\beta_1}\beta_1\alpha_1}\big)\int_{U}|f|^2_{h}c(-\psi)<+\infty.
\end{split}
\end{equation}

Note that $(1+e^\psi)^{-(1+\beta_1)}$ and $e^{-2\beta|z|^2}$  have positive lower bound on any $\Omega_k\Subset \Omega$, for any $k_1>k$, then we have
$$\sup_{k_1>k}\int_{\Omega_k}|F_{k_1}|^2_{h}<+\infty.$$
Note that $h_1\le h$, then we have
$$\sup_{k_1>k}\int_{\Omega_k}|F_{k_1}|^2_{h_1}<+\infty.$$
By diagonal method, there exists a subsequence $F_{k'}$ uniformly convergent on any
$\overline{\Omega_k}$ to an $E$-valued holomorphic $(n,0)$-form on $\Omega$ denoted by $F$. It follows from inequality \eqref{eq:estimate fo Fk 2} and Fatou's lemma that we have
\begin{equation}
\begin{split}
\label{eq:estimate fo Fk 3}
\int_{\Omega_k}\frac{|F|^2_{h_i}e^{-2\beta|z|^2}}{(1+e^\psi)^{1+\beta_1}}
\le &\liminf_{k'\to +\infty}\int_{\Omega_k}\frac{|F_{k'}|^2_{h_i}e^{-2\beta|z|^2}}{(1+e^\psi)^{1+\beta_1}}\\
\le &\liminf_{k'\to +\infty}\int_{\Omega_k}\frac{|F_{k'}|^2_{h}e^{-2\beta|z|^2}}{(1+e^\psi)^{1+\beta_1}}\\
\le &\liminf_{k'\to +\infty}\int_{\Omega_{k'}}\frac{|F_{k'}|^2_{h}e^{-2\beta|z|^2}}{(1+e^\psi)^{1+\beta_1}}\\
\le & \big(\frac{1}{\alpha_1}+\frac{36}{2^{\beta_1}\beta_1\alpha_1}\big)\int_{U}|f|^2_{h}c(-\psi).
\end{split}
\end{equation}
Letting $i\to +\infty$ and $k\to+\infty$ in inequality \eqref{eq:estimate fo Fk 3}, by monotone convergence theorem, we have
\begin{equation}\nonumber
\begin{split}
\int_{\Omega}\frac{|F|^2_{h}e^{-2\beta|z|^2}}{(1+e^\psi)^{1+\beta_1}}
\le  \big(\frac{1}{\alpha_1}+\frac{36}{2^{\beta_1}\beta_1\alpha_1}\big)\int_{U}|f|^2_{h}c(-\psi).
\end{split}
\end{equation}
Then we have
\begin{equation}\nonumber
\begin{split}
\int_{\Omega}\frac{|F|^2_{h}}{(1+e^\psi)^{1+\beta_1}}
\le e^{2\beta\sup_{\Omega} |z|^2} \big(\frac{1}{\alpha_1}+\frac{36}{2^{\beta_1}\beta_1\alpha_1}\big)\int_{U}|f|^2_{h}c(-\psi).
\end{split}
\end{equation}
It follows from inequality \eqref{eq:estimate fo Fk 1} and Fatou's lemma that we have
\begin{equation}
\begin{split}
\label{eq:estimate fo Fk 4}
\int_{U\cap\Omega_k}|F|^2_{h_i}e^{-2\beta|z|^2}c(-\psi)
&\le\liminf_{k'\to +\infty}\int_{U\cap\Omega_k}|F_{k'}|^2_{h_i}e^{-2\beta|z|^2}c(-\psi)\\
&\le\liminf_{k'\to +\infty}\int_{U\cap\Omega_k}|F_{k'}|^2_{h}e^{-2\beta|z|^2}c(-\psi)\\
&\le\liminf_{k'\to +\infty}\int_{U\cap\Omega_{k'}}|F_{k'}|^2_{h}e^{-2\beta|z|^2}c(-\psi)\\
 &\le\big(2+\frac{72\alpha_2}{\beta_1\alpha_1}\big)\int_{U}|f|^2_{h}c(-\psi).
\end{split}
\end{equation}
Letting $i\to +\infty$ and $k\to+\infty$ in inequality \eqref{eq:estimate fo Fk 4}, by monotone convergence theorem, we have
\begin{equation}\nonumber
\begin{split}
\int_{U}|F|^2_{h}e^{-2\beta|z|^2}c(-\psi)
\le  \big(2+\frac{72\alpha_2}{\beta_1\alpha_1}\big)\int_{U}|f|^2_{h}c(-\psi).
\end{split}
\end{equation}
Then we have
\begin{equation}\nonumber
\begin{split}
\int_{U}|F|^2_{h}c(-\psi)
\le  e^{2\beta\sup_{\Omega}|z|^2}\big(2+\frac{72\alpha_2}{\beta_1\alpha_1}\big)\int_{U}|f|^2_{h}c(-\psi).
\end{split}
\end{equation}
Hence $F$ satisfies the desired $L^2$ estimates. Proposition \ref{key pro l2 method} is proved.
\end{proof}

The following optimal $L^2$ extension theorem for vector bundles with smooth hermitian metric on Stein manifolds will be used in our discussion.
\begin{Theorem}[see \cite{guan-zhou13ap}]\label{l2 extension on stein} Let $c(t)\in \mathcal{G}_{T,\delta}$ for some $T\in(-\infty,+\infty)$ and $0<\delta<+\infty$.
Let $M$ be a Stein manifold and $\omega$ be a hermitian metric on $M$. Let $h$ be a smooth hermitian metric on a holomorphic vector bundle $E$ on $M$ with rank $r$. Let $\psi<-T$ be a quasi-plurisubharmonic function on $X$ with neat analytic singularities. Let $Y:=V(\mathcal{I}(\psi))$ and assume that $\psi$ has log canonical singularities along $Y$. Assume that\\
(1) $\sqrt{-1}\Theta_h+\sqrt{1}\partial\bar{\partial}\psi$ is Nakano semi-positive on $M\backslash\{\psi=-\infty\}$, \\
(2) there exists a continuous function $a(t)$ on $(T,+\infty]$ such that $0<a(t)\le s(t)$ and $a(-\psi)\sqrt{-1}\Theta_{he^{-\psi}}+\sqrt{1}\partial\bar{\partial}\psi$ is Nakano semi-positive on $M\backslash\{\psi=-\infty\}$, where
$$s(t):=\frac{
  \int_{T}^{t}\big(\frac{1}{\delta}c(T)e^{-T}+\int_{T}^{t}c(t_1)e^{-t_1}dt_1\big)dt_2+\frac{1}{\delta^2}c(T)e^{-T}}
  {\frac{1}{\delta}c(T)e^{-T}+\int_{T}^{t}c(t_1)e^{-t_1}dt_1}.$$
Then  for any holomorphic section $f$ of $K_M\otimes E|_Y$ on $Y$ satisfying
$$\int_{Y_0}|f|^2_{\omega,h}dV_{M,\omega}[\psi]<+\infty,$$
there exists a holomorphic section $F$ of   $K_M\otimes E$ on $M$ satisfying $F|_Y=f$ and
$$\int_M c(-\psi)|F|^2_{\omega,h}dV_{M,\omega}\le \left(\frac{1}{\delta}c(T)e^{-T}+\int_{T}^{+\infty}c(t_1)e^{-t_1}dt_1\right)\int_{Y_0}|f|^2_{\omega,h}dV_{M,\omega}[\psi].$$

\end{Theorem}
\begin{Remark}\label{weakly pseudoconvex case extension}When $M$ is a weakly pseudoconvex  K\"ahler manifold and $f$ is only defined on $Y_0$, Theorem \ref{l2 extension on stein} can be referred to \cite{ZZ2019} (see also \cite{ZhouZhu-jdg}).
\end{Remark}

By using Theorem \ref{l2 extension on stein}, we present the following $L^2$ extension theorem for singular hermitian metric on vector bundles in the local case. Let $\Omega\Subset \Omega_1 \Subset \mathbb{C}^n$ be two balls in $\mathbb{C}^n$. Let $E:=\Omega_1\times \mathbb{C}^r$, where $r\ge 1$.  Let $h$ be a singular hermitian metric on $E$. We assume that  $\Theta_{h}(E)\ge^s_{Nak} 0$ in the sense of Definition \ref{singular nak}. Let $\psi<-T$ be a plurisubharmonic function on $\Omega$ with neat analytic singularities. Let $Y:=V(\mathcal{I}(\psi))$ and assume that $\psi$ has log canonical singularities along $Y$. Denote $Y^0=Y_{\rm{reg}}$ the regular point set of $Y$.

\begin{Proposition}\label{key pro l2 extension local}
Let $c(t)$ be the same as in Theorem \ref{l2 extension on stein}. We also assume that $\inf_{t\in (T,+\infty)}c(t)>0$. Let $\Omega$, $\Omega_1$, $E$, $h$, $\psi$ and $Y$ be as above.
Then for any holomorphic section $f$ of $K_{\Omega}\otimes E|_{Y_0}$ on $Y_0$ satisfying
$$\int_{Y_0}|f|^2_hdV_{\Omega}[\psi]<+\infty,$$
there exists a real constant $C_{\Omega}$ (depends on $\Omega$) and a holomorphic section $F$ of   $K_{\Omega}\otimes E$ on $\Omega$ satisfying $F|_{Y_0}=f$ and
$$\int_{\Omega} c(-\psi)|F|^2_hdV_{\Omega}\le C_{\Omega}\left (\frac{1}{\delta}c(T)e^{-T}+\int_{T}^{+\infty}c(t_1)e^{-t_1}dt_1\right)\int_{Y_0}|f|^2_hdV_{\Omega}[\psi].$$

\end{Proposition}
\begin{proof}It follows from $\Omega\Subset\Omega_1\Subset \mathbb{C}^n$ and $\Theta_h(E)\ge^s_{Nak} 0$ in the sense of Definition \ref{singular nak} that we know that there exists a sequence of $C^2$ smooth metrics $h_j$ ($j\ge 1$) convergent point-wisely to $h$ on a neighborhood of $\overline{\Omega}$ which satisfies
\par
(1) for any $x\in \Omega:\ |e_x|_{h_{j}}\le |e_x|_{h_{j+1}},$ for any $j\ge 1$ and any $e_x\in E_x$;
\par
(2) $\Theta_{h_{j}}(E)\ge_{Nak} -\lambda_{s}\omega\otimes Id_E$ on $\overline{\Omega}$, where $\omega$ is the metric form of Euclidean metric on $\Omega_1$;
\par
(3)  $\lambda_{j}\to 0$ a.e. on $\overline{\Omega}$, where $\lambda_{j}$ is a sequence of continuous functions on $\overline{\Omega}$;
\par
(4)  $0\le \lambda_{j}\le \lambda_0$ on $\overline{\Omega}$, for any $s\ge 1$, where $\lambda_0$ is a continuous function $\overline{\Omega}$.

By (2) and (4), we know that
$$\Theta_{h_{j}}(E)\ge_{Nak} -\lambda_{s}\omega\otimes Id_E\ge -\lambda_0\omega\otimes Id_E\ge
-M_{0}\omega\otimes Id_E$$
on $\Omega$, where $M_0:=\sup\limits_{\overline{\Omega}}\lambda_0$. Since the case is local, we know that there exists a smooth plurisubharmonic function $u$ on $\overline{\Omega}$ such that $\sqrt{-1}\partial\bar{\partial}u=(M_{0}+1)\omega$. Denote $\tilde{h}_j:=h_{j}e^{-u}$. Then $\tilde{h}_j$ is a sequence of Nakano positive smooth metrics on $E$. Note that $u$ is smooth on $\overline{\Omega}$ and $h_j\le h$ on $\Omega$, it follows from $\int_{Y_0}|f|^2_hdV_{\Omega}[\psi]<+\infty$ that we have $$\int_{Y_0}|f|^2_{\tilde{h}_j}dV_{\Omega}[\psi]<+\infty,$$
 for any $j \ge 1$. For fixed $j\ge 1$, it follows from Theorem \ref{l2 extension on stein} and Remark \ref{weakly pseudoconvex case extension} that we know that there exists a holomorphic section $F_j$ of   $K_{\Omega}\otimes E$ on $\Omega$ satisfying $F_j|_{Y_0}=f$ and
$$\int_{\Omega} c(-\psi)|F_j|^2_{\tilde{h}_j}dV_{\Omega}\le \left (\frac{1}{\delta}c(T)e^{-T}+\int_{T}^{+\infty}c(t_1)e^{-t_1}dt_1\right)\int_{Y_0}|f|^2_{\tilde{h}_j}dV_{\Omega}[\psi].$$
Hence we have
\begin{equation}\label{estimate in local l2 extenion 1}
\begin{split}
&\int_{\Omega} c(-\psi)|F_j|^2_{h_j}dV_{\Omega}\\
\le &\frac{1}{\inf_{\Omega}e^{-u}}\int_{\Omega} c(-\psi)|F_j|^2_{\tilde{h}_j}dV_{\Omega}\\
\le &\frac{1}{\inf_{\Omega}e^{-u}}\left (\frac{1}{\delta}c(T)e^{-T}+\int_{T}^{+\infty}c(t_1)e^{-t_1}dt_1\right)\int_{Y_0}|f|^2_{\tilde{h}_j}dV_{\Omega}[\psi]\\
\le &\frac{\sup_{\Omega}e^{-u}}{\inf_{\Omega}e^{-u}}\left (\frac{1}{\delta}c(T)e^{-T}+\int_{T}^{+\infty}c(t_1)e^{-t_1}dt_1\right)\int_{Y_0}|f|^2_{h}dV_{\Omega}[\psi]\\
= &C_{\Omega}\left (\frac{1}{\delta}c(T)e^{-T}+\int_{T}^{+\infty}c(t_1)e^{-t_1}dt_1\right)\int_{Y_0}|f|^2_{h}dV_{\Omega}[\psi]<+\infty,
\end{split}
\end{equation}
where $C_{\Omega}:=\frac{\sup_{\Omega}e^{-u}}{\inf_{\Omega}e^{-u}}$. Note that $\inf_{t\in (T,+\infty)}c(t)>0$ and $h_1\le h_j$ for any $j\ge 1$, by inequality \eqref{estimate in local l2 extenion 1}, we have
\begin{equation}\nonumber
\sup_{j\ge 1}\int_{\Omega} |F_j|^2_{h_1}dV_{\Omega}<+\infty.
\end{equation}
By Montel Theorem, we know that there exists a subsequence of $\{F_j\}_{j=1}^{+\infty}$ (also denoted by $F_j$) compactly convergent to a holomorphic section $F$ of   $K_{\Omega}\otimes E$ on $\Omega$. It follows from inequality \eqref{estimate in local l2 extenion 1} and Fatou's Lemma that, for any $i\ge 1$, we have
\begin{equation}\label{estimate in local l2 extenion 2}
\begin{split}
&\int_{\Omega} c(-\psi)|F|^2_{h_i}dV_{\Omega}\\
\le&\liminf_{j\to +\infty} \int_{\Omega} c(-\psi)|F_j|^2_{h_i}dV_{\Omega}\\
\le&\liminf_{j\to +\infty} \int_{\Omega} c(-\psi)|F_j|^2_{h_j}dV_{\Omega}\\
\le  &C_{\Omega}\left (\frac{1}{\delta}c(T)e^{-T}+\int_{T}^{+\infty}c(t_1)e^{-t_1}dt_1\right)\int_{Y_0}|f|^2_{h}dV_{\Omega}[\psi]<+\infty,
\end{split}
\end{equation}
Letting $i\to +\infty$ in inequality \eqref{estimate in local l2 extenion 2}, by monotone convergence theorem, we have
\begin{equation}\nonumber
\begin{split}
\int_{\Omega} c(-\psi)|F|^2_{h}dV_{\Omega} \le C_{\Omega}\left (\frac{1}{\delta}c(T)e^{-T}+\int_{T}^{+\infty}c(t_1)e^{-t_1}dt_1\right)\int_{Y_0}|f|^2_{h}dV_{\Omega}[\psi]<+\infty.
\end{split}
\end{equation}

Proposition \ref{key pro l2 extension local} is proved.
\end{proof}

The following lemma will be used in the proof of the main theorem.
\begin{Lemma}[see Theorem 4.4.2 in \cite{Hormander}] \label{hormander} Let $\Omega$ be a pseudoconvex domain in $\mathbb{C}^n$, and $\varphi$ be a plurisubharmonic function on $\Omega$. For any $w\in L^2_{p,q+1}(\Omega,e^{-\varphi})$ with $\bar{\partial}w=0$, there exists a solution $s\in L^2_{p,q}(\Omega,e^{-\varphi})$ of the equation $\bar{\partial}s=w$ such that
$$\int_{\Omega}\frac{|s|^2}{(1+|z|^2)^2}e^{-\varphi}d\lambda\le \int_{\Omega}|w|^2e^{-\varphi}d\lambda,$$
where $d\lambda$ is the Lebesgue measure on $\mathbb{C}^n$.
\end{Lemma}

We recall following results of positive definite hermitian matrices.

Let $\mathcal{M}:=\{M\in M_n(\mathbb{C}): M \text{\ is a positive definite hermitian matrix}\}$. Note that $M_n(\mathbb{C})$ is a $2n^2$-dimensional real manifold. Then $\mathcal{M}$ is an $n^2$-dimensional real sub-manifold of $M_n(\mathbb{C})$.
Denote $F:M_n(\mathbb{C})\to M_n(\mathbb{C})$ by $F(X)=X^2$ for any $X\in M_n(\mathbb{C})$. Denote $F|_\mathcal{M}: \mathcal{M}\to \mathcal{M}$. We have the following property of $F|_\mathcal{M}$.

\begin{Lemma}[see \cite{GMY-boundary5}]\label{sqrt of positive definite matrices}
$F|_\mathcal{M}: \mathcal{M}\to \mathcal{M}$ is a smooth diffeomorphism. Therefore for every positive definite hermitian matrix $h$, one can find positive definite hermitian matrix $C$ such that $h=C^2$ and $C$ depends smoothly on $h$ in $\mathcal{M}$.
\end{Lemma}

By Lemma \ref{sqrt of positive definite matrices}, we have following result.
\begin{Lemma}[see \cite{GMY-boundary5}]\label{metric consturcution}
Let $A$ and $B$ be two $n\times n$ positive definite hermitian matrices. Then there exists a unique matrix $C$ with positive eigenvalue such that $A=CB\overline{C}^T$ and $CB=B\overline{C}^T$. The matrix $C$ depends smoothly on $A$ and $B$ in $\mathcal{M}\times\mathcal{M}$.
\end{Lemma}

Let $X$ be an $n-$dimensional complex manifold and $\omega$ be a hermitian metric on $X$. Let $Q$ be a holomorphic vector bundle on $X$ with rank $r$. Let $h_1$ be a measurable metric on $Q$ and $h_2$ be a smooth hermitian metric on $Q$. Let $M$ be a relatively compact open subset of $X$.  Denote $\mathcal{H}_i:=L^2(M,K_X\otimes Q,h_i,dV_{\omega})$ for $i=1,2$. Denote $||g||_{\omega,h_i}$ be the norm of $g\in \mathcal{H}_i$. We recall the following lemma about weakly convergence.
Using Lemma \ref{metric consturcution}, we have
\begin{Lemma}\label{equiv of weak convergence}Assume that $h_2\le C' h_1$ for some constant $C'>0$ on $\overline{M}$. Let $\{f_k\}_{k\in\mathbb{Z}^+}$ be a sequence in $\mathcal{H}_1$ which is weakly converges to $0$ as $k\to +\infty$. Then the sequence $\{f_k\}$ belongs to $\mathcal{H}_2$ and  also weakly converges to $0$ in $\mathcal{H}_2$ as $k\to +\infty$.
\end{Lemma}
\begin{proof} Since $\{f_k\}$ weakly converges to $0$ as $k\to +\infty$ in $\mathcal{H}_1$, we know that $||f_k||_{\omega, h_1}$ is uniformly bounded with respect to $k$. As $h_2\le C' h_1$ for some constant $C'>0$ on $\overline{M}$, we know that $||f_k||_{\omega, h_2}$ is uniformly bounded with respect to $k$. Hence $f_k\in \mathcal{H}_2$ for any $k\ge 1$.

Let $V\Subset M$ be a small open set. Let $(e_1,\cdots,e_r)$ be a $h_2$-orthogonal frame of $K_X\otimes Q$ on $V$, i.e., under the frame $(e_1,\cdots,e_r)$ , we have
$$h_2=\sum_{i=1}^r e^*_i\otimes \bar{e_i}^*.$$

Denote $H_1$ be the matrix of $h_1$ under the frame $(e_1,\cdots,e_r)$. It follows from Lemma \ref{metric consturcution} that there exists a unique positive definite matrix of functions $C=(C_{p,q}(x))_{r\times r}$  such that $C_{p,q}(x)$ is measurable functions on $V$, $H_1=C\overline{C}^T$ and $C=\overline{C}^T$. If $s=\sum_{i=1}^r s_i\otimes e_i$ is any local section of $K_X\otimes Q$ on $V$, then we simply write $s=(s_1,\cdots,s_r)$. Denote $H(s)=(s_1,\cdots,s_r)C$ and
$H^{-1}(s)=(s_1,\cdots,s_r)C^{-1}$.

 Let $g$ be any compact supported  smooth section of  $K_X\otimes Q$ on $V$. Then $$\int_{V}<H^{-1}(g),H^{-1}(g)>_{\omega,h_1}dV_{\omega}
 =\int_{V}<g,g>_{\omega,h_2}dV_{\omega}<+\infty,$$
 which implies that  $H^{-1}(g)\in \mathcal{H}_1$. As $h_2\le C h_1$ for some constant $C'>0$ on $\overline{M}$, we know that $H^{-1}(g)\in \mathcal{H}_2$. Hence
 $$\int_{V}<H^{-1}\big(H^{-1}(g)\big),H^{-1}\big(H^{-1}(g)\big)>_{\omega,h_1}dV_{\omega}
 =\int_{V}<H^{-1}(g),H^{-1}(g)>_{\omega,h_2}dV_{\omega}<+\infty,$$
 which implies that  $H^{-1}\big(H^{-1}(g)\big)\in \mathcal{H}_1$.

Then for any $g\in C^{\infty}_{c}(V,K_X\otimes Q)$, note that $H^{-1}\big(H^{-1}(g)\big)\in \mathcal{H}_1$, we have
\begin{equation}\label{h2 convergence}
\begin{split}
   \lim_{k\to+\infty}\int_V<f_k,g>_{\omega,h_2}dV_{\omega}
   & = \lim_{k\to+\infty}\int_V <H^{-1}(f_k),H^{-1}(g)>_{h_1} dV_{\omega}\\
      & = \lim_{k\to+\infty}\int_V <f_k,H^{-1}\big(H^{-1}(g)\big)>_{h_1} dV_{\omega}\\
     &=0,
\end{split}
\end{equation}
where the last inequality holds since $\{f_k\}$ weakly converges to $0$ as $k\to +\infty$ in $\mathcal{H}_1$.

As $M$ is relatively compact in $X$, by partition of unity, we know that for any $\gamma\in C^{\infty}_{c}(M,K_X\otimes Q)$, we have
\begin{equation}\label{h2 convergence2}
\begin{split}
   \lim_{k\to+\infty}\int_{M}<f_k,\gamma>_{\omega,h_2}dV_{\omega} =0.
\end{split}
\end{equation}
Then for any $\eta\in \mathcal{H}_2$, we can find a sequence of $\gamma_l\in C^{\infty}_{c}(M,K_X\otimes Q)$ such that $\lim_{l\to+\infty}||\gamma_l-\eta||_{\omega,h_2}=0$. Hence we have
\begin{equation}\label{h2 convergence3}
\begin{split}
   &\lim_{k\to+\infty}|\int_{M}<f_k,\eta>_{\omega,h_2}dV_{\omega}|\\
   \le & \lim_{k\to+\infty}\bigg(|\int_{M}<f_k,\gamma_l>_{\omega,h_2}dV_{\omega}|
   +|\int_{M}<f_k,\eta-\gamma_l>_{\omega,h_2}dV_{\omega}|\bigg)\\
   \le & 0+\big(\sup_{k}||f_k||_{\omega,h_2}\big)||\eta-\gamma_l||_{\omega,h_2}.
\end{split}
\end{equation}
Note that $||f_k||_{\omega, h_2}$ is uniformly bounded with respect to $k$. It follows from inequality \eqref{h2 convergence3} that
\begin{equation}
\begin{split}
   \lim_{k\to+\infty}|\int_{M}<f_k,\eta>_{\omega,h_2}dV_{\omega}|=0,
\end{split}
\end{equation}
which means that $\{f_k\}$  also weakly converges to $0$ in $\mathcal{H}_2$ as $k\to +\infty$.
\end{proof}

Using Lemma \ref{sqrt of positive definite matrices}, we have the following lemma.

\begin{Lemma}\label{existence of bounded tra} Let $\Omega\subset \mathbb{C}^n$ be an open subset with coordinate $z$. Let $E:=\Omega\times \mathbb{C}^r$, where $r$ is a positive integer. Let $\{e_i\}_{i=1}^r$ be a smooth frame of $E$ on $\Omega$. Let $h$ be a measurable metric on $E$ such that $h_e\ge C I_r$ under the frame $\{e_i\}_{i=1}^r$, where $C>1$ is a constant and $I_r$ is the standard metric on $E$ under the frame $\{e_i\}_{i=1}^r$. Then we can find a measurable frame $\{w_i\}_{i=1}^r$ of $E$ such that $h_w=(\det{h}) I_r$, where $h_w$ is the representation of $h$ under the frame $\{w_i\}_{i=1}^r$ and $(w_1,\cdots,w_r)=(e_1,\cdots,e_r)B^{-1}$. Moreover, each element $b_{i,j}(z)$ of $B$ is a bounded function on $\Omega$.
\end{Lemma}
\begin{proof}
It follows from Lemma \ref{sqrt of positive definite matrices} that there exist positive definite hermitian matrixes $C(z):= \big( C_{i,j}(z)\big)_{r\times r}$ such that $h_e=C^2$ and the elements $C_{i,j}(z)$ of $C$ are measurable functions on $\Omega$. Note that $C=C^*$, where $C^*=\bar{C}^T$. We know that $\sqrt{\det h_e}C^{-1}h_e{(C^*)}^{-1}\sqrt{\det h_e}=(\det{h}_e) I_r$.
Denote $$B:=\frac{C^T}{\sqrt{\det h_e}}.$$
Then we know that $(B^{-1})^Th_e \overline{B^{-1}}=(\det{h}_e) I_r$.

 Now we prove that each element $b_{i,j}(z)$ of $B$ is a bounded function on $\Omega$.
It follows from $h_e\ge C I_r$ for some $C>1$ that we know that every eigenvalue $\lambda_i$ ($i=1 \ldots r$) of $h_e$ is bigger than 1. As $h_e=C^2$, we know that the eigenvalues of $C$ are $\{\sqrt{\lambda_i}\}_{i=1}^r$ and $\det C=\sqrt{\det h_e }$. Hence we know that $\sqrt{\lambda_i}\le \sqrt{\det h_e }=\det C $ and $B\le I_r$. It follows from $B$ is bounded above and $C^T$ is also a positive definite hermitian matrix that we know each element $b_{i,j}(z)$ of $B$ is a bounded function on $\Omega$.

Lemma \ref{existence of bounded tra} has been proved.
\end{proof}

We recall the following regularization result of quasi-plurisubharmonic functions which will be used in the proof of main theorem.

\begin{Lemma}
[Theorem 6.1 in \cite{DemaillyReg}, see also Theorem 2.2 in \cite{ZZ2019}]
\label{regularization on cpx mfld}
Let ($M,\omega$) be a complex manifold equipped with a hermitian metric
$\omega$, and $\Omega \subset \subset M $ be an open set. Assume that
$T=\widetilde{T}+\frac{\sqrt{-1}}{\pi}\partial\bar{\partial}\varphi$ is a closed
(1,1)-current on $M$, where $\widetilde{T}$ is a smooth real (1,1)-form and
$\varphi$ is a quasi-plurisubharmonic function. Let $\gamma$ be a continuous real
(1,1)-form such that $T \ge \gamma$. Suppose that the Chern curvature tensor of
$TM$ satisfies
\begin{equation}\nonumber
\begin{split}
(\sqrt{-1}&\Theta_{TM}+\varpi \otimes Id_{TM})(\kappa_1 \otimes \kappa_2,\kappa_1
\otimes \kappa_2)\ge 0 \\
&\forall \kappa_1,\kappa_2 \in TM \quad with \quad \langle \kappa_1,\kappa_2
\rangle=0
\end{split}
\end{equation}
for some continuous nonnegative (1,1)-form $\varpi$ on $M$. Then there is a family
of closed (1,1)-currents
$T_{\zeta,\rho}=\widetilde{T}+\frac{\sqrt{-1}}{\pi}\partial\bar{\partial}
\varphi_{\zeta,\rho}$ on M ($\zeta \in (0,+\infty)$ and $\rho \in (0,\rho_1)$ for
some positive number $\rho_1$) independent of $\gamma$, such that
\par
$(i)\ \varphi_{\zeta,\rho}$ is quasi-plurisubharmonic on a neighborhood of
$\bar{\Omega}$, smooth on $M\backslash E_{\zeta}(T)$,
\\
increasing with respect to
$\zeta$ and $\rho$ on $\Omega $ and converges to $\varphi$ on $\Omega$ as $\rho
\to 0$,
\par
$(ii)\ T_{\zeta,\rho}\ge\gamma-\zeta\varpi-\delta_{\rho}\omega$ on $\Omega$,\\
where $E_{\zeta}(T):=\{x\in M:v(T,x)\ge \zeta\}$ ($\zeta>0$) is the $\zeta$-upper level set of
Lelong numbers and $\{\delta_{\rho}\}$ is an increasing family of positive numbers
such that $\lim\limits_{\rho \to 0}\delta_{\rho}=0$.
\end{Lemma}
\begin{Remark}[see Remark 2.1 in \cite{ZZ2019}]
Lemma \ref{regularization on cpx mfld} is stated in \cite{DemaillyReg} in the case $M$ is a compact complex manifold. The similar proof as in \cite{DemaillyReg} shows that Lemma \ref{regularization on cpx mfld} on noncompact complex manifold still holds where the uniform estimates $(i)$ and $(ii)$ are obtained only on a relatively compact subset $\Omega$.
\end{Remark}

Let $c(t)$  belongs to class $\mathcal{G}_{T,\delta}$.
Recall that
$$s(t):=\frac{\int^t_T\bigg(\frac{1}{\delta}c(T)e^{-T}+\int^{t_2}_T c(t_1)e^{-t_1}dt_1\bigg)dt_2+\frac{1}{\delta^2}c(T)e^{-T}}{\frac{1}{\delta}c(T)e^{-T}+\int^t_T
c(t_1)e^{-t_1}dt_1}.$$
We have following regularization lemma for $c(t)$.
\begin{Lemma}\label{approx of ct}
Let $c(t)\in \mathcal{G}_{T,\delta}$. Let $\{\beta_m<1\}$ be a sequence of positive real numbers such that $\beta_m$ decreasingly converges to $0$ as $m\to +\infty$. Then there exists a sequence of positive functions $c_m(t)$ on $[T,+\infty)$, which satisfies:\\
(1) $c_m(t)\in \mathcal{G}_{T,\delta}$;\\
(2) $c_m(t)e^{-t}$ is decreasing with respect to $t$ near $+\infty$;\\
(3) $c_m(t)$ is smooth on $[T+4\beta_m,+\infty)$;\\
(4) $c_m(t)$ are uniformly convergent to $c(t)$ on any compact subset of $(T,+\infty)$;\\
(5) $\frac{1}{\delta}c_m(T)e^{-T}+\int_{T}^{+\infty}c_m(t)e^{-t}dt$ converges to $\frac{1}{\delta}c(T)e^{-T}+\int_{T}^{+\infty}c(t)e^{-t}dt<+\infty$ as $m\to +\infty$;\\
(6) For each $m$, there exists $\kappa_m>0$ such that
$$S_m(t):=\frac{\int^t_T\bigg(\frac{1}{\delta}c_m(T)e^{-T}+\int^{t_2}_T c_m(t_1)e^{-t_1}dt_1\bigg)dt_2+\frac{1}{\delta^2}c_m(T)e^{-T}+\kappa_m}{\frac{1}{\delta}c_m(T)e^{-T}+\int^t_T
c_m(t_1)e^{-t_1}dt_1}>s(t),$$
for any $t\ge T$ and $S'_m(t)>0$ on $[T+\beta_m,+\infty)$.
\end{Lemma}
\begin{proof} The following constructions  of $c_m(t)$ was inspired by Lemma 4.8 in \cite{guan-zhou13ap}.

By direct calculation, we have
\begin{equation}\nonumber
s'(t)=1-\frac{c(t)e^{-t}
(\int_{T}^{t}(\frac{1}{\delta}c(T)e^{-T}+\int_{T}^{t_{2}}c(t_{1})e^{-t_{1}}dt_{1})
dt_{2}+\frac{1}{\delta^{2}}c(T)e^{-T})}
{(\frac{1}{\delta}c(T)e^{-T}+\int_{T}^{t}c(t_{1})e^{-t_{1}}dt_{1})^{2}}.
\end{equation}
It follows from inequality \eqref{integarl condition 1} that $s'(t)>0$ for any $t\ge T$.
Hence, for any $\epsilon, N>0$, we can choose suitable constant $\kappa_{\epsilon,N}>0$ such that
$$\mathcal{S}_{\epsilon,N}(t):=\frac{\int^t_T\bigg(\frac{1}{\delta}c(T)e^{-T}+\int^{t_2}_T c(t_1)e^{-t_1}dt_1\bigg)dt_2+\frac{1}{\delta^2}c(T)e^{-T}+\kappa_{\epsilon,N}}{\frac{1}{\delta}c(T)e^{-T}+\int^t_T
c(t_1)e^{-t_1}dt_1}$$
satisfies $\mathcal{S}'_{\epsilon,N}(t)>0$ on $[T+\epsilon,T+N]$.
For the convenience of notation, we denote
$$G(t):=\frac{1}{\delta}c(T)e^{-T}+\int_{T}^{t}c(t_1)e^{-t_1}dt_1.$$

  As $\int_{T}^{+\infty}c(t)e^{-t}<+\infty$, there exists a sequence of real number $\{B_m\}_{m\in\mathbb{Z}^+}$ such that $B_m$ increasingly converges to $+\infty$ as $m\to+\infty$ and  $\int_{T+B_m}^{+\infty}c(t_1)e^{-t_1}dt_1<\frac{1}{m}$. Denote $g_{m}(t)=c(t)$ when $t\in [T,T+B_m)$ and $g_{m}(t)$  is a smooth decreasing function with respect to $t$ on $[T+B_m,+\infty)$ such that $g_{m}(T+B_m)=c(T+B_m)$. Denote $G_m(t):=\frac{1}{\delta}c(T)e^{-T}+\int_{T}^{t}g_m(t_1)e^{-t_1}dt_1$. We will determine the value of $g_m(t)$ on $[T+B_m,+\infty)$ in the following discussion.

  Let
$\epsilon=\beta_m$ and $N=B_m+2$, then we can choose $\kappa_m>0$ such that
$$\mathcal{S}_{m}(t):=\frac{\int^t_T G(t_2)dt_2+\frac{1}{\delta^2}c(T)e^{-T}+\kappa_{m}}{G(t)}$$
satisfies $\mathcal{S}'_{m}(t)>0$ on $[T+\beta_m,T+B_m+2]$. Denote
$$\hat{S}_{m}(t):=\frac{\int^t_T G_m(t_2)dt_2+\frac{1}{\delta^2}c(T)e^{-T}+\kappa_{m}}{G_m(t)}.$$
Note that $\mathcal{S}_{m}(t)=\hat{S}_{m}(t)$ for any $t\in[T,T+B_m]$.
By direct calculation, we have
$$\mathcal{S}'_{m}(t)=1-\frac{G'(t)
\bigg(\int_{T}^{t}G(t_2)dt_2+\frac{1}{\delta^2}c(T)e^{-T}+\kappa_{m}\bigg)}{G^2(t)},$$
and
$$\hat{S}'_{m}(t)=1-\frac{G'_m(t)
\bigg(\int_{T}^{t}G_m(t_2)dt_2+\frac{1}{\delta^2}c(T)e^{-T}+\kappa_{m}\bigg)}{G^2_m(t)}.$$
Hence $\hat{S}'_{m}(t)>0$ on $[T+B_m,+\infty)$ if and only if
\begin{equation}\label{def of It}
I_m(t):=\bigg(G_{m}(t)\bigg)^2
-G_{m}'(t)\bigg(\int_{T}^{t}G_{m}(t_2)dt_2+\frac{1}{\delta^2}c(T)e^{-T}+\kappa_m\bigg)>0
\end{equation}
 holds on  $[T+B_m,+\infty)$. For any $t>T+{B_m}$, direct calculation shows
\begin{equation}\label{der of GBt}
 I'_m(t)=G_{m}(t)G'_{m}(t)-G_{m}^{''}(t)\int_{T}^{t}G_{m}(t_2)dt_2
 -G_{m}^{''}(t)\frac{1}{\delta^2}c(T)e^{-T}-G_{m}^{''}(t)\kappa_m.
\end{equation}
 By direct calculation,  we have
\begin{equation}\label{derivative of GBt}
G_{m}''(t)=g'_{m}(t)e^{-t}-g_{m}(t)e^{-t}.
\end{equation}
Note that $g'_{m}(t)<0$ on $[T+B_m,+\infty)$ and $g_{m}(t)$ is positive, then
we know that $G_{m}''(t)< 0$ on $[T+B_m,+\infty)$. It follows from $G_{m}(t)$, $G'_{m}(t)$, $\frac{1}{\delta^2}c(T)e^{-T}$, $\kappa_m$ are all positive and $G_{m}''(t)< 0$ on $[T+B_m,+\infty)$ that
$I'_m(t)>0$ for $t\ge T+B_m$.
Note that $\mathcal{S}'_{m}(T+B_m)=\hat{S}'_{m}(T+B_m)$ and $\mathcal{S}'_{m}(t)>0$ on $[T+\beta_m,T+B_m+2]$.
We know $I_m(T+B_m)>0$. Then for any $t\ge T+B_m$, $I_m(t)>0$. We denote
$$L_m:=\inf_{t\in [T+\beta_m,+\infty)}I_m(t)=\inf_{t\in [T+\beta_m,T+B_m]} I_m(t),$$
and we note that $L_m$ is a positive number for fixed $m$.

Note that $G(t)>0$ is continuous and increasing on $[T,+\infty)$ and denote $G(+\infty):=\lim_{t\to+\infty}G(t)<+\infty$.  As $\mathcal{S}'_{m}(t)>0$ on $[T+\beta_m,T+B_m+2]$, there exists $\alpha>0$ such that $\mathcal{S}_{m}(t)<\mathcal{S}_{m}(T+B_m)+\frac{\kappa_m}{2G(+\infty)}$ for any $t\in[T+B_m,T+B_m+\alpha]$. Then we can choose $g_m(t)$ decreasing so fast on $[T+B_m,T+B_m+\alpha]$ such that
 \begin{equation}\nonumber
 \begin{split}
    &\int_{T+B_m}^{T+B_m+\alpha}g_m(t)e^{-t}dt<\min\{\frac{1}{2m},\int_{T+B_m}^{T+B_m+\alpha}c(t)e^{-t}dt\}, \\
      & \int_{T+B_m}^{T+B_m+\alpha}G_m(t_2)dt_2<\int_{T+B_m}^{T+B_m+\alpha}G(t_2)dt_2,\\
      &g(T+B_m+\alpha)<c(T+B_m+\alpha).
 \end{split}
 \end{equation}
  As $\hat{S}'_{m}(t)>0$ on $[T+B_m,T+B_m+\alpha]$ and $\hat{S}_m(T+B_m)=\mathcal{S}_m(T+B_m)$, we know that $\hat{S}_m(t)>\hat{S}_m(T+B_m)=\mathcal{S}_{m}(T+B_m)>\mathcal{S}_{m}(t)-\frac{\kappa_m}{2G(+\infty)}>s(t)$ on $[T+B_m,T+B_m+\alpha]$.

By direct calculation, we have
\begin{equation}\label{minus derivative}
 \begin{split}
    \mathcal{S}'_{m}(t)-\hat{S}'_{m}(t)  =&\frac{G'_m(t)
\bigg(\int_{T}^{t}G_m(t_2)dt_2+\frac{1}{\delta^2}c(T)e^{-T}+\kappa_{m}\bigg)}{G^2_m(t)} \\
      & -\frac{G'(t)
\bigg(\int_{T}^{t}G(t_2)dt_2+\frac{1}{\delta^2}c(T)e^{-T}+\kappa_{m}\bigg)}{G^2(t)}.
 \end{split}
\end{equation}
Since $\lim_{t\to+\infty}G(t)<+\infty$ and $\int_{B_m}^{+\infty}c(t_1)e^{-t_1}dt_1<\frac{1}{m}$, we can also choose $g_m(t)<c(t)$ decreasing so fast on $(T+B_m+\alpha,+\infty)$ such that
$$\int_{T+B_m+\alpha}^{+\infty}g_m(t)e^{-t}dt<\frac{1}{2m},$$
and
$$\frac{G'_m(t)}{G^2_m(t)}-\frac{G'(t)}{G^2(t)}<0$$ holds on $(T+B_m+\alpha,+\infty)$. We also note that $\int_{T+B_m}^{T+B_m+\alpha}G_m(t_2)dt_2<\int_{T+B_m}^{T+B_m+\alpha}G(t_2)dt_2$, $\int_{T+B_m}^{T+B_m+\alpha}g_m(t)e^{-t}dt<\int_{T+B_m}^{T+B_m+\alpha}c(t)e^{-t}dt$ and $G_{m}(t)<G(t)$ on $(T+B_m+\alpha,+\infty)$. Hence it follows from above discussion and equality \eqref{minus derivative} that
\begin{equation}
 \begin{split}
    \mathcal{S}'_{m}(t)-\hat{S}'_{m}(t) <0
 \end{split}
\end{equation}
holds on $(T+B_m+\alpha,+\infty)$. It follows from  $\hat{S}_{m}(t)>\mathcal{S}_{m}(t)-\frac{\kappa_m}{2G(+\infty)}$ on $[T+B_m,T+B_m+\alpha]$ and $ \hat{S}'_{m}(t)>\mathcal{S}'_{m}(t)$ holds on $(T+B_m+\alpha,+\infty)$ that we have $\hat{S}_{m}(t)>\mathcal{S}_m(t)-\frac{\kappa_m}{2G(+\infty)}>s(t)$ on $[T+B_m+\alpha,+\infty)$. Then we know that $\hat{S}_{m}(t)>s(t)$ on $[T,+\infty)$.
We also note that we have $\hat{S}'_{m}(t)>0$ on $[T+\beta_m,+\infty)$.

By the construction of $g_m(t)$, we have
\begin{equation}\label{integral of gb near infty case I}
  \int_{T+B_m}^{+\infty}g_{m}(t)e^{-t}dt<\frac{1}{m}.
\end{equation}
 Hence it is easy to see that we have $$\lim_{m\to+\infty}\int_{T}^{+\infty}g_{m}(t)e^{-t}dt=\int_{T}^{+\infty}c(t)e^{-t}dt,$$
 and  $\frac{1}{\delta}g_{m}(T)e^{-T}+\int_{T}^{+\infty}g_{m}(t)e^{-t}dt$ converges to $\frac{1}{\delta}c(T)e^{-T}+\int_{T}^{+\infty}c(t)e^{-t}dt<+\infty$ as $m\to +\infty$.

Now we have constructed a sequence of function $g_m(t)$ on $[T,+\infty)$ such that\\
(1) $g_m(t)$ is continuous on $[T,+\infty)$ and smooth on $(T+B_m,+\infty)$;\\
(2) $g_m(t)=c(t)$ on $[T,T+B_m]$ and $g_m(t)e^{-t}$ is decreasing with respect to $t$ on $[T+B_m,+\infty)$;\\
(3) $\lim_{m\to+\infty}\int_{T}^{+\infty}g_m(t)e^{-t}dt=\int_{T}^{+\infty}c(t)e^{-t}dt$;\\
(4) The corresponding  $\hat{S}_{m}(t)$
satisfies  $\hat{S}_{m}(t)>\mathcal{S}_m(t)-\frac{\kappa_m}{2G(+\infty)}$ on $[T,+\infty]$ and $\hat{S}'_{m}(t)>0$ on $[T+\beta_m,+\infty)$, i.e., $\hat{S}_{m}(t)$ is increasing with respect to $t\in [T+\beta_m,+\infty)$. Note that $G(t)<G(+\infty)$ for any  $t\in[T,+\infty]$, we know that there exits a positive number $\tau_m>0$ such that $\hat{S}_{m}(t)>s(t)+\tau_m$ on $[T,+\infty)$.

\emph{Next we will use convolution to regularize each $c_m(t)$ on $[T+\beta_m,T+B_m]$.}

Without loss of generality, we may assume that $B_m\ge 4$ for any $m\ge 1$. From now on, we fix some $m\ge 1$ and hence $B_m$ is fixed. We note that $[T+\beta_m,T+B_m+3]\Subset (T,+\infty)$. Let $0\le \chi\le 1$ be a cut-off function which is smooth on $\mathbb{R}$, $\chi(t)\equiv 1$ on $[T+3\beta_m,T+B_m+1]$ and $\text{supp}\chi(t)\Subset [T+2\beta_m,T+B_m+2]$. Then we have
$$g_{m}(t)e^{-t}=\chi(t)g_{m}(t)e^{-t}+\big(1-\chi(t)\big)g_{m}(t)e^{-t},$$
and denote $\Gamma_{m}(t)=\chi(t)g_{m}(t)e^{-t}$.

Let $\rho_{\epsilon}(y)$ be convolution kernel function such that $\text{supp}\rho_{\epsilon}(y)\Subset [-\epsilon,+\epsilon]$ and $\int_{\mathbb{R}}\rho_{\epsilon}(y)dy=1$ for any positive real number $\epsilon<1$. Denote
\begin{equation}\label{definition of gamma delta}
  \begin{split}
     \Gamma_{m,\epsilon}(t)& =\int_{\mathbb{R}}\Gamma_{m}(y)\rho_{\epsilon}(t-y)dy \\
       & =\int_{\mathbb{R}}\chi(y)g_{m}(y)e^{-y}\rho_{\epsilon}(t-y)dy\\
       &=\int_{\mathbb{R}}\chi(t-y)g_{m}(t-y)e^{-t+y}\rho_{\epsilon}(y)dy.
  \end{split}
\end{equation}
Denote $g_{m,\epsilon}(t):=\Gamma_{m,\epsilon}(t)e^t+\big(1-\chi(t)\big)g_{m}(t)$.  As $g_m(T)$ is continuous  on $[T,T+m]$, we know that $g_{m,\epsilon}(t)$ is continuous on $[T,+\infty)$ and smooth on $[T+4\beta_m,+\infty)$. Then
\begin{equation}\label{gb minus gbdelta}
  \begin{split}
  g_{m}(t)-g_{m,\epsilon}(t)&=\Gamma_{m,\epsilon}(t)e^t-\chi(t)g_{B_m}(t)\\
       &=\int_{\mathbb{R}}\chi(t-y)g_{m}(t-y)e^y\rho_{\epsilon}(y)dy-\chi(t)g_{m}(t)\\
       &=\int_{\mathbb{R}}\bigg(\chi(t-y)g_{m}(t-y)e^y-\chi(t)g_{m}(t)\bigg)\rho_{\epsilon}(y)dy.
  \end{split}
\end{equation}
When $t\in [T,+\infty)\backslash [T+\beta_m,T+B_m+3]$, by above formula, it is easy to see that $g_{m}(t)=g_{m,\epsilon}(t)$.  Note that $\chi(t)g_{m}(t)$ is a uniformly continuous function on $\mathbb{R}$ and $e^y$ is continuous near $0$. We can take $\epsilon_{m}$ small enough such that when $\epsilon<\epsilon_m$,
$$|g_{m}(t)-g_{m,\epsilon}(t)|<\tau$$
holds for any $t\in [T+\beta_m,T+B_m+3]$ and any given small $\tau>0$.
Hence when $\epsilon<\epsilon_m$, we have
\begin{equation}\label{result 1}
 \max_{t\in[T,+\infty)}|g_{m}(t)-g_{m,\epsilon}(t)|< \tau,
\end{equation}
which implies that $g_{B_m,\epsilon}(t)$ uniformly converges to $g_{B_m}(t)$ on $[T,+\infty)$ as $\epsilon\to 0$. Hence
$\int_{T}^{+\infty}g_{m,\epsilon}(t)e^{-t}dt$ uniformly converges to $\int_{T}^{+\infty}g_{m}(t)e^{-t}dt$  as $\epsilon\to 0$.

For any $\epsilon$, we define
\begin{equation}\label{def of Gbdelta}
\begin{split}
G_{m,\epsilon}(t):&=\frac{1}{\delta}c(T)e^{-T}+\int_{T}^{t}g_{m,\epsilon}(t_1)e^{-t_1}dt\\
&=\frac{1}{\delta}c(T)e^{-T}+\int_{T}^{t}\bigg(\Gamma_{m,\epsilon}(t_1)e^{t_1}+\big(1-\chi(t_1)\big)g_{m}(t_1)\big)\bigg)e^{-t_1}dt\\
&=\frac{1}{\delta}c(T)e^{-T}+
\int_{T}^{t}\Gamma_{m,\epsilon}(t_1)dt_1+
\int_{T}^{t}\bigg(\big(1-\chi(t_1)\big)g_{m}(t_1)\bigg)e^{-t_1}dt.
\end{split}
\end{equation}

Recall that
\begin{equation}\label{def of Gbm}
\begin{split}
G_{m}(t)&=\frac{1}{\delta}c(T)e^{-T}+\int_{T}^{t}g_{m}(t_1)e^{-t_1}dt_1\\
&=\frac{1}{\delta}c(T)e^{-T}+\int_{T}^{t}\chi(t_1)g_{m}(t_1)e^{-t_1}dt
+
\int_{T}^{t}\bigg(\big(1-\chi(t_1)\big)g_{m}(t_1)\bigg)e^{-t_1}dt
 \end{split}
\end{equation}
Hence we have
\begin{equation}\label{gm minus gmepsilon}
\begin{split}
G_{m}(t)-G_{m,\epsilon}(t)
&=\int_{T}^{t}\chi(t_1)g_{m}(t_1)e^{-t_1}dt_1-\int_{T}^{t}\Gamma_{m,\epsilon}(t_1)dt_1\\
&=\int_{T}^{t}\int_{\mathbb{R}}\chi(t_1)g_{m}(t_1)e^{-t_1}dt_1\rho_{\epsilon}(y)dt_1dy\\
&-\int_{T}^{t}\int_{\mathbb{R}}\chi(t_1-y)g_{m}(t_1-y)e^{-t_1+y}\rho_{\epsilon}(y)dydt_1\\
 \end{split}
\end{equation}
It follows from equality \eqref{gm minus gmepsilon},
$\text{supp}\chi(t)\Subset [T+2\beta_m,T+B_m+2]$ and property of convolution that we know
when $t\in [T,+\infty)\backslash [T+\beta_m,T+B_m+3]$, $ G_{m}(t)=G_{B_m,\epsilon}(t)$
and $G_{m,\epsilon}(t)$ uniformly converges to $G_{m}(t)$ on $[T+\beta_m,T+B_m+3]$ as $\epsilon \to 0$. Hence $G_{m,\epsilon}(t)$ uniformly converges to $G_{m}(t)$ on $[T,+\infty)$ as $\epsilon \to 0$.

It follows from definitions \eqref{def of Gbdelta} and \eqref{def of Gbm} that
$$G'_{m,\epsilon}(t)=g_{B_m,\epsilon}(t_1)e^{-t_1},$$
 and $$G'_{m}(t)=g_{m}(t_1)e^{-t_1}.$$
It follows from $g_{m,\epsilon}(t)$ uniformly converges to $g_{m}(t)$ on $[T,+\infty)$ as $\epsilon\to 0$ that we know $G'_{B_m,\epsilon}(t)$ uniformly converges to $G'_{B_m}(t)$ on $[T,+\infty)$ as $\epsilon\to 0$

It follows from $ G_{m}(t)=G_{m,\epsilon}(t)$ when $t\in [T,+\infty)\backslash [T+\beta_m,T+B_m+3]$
and $G_{m,\epsilon}(t)$ uniformly converges to $G_{m}(t)$ on $[T+\beta_m,T+B_m+3]$ as $\epsilon \to 0$. We know that
$\int_T^{t}G_{m,\epsilon}(t_1)dt_1$ uniformly converges to $\int_T^{t}G_{m}(t_1)dt_1$ for any $t\in[T,+\infty)$ as $\epsilon\to 0$.

Denote
$$I_{m,\epsilon}(t):=\big(G_{m,\epsilon}(t)\big)^2
-G'_{m,\epsilon}(t)\bigg(\int_{T}^{t}G_{m,\epsilon}(t_1)dt_1+\frac{1}{\delta^2}c(T)e^{-T}+\kappa_m\bigg).$$
Note that we have proved:\\
(I) $G_{m,\epsilon}(t)$ uniformly converges to $G_{B_m}(t)$ on $[T,+\infty)$ as $\epsilon \to 0$;\\
(II) $G'_{m,\epsilon}(t)$ uniformly converges to $G'_{B_m}(t)$ on $[T,+\infty)$ as $\epsilon\to 0$;\\
(III) $\int_T^{t}G_{m,\epsilon}(t_1)dt_1$ uniformly converges to $\int_T^{t}G_{m}(t_1)dt_1$ for any $t\in[T,+\infty)$ as $\epsilon\to 0$.\\
Hence we have $I_{m,\epsilon}(t)=I_{m}(t)$ on $[T,T+\beta_m]$ and
$I_{m,\epsilon}(t)$ uniformly converges to $I_{m}(t)$ on $[T+\beta_m,+\infty)$ as $\epsilon \to 0$.
As $L_m=\inf_{t\in [T+\beta_m,+\infty)}I_m(t)=\inf_{t\in [T+\beta_m,T+B_m+\alpha]} I_m(t)>0$, we can choose $\hat{\epsilon}_m$ small such that for any $\epsilon<\hat{\epsilon}_m$, $I_{m,\epsilon}(t)>0$ for  any $t\ge T$. Hence we know that if $\epsilon<\hat{\epsilon}$,
$$S_{m,\epsilon}(t):=\frac{\int^t_T G_{m,\epsilon}(t_2)dt_2+\frac{1}{\delta^2}c(T)e^{-T}+\kappa_{m}}{G_{m,\epsilon}(t)}$$
is increasing on $[T+\beta_m,+\infty)$ \big(or equivalently, $S'_{m,\epsilon}(t)>0$ on $[T+\beta_m,+\infty)$\big). From (I) and (III), we also can choose $\epsilon_m$ small such that when $\epsilon<\epsilon_m$,
$$|S_{m,\epsilon}(t)-S_{m}(t)|<\frac{\tau_m}{4},$$
which implies that $S_{m,\epsilon}(t)>S(t)$ on $[T,+\infty)$.
Note that $g_{m}(t)$ are uniformly convergent to $c(t)$ on any compact subset of $(T,+\infty)$ .
It follows from inequality \eqref{result 1} that we can choose $\epsilon_m$ small enough
$g_{m,\epsilon_m}(t)$ are uniformly convergent to $c(t)$ on any compact subset of $(T,+\infty)$.

Denote $c_m(t):=g_{m,\epsilon_m}(t)$. It is easy to see that $c_m(t)$ satisfies all the condition in Lemma \ref{approx of ct}.

\end{proof}

We recall the following lemma which can be referred to \cite{guan-zhou13ap}

\begin{Lemma}[see Lemma 4.11 in \cite{guan-zhou13ap}]
\label{class GT} Let $c(t)\in \mathcal{G}_{T}$. For any $T_1>T$, there exists $T_2$  and $\delta_2>0$, such that $T<T_2<T_1$ and there exists $c_{T_2}(t)\in \mathcal{G}_{T_2,\delta_2}$ satisfying\\
(1) $c_{T_2}(t)=c(t)|_{[T_1,+\infty)}$;\\
(2) $$\frac{1}{\delta_2}c_{T_2}(T_2)e^{-T_2}+\int_{T_2}^{+\infty}c_{T_2}(t_1)e^{-t_1}dt_1
=\int_{T}^{+\infty}c(t_1)e^{-t_1}dt_1.$$
\end{Lemma}

The following Lemma will be used in the proof of the theorem \ref{main result}.
\begin{Lemma}[see \cite{GY-concavity1}]
	\label{l:converge}
	Let $M$ be a complex manifold. Let $S$ be an analytic subset of $M$.  	
	Let $\{g_j\}_{j=1,2,...}$ be a sequence of nonnegative Lebesgue measurable functions on $M$, which satisfies that $g_j$ are almost everywhere convergent to $g$ on  $M$ when $j\rightarrow+\infty$,  where $g$ is a nonnegative Lebesgue measurable function on $M$. Assume that for any compact subset $K$ of $M\backslash S$, there exist $s_K\in(0,+\infty)$ and $C_K\in(0,+\infty)$ such that
	$$\int_{K}{g_j}^{-s_K}dV_M\leq C_K$$
	 for any $j$, where $dV_M$ is a continuous volume form on $M$.
	
 Let $\{F_j\}_{j=1,2,...}$ be a sequence of holomorphic $(n,0)$ form on $M$. Assume that $\liminf_{j\rightarrow+\infty}\int_{M}|F_j|^2g_j\leq C$, where $C$ is a positive constant. Then there exists a subsequence $\{F_{j_l}\}_{l=1,2,...}$, which satisfies that $\{F_{j_l}\}$ is uniformly convergent to a holomorphic $(n,0)$ form $F$ on $M$ on any compact subset of $M$ when $l\rightarrow+\infty$, such that
 $$\int_{M}|F|^2g\leq C.$$
\end{Lemma}

\subsection{Concavity property of minimal $L^2$ integrals}\label{minimal l2 integrals}

Let $(M,\omega)$ be a weakly pseudoconvex K\"ahler manifold. Let $\psi<0$ be a plurisubharmonic function on $X$ with neat analytic singularities. Let $Y:=V(\mathcal{I}(\psi))$ and assume that $\psi$ has log canonical singularities along $Y$. Let $\varphi$ be a Lebesgue measurable function on $M$ such that $\varphi+\psi$ is a plurisubharmonic function on $M$. Let $E$ be a holomorphic vector bundle on $M$. Let $(M,E,\Sigma,M_k,h,h_{k,s})$ be a singular metric on $E$.
Assume that  $\Theta_{h}(E)\ge^s_{Nak} 0$ on $M$ in the sense of Definition \ref{singular nak} and $he^{-\varphi}$ is locally lower bounded.

We firstly recall the following property of singular metric on $L:=M\times C$.
\begin{Proposition}[see Remark 9.13 in \cite{GMY-boundary5}]\label{psh fun is singular metric}
Let $M$ be a weakly pseudoconvex K\"ahler manifold. Let $\varphi$ be a plurisubharmonic function on $M$. Then $h:=e^{-\varphi}$ is a singular metric on $E:=M\times \mathbb{C}$ in the sense of Definition \ref{singular metric} and $h$ satisfies $\Theta_h(E)\ge^s_{Nak} 0$ in the sense of Definition \ref{singular nak}.
\end{Proposition}

It follows from Proposition \ref{psh fun is singular metric}, $\varphi+\psi$ is a plurisubharmonic function on $M$ and $\Theta_{h}(E)\ge^s_{Nak} 0$ on $M$  that
$\tilde{h}:=he^{-\varphi-\psi}$ is singular Nakano semi-positive on $M$ in the sense of Definition \ref{singular nak}.

Let $c(t)\in \mathcal{G}_{0}$ such that $c(t)e^{-t}$ is decreasing with respect to $t\in(0,+\infty)$.
Let $f \in H^0(Y^0,(K_M\otimes E)|_{Y^0})$ be a nonzero section of $K_M\otimes E$ on $Y^0=Y_{\text{reg}}$ such that
\begin{equation}\nonumber
 \int_{Y^0}|f|^2_{\omega,h}e^{-\varphi}dV_{M,\omega}[\psi]<+\infty.
\end{equation}
Then it follows from Theorem \ref{main result 1} that there exists a holomorphic $E$-valued $(n,0)$-form $f_1$ such that $f_1|_{Y_0}=f$ and
\begin{equation}\label{f_1 is finite}
\int_M c(-\psi)|f_1|^2_{\omega,h}e^{-\varphi}dV_{M,\omega}\le \left(\int_{0}^{+\infty}c(t_1)e^{-t_1}dt_1\right)
\int_{Y^0}|f|^2_{\omega,h}e^{-\varphi}dV_{M,\omega}[\psi].
\end{equation}

Let $Z_0:=Y_0$ and $V$ be an open subset of $M$ containing $Z_0$.
Denote $\mathcal{F}_{z_0}=\mathcal{E}(e^{-\psi})_{z_0}$ for any $z_0\in Z_0$. Now we can define the \textbf{minimal $L^{2}$ integral} as follows
\begin{equation}\nonumber
\begin{split}
G(t;c,\psi,he^{-\varphi},\mathcal{F},f_1):=\inf\Bigg\{ \int_{ \{ \psi<-t\}}|\tilde{f}|^2_{\omega,h}e^{-\varphi}c(-\psi)dV_{M,\omega}: \tilde{f}\in
H^0\big(\{\psi<-t\},\mathcal{O} (K_M\otimes E) \big ) \\
\&\, (\tilde{f}-f_1)_{z_0}\in
\mathcal{O} (K_M)_{z_0} \otimes \mathcal{F}_{z_0},\text{for any }  z_0\in Z_0 \Bigg\}.
\end{split}
\end{equation}
We simply denote $G(t;c,\psi,he^{-\varphi},\mathcal{F},f_1)$ by $G(t)$.

In \cite{GMY-boundary5}, we established the following concavity property of $G(t)$.
\begin{Theorem}[see \cite{GMY-boundary5}]\label{concavity of min l2 int}
If there exists $t \in [0,+\infty)$ satisfying that $G(t)\in(0,+\infty)$, we have that $G(h^{-1}(r))$ is concave with respect to  $r\in (0,\int_{0}^{+\infty}c(t)e^{-t}dt)$, $\lim\limits_{t\to 0+0}G(t)=G(0)$ and $\lim\limits_{t \to +\infty}G(t)=0$, where $h(t)=\int_{t}^{+\infty}c(l)e^{-l}dl$.
\end{Theorem}

In \cite{GMY-boundary5}, we gave a necessary condition for the concavity property degenerating to linearity.
\begin{Corollary}[see \cite{GMY-boundary5}]
\label{necessary condition for linear of G}
Let $c(t)\in \mathcal{G}_{0}$ such that $c(t)e^{-t}$ is decreasing with respect to $t\in[0,+\infty)$.
Assume that $G(t)<+\infty$ for some $t\ge 0$, and $G(h^{-1}(r))$ is linear with respect to $r\in[0,\int_0^{+\infty}c(s)e^{-s}ds)$, where $h(t)=\int_{t}^{+\infty}c(l)e^{-l}dl$.

Then there exists a unique $E$-valued holomorphic $(n,0)$-form $\tilde{F}$ on $M$
such that $(\tilde{F}-f)_{z_0}\in\mathcal{O} (K_M)_{z_0} \otimes \mathcal{F}_{z_0}$ holds for any  $z_0\in Z_0$,
and $G(t)=\int_{\{\psi<-t\}}|\tilde{F}|^2_{\omega,h}e^{-\varphi}c(-\psi)dV_{M,\omega}$ holds for any $t\ge 0$.
\end{Corollary}

\section{Proofs of Theorem \ref{main result} and Theorem \ref{main result 1}}
In this section, we prove Theorem \ref{main result} and Theorem \ref{main result 1}.

We firstly prove Theorem \ref{main result}.
\begin{proof}

  As $M$ is weakly pseudoconvex, there exists a smooth plurisubharmonic
exhaustion function $P$ on $M$. Let $M_k:=\{P<k\}$ $(k=1,2,...,) $. We choose $P$ such that
$M_1\ne \emptyset$.\par
Then $M_k$ satisfies $
M_k\Subset  M_{k+1}\Subset  ...M$ and $\cup_{k=1}^n M_k=M$. Each $M_k$ is weakly
pseudoconvex K\"ahler manifold with exhaustion plurisubharmonic function
$P_k=1/(k-P)$.
\par
\emph{We will fix $k$ during our discussion until the end of Step 10.}

\

\textbf{Step 1: regularization of $c(t)$.}

\

As $e^{\psi}$ is a smooth function on $M$ and $\psi<-T$ on $M$, we know that
$$\sup_{M_k}\psi<-T-8\epsilon_k,$$
where $\epsilon_k >0$ is a real number depending on $k$.

It follows from $c(t)$ belongs to class $\mathcal{G}_{T,\delta}$, by Lemma \ref{approx of ct}, that we have a sequence of functions $\{c_{k}(t)\}_{k\in\mathbb{Z}^+}$ which satisfies $c_k(t)$ is continuous on $[T,+\infty)$ and smooth on $[T+4\epsilon_k,+\infty)$ and other conditions in Lemma \ref{approx of ct}. Condition (6) of Lemma \ref{approx of ct} tells that
$$S_k(t):=\frac{\int^t_T\bigg(\frac{1}{\delta}c_k(T)e^{-T}+\int^{t_2}_T c_k(t_1)e^{-t_1}dt_1\bigg)dt_2+\frac{1}{\delta^2}c_k(T)e^{-T}+\kappa_k}{\frac{1}{\delta}c_k(T)e^{-T}+\int^t_T
c_k(t_1)e^{-t_1}dt_1}>s(t),$$
for any $t\ge T$ and $S'_k(t)>0$ on $[T+\epsilon_k,+\infty)$

 As $S_k(t)>s(t)$ on  $t\ge T$, we know that
 $$S_k(-\psi)\big(\sqrt{-1}\partial\bar{\partial}\varphi+\sqrt{-1}\partial\bar{\partial}\psi\big)
+\sqrt{-1}\partial\bar{\partial}\psi\ge 0$$ on $M\backslash\{\psi=-\infty\}$ in the sense of currents. Denote $u_k(t):=-\log(\frac{1}{\delta}c(T)e^{-T}+\int_{T}^{t}c_k(t_1)e^{-t_1}dt_1$. We note that we still have $S'_k(t)-S_k(t)u'_k(t)=1$ and $(S_k(t)+\frac{S'^2_k(t)}{u''_k(t)S_k(t)-S''_k(t)})e^{u_k(t)-t}=\frac{1}{c_k(t)}$.

\

\textbf{Step 2: construction of a family of smooth extensions $\tilde{f}_t$ of $f$ to a neighborhood of $\overline{M}_k\cap Y$ in $M$ with suitable estimates.}

\

Step 2 will be divided into four parts.

\

\emph{Part 1: construction of local coordinate charts $\{\Omega_i\}_{i=1}^N$, $\{U_i\}_{i=1}^N$ and a partition of unity $\{\xi_i\}_{i=1}^{N+1}$.}

\

Let $x\in Y$ be any point, we can find a local coordinate ball $\Omega'_x$ in $X$ centered at $x$ such that $E|_{\Omega'_x}$ is trivial. We also assume that $\psi$ can be written as
\begin{equation}\label{fl:local exp of psi}
\psi=c_x\log \sum_{1\le j\le j_0}|g_{x,j}|^2+u_{x}
\end{equation}
on $\Omega'_x$, where $c_x>0$ is a real number, $g_{x,j}\in \mathcal{O}_{\Omega'_x}$ and $u_{x}\in C^{\infty} (\Omega'_x)$.

Let $U_x\Subset \Omega_x \Subset \Omega'_x$ be three smaller local coordinate balls. Since $\overline{M_k}\cap Y$ is compact, we can find $x_1,\cdots,x_N\in \overline{M_k}\cap Y$ such that $\overline{M_k}\cap Y \subset \cup_{i=1}^N U_{x_i}$. We simply denote $U_{x_i}, \Omega_{x_i}, \Omega'_{x_i}$ by $U_{i}, \Omega_{i}, \Omega'_{i}$ respectively. We also write the local expression \eqref{fl:local exp of psi} of $\psi$ on $\Omega'_{i}$ by
$$\psi=\Gamma_i+u_i.$$

Choose an open set $U_{N+1}$ in $M$ such that $\overline{M_k}\cap Y\subset M\backslash \overline{U_{N+1}}\Subset \cup_{i=1}^N U_{i}$. Denote $U:=M\backslash \overline{U_{N+1}}$. Let $\{\xi_i\}_{i=1}^{N+1}$  be a partition of unity subordinate to the cover $\cup_{i=1}^{N+1} U_{i}$ of $M$. Then we know that $\text{supp} \xi_i \Subset U_i$ for $i=1,2,\cdots,N$ and $\sum_{i=1}^{N}\xi_i=1$ on $U$.

\

\emph{Part 2: construction of local holomorphic extensions $\hat{f}_{i,t}$ ($1\le i\le N$) of $f$ to $\Omega_i\cap \{\psi<-t\}$.}

\

By the proof of Proposition \ref{p:inte} \big(see step 1 formula \eqref{eq:220627b}\big), we know that inequality \eqref{mainth:ohsawa measure finite} implies
\begin{equation*}
   \int_{w'\in D_{p_0}}\frac{|f\circ\mu|^2_{\omega,h}\xi e^{-\tilde u-\varphi\circ\mu}}{|(w')^{ca'-b'}|^2}d\lambda(w')<+\infty.
\end{equation*}
It follows from Proposition \ref{p:soc2} that there exists a positive number $\beta\in(0,1)$ such that
\begin{equation}\label{step1part2 1}
\int_{\Omega_i\cap Y^0}|f|^2_{w,h(\text{det}h)^{\beta}}e^{-(1+\beta r)\varphi}dV_{M,\omega}[\psi]<+\infty,
\end{equation}
where $r$ is the rank of vector bundle $E$.

Let $T_1\ge T$ be a fixed number such that $c_k(t)e^{-t}$ is decreasing with respect to $t$ on $[T_1,+\infty)$. Let $\beta_2$ be a positive number which will be determined later. Denote $\hat{c}_0(t):=c(T_1)e^{(1-\beta_2)(t-T_1)}$ for $t\ge T_1$. Let
$$\hat{c}_1(t):=e^{-T_1}\max\{\hat{c}_0\big(t+(T_1-T)\big),c_k\big(t+(T_1-T)\big)\},$$
where $t\in [T,+\infty)$. Then $\hat{c}_1(t)e^{-t}$ is decreasing with respect to $t$ on $[T,+\infty)$ and satisfies all the conditions in class $\mathcal{G}_{T,\delta}$.

Denote $m_i:=\inf_{\Omega_i} u_i$ and $\hat{M}_i:=\sup_{\Omega_i} u_i$. For any $t\in[T,+\infty)$, it follows from inequality \eqref{step1part2 1}, $he^{-\varphi}$ is locally lower bounded and Proposition \ref{key pro l2 extension local} \big($\Omega\sim \Omega_i\cap \{\Gamma_i<-t-m_i\}$, $Y\sim \Omega_i\cap Y$, $\psi\sim \Gamma_i+t+m_i$, $c(t)\sim \hat{c}_1(t)$, $h\sim h(\det h)^{\beta}e^{-(1+\beta r)\varphi}$, $f\sim f$ with $L^2$ estimate \eqref{step1part2 1}\big) that $f$ has an $L^2$ holomorphic extension $\hat{f}_{i,t}$ from $\Omega_i\cap Y^0$ to $\Omega_i\cap \{\Gamma_i<-t-m_i\}$. Specifically, for any $1\le i\le N$, there exists a constant $\tilde{C}_i>0$ ($\tilde{C}_i$ depends on $\Omega_i$ and does not depend on $t$) and holomorphic extension $\hat{f}_{i,t}$  of $f$ from $\Omega_i\cap Y^0$ to $\Omega_i\cap \{\Gamma_i<-t-m_i\}$ which satisfies
\begin{equation}\label{step1part2 2}
\begin{split}
    & \int_{\Omega_i\cap \{\Gamma_i<-t-m_i\}}\hat{c}_1(-\Gamma_i-t-m_i)|\hat{f}_{i,t}|^2_{\omega,h(\text{det}h)^{\beta}}e^{-(1+\beta r)\varphi}dV_{M,\omega}\\
    \le & \tilde{C}_i\int_{\Omega_i\cap Y^0}|f|^2_{w,h(\text{det}h)^{\beta}}e^{-(1+\beta r)\varphi}dV_{M,\omega}[\Gamma_i+t+m_i]\\
    \le & C_1\int_{\Omega_i\cap Y^0}|f|^2_{w,h(\text{det}h)^{\beta}}e^{-(1+\beta r)\varphi}dV_{M,\omega}[\Gamma_i+t+m_i]\\
     \le & C_2e^{-t}\int_{\Omega_i\cap Y^0}|f|^2_{w,h(\text{det}h)^{\beta}}e^{-(1+\beta r)\varphi}dV_{M,\omega}[\psi]<+
    \infty,\\
\end{split}
\end{equation}
where $C_1:=\sup_{1\le i\le N}\tilde{C}_i$ and $C_2$ is a constant independent of $i$ and $t$.

\

\emph{Part 3: construction of local holomorphic extensions $\tilde{f}_{i,t}$ ($1\le i\le N$) of $f$ to $\Omega_i$.}

\

For each fixed $t$, we use inequality \eqref{step1part2 2} and Proposition \ref{key pro l2 method} ($\Omega\sim\Omega_i$, $\psi\sim\Gamma_i+t+m_i$,
$h\sim h(\text{det}h)^{\beta}e^{-(1+\beta r)\varphi}$) to $\hat{f}_{i,t}$ and some positive number $\beta_1$ which will be determined later and then we get a holomorphic section $\tilde{f}_{i,t}$ ($1\le i\le N$) on $\Omega_i$ satisfying $\tilde{f}_{i,t}=\hat{f}_{i,t}=f$ on $\Omega_i\cap Y^0$ with estimates,

\begin{equation}\label{step1part3 1}
\begin{split}
    \int_{\Omega_i\cap \{\Gamma_i<-t-m_i\}}\hat{c}_1(-\Gamma_i-t-m_i)|\tilde{f}_{i,t}|^2_{\omega,h(\text{det}h)^{\beta}}e^{-(1+\beta r)\varphi}dV_{M,\omega}
    \le  C_3e^{-t}
\end{split}
\end{equation}
and
\begin{equation}\label{step1part3 2}
\begin{split}
    \int_{\Omega_i}\frac{|\tilde{f}_{i,t}|^2_{\omega,h(\text{det}h)^{\beta}}e^{-(1+\beta r)\varphi}}{(1+e^{\Gamma_i+t+m_i})^{1+\beta_1}}dV_{M,\omega}
    \le  C_3e^{-t}
\end{split}
\end{equation}
for some $C_3>0$ which is independent of $t$.

It follows from $\liminf_{t\to+\infty}\hat{c}_1(t)>0$ and inequality \eqref{step1part3 1} that we have
\begin{equation}\label{step1part3 3}
\begin{split}
    \int_{\Omega_i\cap \{\psi<-t\}}|\tilde{f}_{i,t}|^2_{\omega,h(\text{det}h)^{\beta}}e^{-(1+\beta r)\varphi}dV_{M,\omega}
    \le  C_4e^{-t}
\end{split}
\end{equation}
for any $t$, where $C_4>0$ is independent of $t$.

Since $\Gamma_i$ is upper bounded on $\Omega_i$, it follows from inequality \eqref{step1part3 2} that we have
\begin{equation}\label{step1part3 4}
\begin{split}
    \int_{\Omega_i}|\tilde{f}_{i,t}|^2_{\omega,h(\text{det}h)^{\beta}}e^{-(1+\beta r)\varphi}dV_{M,\omega}
    \le  C_5e^{\beta_1 t}
\end{split}
\end{equation}
for any $t$, where $C_5>0$ is independent of $t$.

Note that $he^{-\varphi}$ is locally lower bounded. We have $h(\text{det}h)^{\beta}e^{-(1+\beta r)\varphi}=he^{-\varphi}(\text{det}h)^{\beta}e^{-\beta r\varphi}\ge \tilde{M}_i I_r$ on $\Omega_i$ for some positive number $\tilde{M}_i$, where $I_r$ is the standard metric on $E|_{\Omega_i}\cong \Omega_i\times \mathbb{C}^r$, then we know that
\begin{equation}\label{step1part3 5}
\begin{split}
   \sup_{V_i} |\tilde{f}_{i,t}|^2_{I_r}\le C_6e^{\beta_1 t}
\end{split}
\end{equation}
for any $t$, where $C_6>0$ is independent of $t$.

As $he^{-\varphi}$ is locally lower bounded, it follows from inequalities \eqref{step1part3 3}, \eqref{step1part3 5} and Proposition \ref{p:inte} that for each $\tilde{f}_{i,t}$, we have

	\begin{equation}
		\label{step1part3 6}
		\limsup_{t\rightarrow+\infty}\int_{U_i\cap\{-t-1<\psi<-t\}}\xi_i|\tilde{f}_{i,t}|^2_{\omega,h}e^{-\varphi-\psi}dV_{M,\omega}
\le\int_{U_i\cap Y^0}\xi_i|f|^2_{\omega,h}e^{-\varphi}dV_{M,\omega}[\psi].
	\end{equation}

\

\emph{Part 4: construction of a family of smooth extensions $\tilde{f}_{t}$  of $f$ to a neighborhood of $\overline{M}_k\cap Y$ in $M$.}

\

Define $\tilde{f}_t:=\sum_{i=1}^{N}\xi_i\tilde{f}_{i,t}$ for all $t$. Note that
$$\tilde{f}_t|_{U_j}=\sum_{i=1}^{N} \xi_i \tilde{f}_{j,t}+\sum_{i=1}^{N}\xi_i(\tilde{f}_{i,t}-\tilde{f}_{j,t})$$
for any $i=1,\cdots,N$ and $\sum_{i=1}^{N}\xi_i=1$ on $U$, we have
\begin{equation}\label{step1part4 1}
  |D''\tilde{f}_t|^2_{\omega,h}|_{U_j\cap U}=|\sum_{i=1}^{N} \bar{\partial}\xi_i\wedge (\tilde{f}_{i,t}-\tilde{f}_{j,t})|^2_{\omega,h}
\end{equation}
holds for any $t$.

Let $\mu$ and $W$ be as in the Step 1 of the proof of Proposition \ref{p:inte} where $W$ is a coordinate ball centered at a point $\tilde{z}\in \overline{\mu^{-1}(U_i\cap U_j)}\cap \mu^{-1}(\{\psi=-\infty\})$. We choose $t$ big enough such that $(U_l\cap \{\psi <-t\})\subset U$, for any $l=1,\ldots,N$. Denote $W_{i,j,t}:=W\cap \mu^{-1}(U_i\cap U_j)\cap\{\psi\circ \mu <-t\}$.

 By using similar discussion as in \eqref{eq:220627d}, \eqref{eq:220627g} and \eqref{eq:220627h}
 (recall that $\kappa:=\{p:ca_p-b_p=1\}$), it follows from inequality \eqref{step1part3 5} that we have
\begin{equation}\label{step1part4 2}
  |\tilde{f}_{i,t}\circ \mu-\tilde{f}_{j,t}\circ \mu|^2_{\omega,I_r}|_{W_{i,j,t}}\le C_7e^{\beta_1 t }\prod_{p \in \kappa}|w_p|^2
\end{equation}
when $\kappa \neq \emptyset$ and $t$ is big enough, and we have
\begin{equation}\label{step1part4 3}
  |\tilde{f}_{i,t}\circ \mu-\tilde{f}_{j,t}\circ \mu|^2_{\omega,I_r}|_{W_{i,j,t}}\le C_7e^{\beta_1 t }
\end{equation}
when $\kappa= \emptyset$ and $t$ is big enough, where $I_r$ is the standard metric on $(\mu^{-1}E)|_{W_{i,j,t}}$ and $C_7>0$ is a real number independent of $t$.

\

\

\textbf{Step 3: recall some notations.}

\

Let $\epsilon \in (0,\frac{1}{8})$. Let $\{v_{t_0,\epsilon}\}_{\epsilon \in
(0,\frac{1}{8})}$ be a family of smooth increasing convex functions on $\mathbb{R}$, such
that:
\par
(1) $v_{t_0,\epsilon}(t)=t$ for $t\ge-t_0-\epsilon$, $v_{\epsilon}(t)=constant$ for
$t<-t_0-1+\epsilon$;\par
(2) $v^{''}_{t_0,\epsilon}(t)$ are convergence pointwisely
to $\mathbb{I}_{(-t_0-1,-t_0)}$,when $\epsilon \to 0$, and $0\le
v^{''}_{t_0,\epsilon}(t) \le \frac{1}{1-4\epsilon}\mathbb{I}_{(-t_0-1+\epsilon,-t_0-\epsilon)}$
for ant $t \in \mathbb{R}$;\par
(3) $v^{'}_{t_0,\epsilon}(t)$ are convergence pointwisely to $b_{t_0}(t):=\int^{t}_{-\infty} \mathbb{I}_{\{-t_0-1< s < -t_0\}}ds$  when $\epsilon \to 0$ and $0 \le v^{'}_{t_0,\epsilon}(t) \le 1$ for any
$t\in \mathbb{R}$.\par
One can construct the family $\{v_{t_0,\epsilon}\}_{\epsilon \in (0,\frac{1}{8})}$  by
 setting
\begin{equation}\nonumber
\begin{split}
v_{t_0,\epsilon}(t):=&\int_{-\infty}^{t}\bigg(\int_{-\infty}^{t_1}(\frac{1}{1-4\epsilon}
\mathbb{I}_{(-t_0-1+2\epsilon,-t_0-2\epsilon)}*\rho_{\frac{1}{4}\epsilon})(s)ds\bigg)dt_1\\
&-\int_{-\infty}^{-t_0}\bigg(\int_{-\infty}^{t_1}(\frac{1}{1-4\epsilon}
\mathbb{I}_{(-t_0-1+2\epsilon,-t_0-2\epsilon)}*\rho_{\frac{1}{4}\epsilon})(s)ds\bigg)dt_1-t_0,
\end{split}
\end{equation}
where $\rho_{\frac{1}{4}\epsilon}$ is the kernel of convolution satisfying
$\text{supp}(\rho_{\frac{1}{4}\epsilon})\subset
(-\frac{1}{4}\epsilon,{\frac{1}{4}\epsilon})$.
Then it follows that
\begin{equation}\nonumber
v^{''}_{t_0,\epsilon}(t)=\frac{1}{1-4\epsilon}
\mathbb{I}_{(-t_0-1+2\epsilon,-t_0-2\epsilon)}*\rho_{\frac{1}{4}\epsilon}(t),
\end{equation}
and
\begin{equation}\nonumber
v^{'}_{t_0,\epsilon}(t)=\int_{-\infty}^{t}\bigg(\frac{1}{1-4\epsilon}
\mathbb{I}_{(-t_0-1+2\epsilon,-t_0-2\epsilon)}*\rho_{\frac{1}{4}\epsilon}\bigg)(s)ds.
\end{equation}
Note that $\text{supp}v^{''}_{t_0,\epsilon}(t)\Subset (-t_0-1+\epsilon,-t_0-\epsilon)$ and
$\text{supp}\big(1- v^{'}_{t_0,\epsilon}(t)\big)\Subset (-\infty,-t_0-\epsilon)$.

We also note that $S_k\in C^{\infty}\big([T+4\epsilon_k,+\infty)\big)$ satisfies
 $S_k'>0$ on $[T+\epsilon_k,+\infty)$ and $u_k\in C^{\infty}\big([T+4\epsilon_k,+\infty)\big)$ satisfies
$\lim_{t\to +\infty}u_k(t)=-\log(\frac{1}{\delta}c_k(T)e^{-T}+\int^{+\infty}_T c_k(t_1)e^{-t_1}dt_1)$ and $u'_k<0$.
Recall that $u_k(t)$ and $S_k(t)$ satisfy $$S'_k(t)-S_k(t)u'_k(t)=1$$ and $$(S_k(t)+\frac{S'^2_k(t)}{u''_k(t)S_k(t)-S''_k(t)})e^{u_k(t)-t}=\frac{1}{c_k(t)}.$$
Note that $u_k''S_k-S_k''=-S'_ku'_k>0$ on $[T+2\epsilon_k,+\infty)$.
Denote $\tilde{g}_k(t):=\frac{u''_kS_k-S''_k}{S'^2_k}(t)$, then $\tilde{g}_k(t)$ is a positive smooth function on $[T+4\epsilon_k,+\infty)$.

Denote $\sum:=\{\psi=-\infty\}$.
As $\psi$ has neat analytic singularities, we know that $\sum$ is an analytic subset of $M$ and $\psi$ is smooth on $M\backslash \sum$.

Denote $\eta:=S_k\big(-v_{t_0,\epsilon}(\psi)\big)$, $\phi:=u_k\big(-v_{t_0,\epsilon}(\psi)\big)$ and $g:=\tilde{g}_k\big(-v_{t_0,\epsilon}(\psi)\big)$. Then $\eta$ and $g$ are smooth bounded positive functions on $M_k$ such that $\eta+g^{-1}$ is a smooth bounded positive function on $M_k$.

\

\textbf{Step 4: regularization process of $\varphi+\psi$ and $h$.}

\

Let $\mu:\tilde{M}\to M$ be the proper mapping defined in the proof of the Proposition \ref{p:inte}. Denote $\tilde{M}_{k+1}:=\mu^{-1}(M_{k+1})$, $\tilde{M}_{k}:=\mu^{-1}(M_{k})$ and $\tilde{\sum}:=\mu^{-1}(\sum)$, where $\sum=\{\psi=-\infty\}$. Denote
$$\sigma_1:=\sqrt{-1}\partial\bar{\partial}(\psi\circ\mu)-\sum_{j} q_j[D_j],$$
where $\{D_j\}$ are the irreducible components of $\tilde{\sum}$ and $\{q_j\}$ are positive numbers such that $\sigma_1$ is a smooth real $(1,1)$-form on $\tilde{M}$.

It follows from Lemma \ref{sequence of blow up} that exist a positive number $a_k>0$, a quasi-plurisubharmonic function $\tilde{\Upsilon}$ on $\tilde{M}$ and divisor $H$ on $\tilde{M}$ such that
 $$\tilde{\omega}_k:=a_k\mu^{*}\omega+
 \sqrt{-1}\partial\bar{\partial}\tilde{\Upsilon}-2\pi[H]$$
 is a K\"ahler metric on $\tilde{M}_{k+1}$. By the construction of $\mu$, we know that $H\subset \tilde{\sum}$. Denote $\Upsilon:=\mu_*{\tilde{\Upsilon}}$. By Lemma \ref{property of upsilon}, we know that $\Upsilon$ is upper-semicontinuous function on $M$.

Denote $\Phi=\varphi+\psi$.
Note that $\mu:\tilde{M}\backslash \tilde{\sum}\to M\backslash \sum$ is biholomorphic and $\big(\sum_{j} q_j[D_j]\big)|_{\tilde{M}\backslash \tilde{\sum}}=0$. It follows from the curvature conditions in Theorem \ref{main result} that
$$\sqrt{-1}\partial\bar{\partial}(\Phi\circ\mu)|_{\tilde{M}\backslash \tilde{\sum}}\ge 0$$
and
$$\sqrt{-1}\partial\bar{\partial}(\Phi\circ\mu)|_{\tilde{M}\backslash \tilde{\sum}}+\frac{1}{s(-\psi\circ\mu)}
\sigma_1|_{\tilde{M}\backslash \tilde{\sum}}\ge 0$$
hold on $\tilde{M}\backslash \tilde{\sum}$.
As $\Phi\circ\mu$ is quasi-plurisubharmonic on $\tilde{M}$, given any small open set $U\subset \tilde{M}$, we can find a smooth function $\tau$ on $U$ such that $\Phi\circ\mu+\tau$ is plurisubharmonic function on $U$. As $\tau$ is smooth and the restriction of positive  closed current on any analytic subset is still positive and closed (see  Corollary 2.4 of Chapter 3 in \cite{Demaillybook}), we know that $$\sqrt{-1}\partial\bar{\partial}(\Phi\circ\mu)|_{\tilde{\sum}\cap U}
=\sqrt{-1}\partial\bar{\partial}(\Phi\circ\mu+\tau)|_{\tilde{\sum}\cap U}\ge 0.$$
Hence we know that
\begin{equation}\label{step3 curvature cond 1}
\sqrt{-1}\partial\bar{\partial}(\Phi\circ\mu)=
\sqrt{-1}\partial\bar{\partial}(\Phi\circ\mu)|_{\tilde{M}\backslash \tilde{\sum}}+
\sqrt{-1}\partial\bar{\partial}(\Phi\circ\mu)|_{\tilde{\sum}}\ge 0
\end{equation}
hold on $\tilde{M}$.

As $\sigma_1$ is smooth on $\tilde{M}$, we have
$$\sqrt{-1}\partial\bar{\partial}(\Phi\circ\mu)|_{\tilde{\sum}}+\big(\frac{1}{s(-\psi\circ\mu)}
\sigma_1\big)|_{\tilde{\sum}}=\sqrt{-1}\partial\bar{\partial}(\Phi\circ\mu)|_{\tilde{\sum}}\ge0.$$
Hence
\begin{equation}\label{step3 curvature cond 2}
\begin{split}
&\sqrt{-1}\partial\bar{\partial}(\Phi\circ\mu)+\frac{1}{s(-\psi\circ\mu)}
\sigma_1\\
=&\sqrt{-1}\partial\bar{\partial}(\Phi\circ\mu)|_{\tilde{M}\backslash \tilde{\sum}}+\frac{1}{s(-\psi\circ\mu)}
\sigma_1|_{\tilde{M}\backslash \tilde{\sum}}+\sqrt{-1}\partial\bar{\partial}(\Phi\circ\mu)|_{ \tilde{\sum}}+\frac{1}{s(-\psi\circ\mu)}
\sigma_1|_{\tilde{\sum}}\\
\ge&0
\end{split}
\end{equation}
holds on $\tilde{M}$.

Note that $\tilde{M}_k$ is relatively compact in $\tilde{M}$, there exists a continuous nonnegative (1,1)-form $\varpi$ on $(\tilde{M}_{k+1},\tilde{\omega}_k)$ such that
\begin{equation}\nonumber
\begin{split}
(\sqrt{-1}&\Theta_{T\tilde{M}}+\varpi \otimes Id_{T\tilde{M}})(\kappa_1 \otimes \kappa_2,\kappa_1
\otimes \kappa_2)\ge 0, \quad\forall \kappa_1,\kappa_2 \in T\tilde{M}
\end{split}
\end{equation}
holds on $\tilde{M}_k$.
It follows from Lemma \ref{regularization on cpx mfld}, inequalities \eqref{step3 curvature cond 1} and \eqref{step3 curvature cond 2} that there exists a family of functions $\{\tilde{\Phi}_{\zeta,\rho}\}_{\zeta>0,\rho\in(0,\rho_1)}$ on a neighborhood of the closure of $\tilde{M}_k$ such that \\
(1) $\tilde{\Phi}_{\zeta,\rho}$ is a quasi-plurisubharmonic function on a neighborhood  of the closure of $\tilde{M}_k$, smooth on $\tilde{M}_{k+1}\backslash E_{\zeta}(\Phi\circ \mu)$, increasing with respect to $\zeta$ and $\rho$ on $\tilde{M}_k$ and converges to $\Phi\circ \mu$ on $\tilde{M}_k$ as $\rho \to 0$,\\
(2) $\frac{\sqrt{-1}}{\pi}\partial\bar{\partial}\tilde{\Phi}_{\zeta,\rho}
\ge-\zeta\varpi-\delta_{\rho}\tilde{\omega}_k$ on $\tilde{M}_k$,\\
(3) $\frac{\sqrt{-1}}{\pi}\partial\bar{\partial}\tilde{\Phi}_{\zeta,\rho}
\ge-\frac{1}{s(-\psi\circ\mu)}\frac{\sigma_1}{\pi}-\zeta\varpi-\delta_{\rho}\tilde{\omega}_k$ on $\tilde{M}_k$,\\
where  $E_{\zeta}(\Phi\circ \mu):=\{x\in \tilde{M}:v(\Phi\circ\mu,x)\ge \zeta\}$ is the $\zeta$-upperlevel set of Lelong numbers of $\Phi\circ\mu$ and $\{\delta_{\rho}\}$ is an increasing family of positive numbers such that $\lim_{\rho\to 0}\delta_{\rho}=0$.

As $\tilde{\omega}_{k}$ is positive on $\tilde{M}_{k+1}$ and $\tilde{M}_k$ is relatively compact in $\tilde{M}_{k+1}$, there exists a positive number $n_k>1$ such that $n_k\tilde{w}_k\ge \varpi$ on $\tilde{M}_k$. Let $\rho=\frac{1}{m'}$, where $m'\in\mathbb{Z}_{\ge 1}$. Denote $\delta_{m'}:=\delta_{\frac{1}{m'}}$ and $\zeta=\delta_{m'}$. Denote $\tilde{\Phi}_{m'}:=\tilde{\Phi}_{\delta_{\frac{1}{m'}},\frac{1}{m'}}$. Then we have a sequence of functions $\{\tilde{\Phi}_{m'}\}_{m'\ge 1}$ such that
\\
(1) $\tilde{\Phi}_{m'}$ is a quasi-plurisubharmonic function on a neighborhood of the closure of $\tilde{M}_k$, smooth on $\tilde{M}_{k+1}\backslash E_{m'}(\Phi\circ \mu)$, decreasing with respect to $m'$ on $\tilde{M}_k$ and converges to $\Phi\circ \mu$ on $\tilde{M}_k$ as $m' \to +\infty$,\\
(2) $\frac{\sqrt{-1}}{\pi}\partial\bar{\partial}\tilde{\Phi}_{m'}
\ge-\delta_{m'}n_k\tilde{\omega}_k-\delta_{m'}\tilde{\omega}_k
\ge-2\delta_{m'}n_k\tilde{\omega}_k$ on $\tilde{M}_k$,\\
(3) $\frac{\sqrt{-1}}{\pi}\partial\bar{\partial}\tilde{\Phi}_{m'}
\ge-\frac{1}{s(-\psi\circ\mu)}\frac{\sigma_1}{\pi}-\delta_{m'}n_k\tilde{\omega}_k-\delta_{m'}\tilde{\omega}_k
\ge-\frac{1}{s(-\psi\circ\mu)}\frac{\sigma_1}{\pi}-2\delta_{m'}n_k\tilde{\omega}_k$ on $\tilde{M}_k$,\\
where  $E_{m'}(\Phi\circ \mu):=\{x\in \tilde{M}:v(\Phi\circ\mu,x)\ge \delta_{m'}\}$ is the upperlevel set of Lelong numbers of $\Phi\circ\mu$ and $\{\delta_{m'}\}$ is an decreasing family of positive numbers such that $\lim_{m' \to +\infty}\delta_{m'}=0$. As $\mu:\tilde{M}\backslash \tilde{\sum}\to M\backslash \sum$ is biholomorphic, we know that
$$\sqrt{-1}\partial\bar{\partial}\tilde{\Phi}_{m'}\circ \mu^{-1}
\ge-2\pi\delta_{m'}n_k(\mu^{-1})^*\tilde{\omega}_k$$
and
$$\sqrt{-1}\partial\bar{\partial}\tilde{\Phi}_{m'}\circ \mu^{-1}
\ge-\frac{1}{s(-\psi)}(\mu^{-1})^*\sigma_1-2\pi\delta_{m'}n_k(\mu^{-1})^*\tilde{\omega}_k$$
hold on $M_k\backslash \sum$. By the definition of $\tilde{w}_k$, we have
\begin{equation}\label{step3 curvature result 1}
 \sqrt{-1}\partial\bar{\partial} \big(\tilde{\Phi}_{m'}\circ \mu^{-1}\big)+2\pi n_k\delta_{m'}\sqrt{-1}\partial\bar{\partial}\Upsilon
\ge-2\pi n_ka_k \delta_{m'}\omega
\end{equation}
and
\begin{equation}\label{step3 curvature result 2}
 \sqrt{-1}\partial\bar{\partial} \big(\tilde{\Phi}_{m'}\circ \mu^{-1}\big)+2\pi n_k\delta_{m'}\sqrt{-1}\partial\bar{\partial}\Upsilon+\frac{1}{s(-\psi)}\sqrt{-1}\partial\bar{\partial}\psi
\ge-2\pi n_ka_k\delta_{m'}\omega
\end{equation}
hold on $M_k\backslash \sum$. We simply denote $\tilde{\Phi}_{m'}\circ \mu^{-1}$ by $\Phi_{m'}$.

Note that $E_{m'}(\Phi\circ \mu)$ is an analytic subset in $\tilde{M}$, Remmert's proper mapping theorem shows that
$$\sum_{m'}:=\mu\big(E_{m'}(\Phi\circ \mu)\big)$$
is an analytic set in $M$.

Note that $(M,E,\Sigma,M_k,h,h_{k,s})$ is a singular metric on $E$ and $\Theta_{h}(E)\ge^s_{Nak} 0$ on $M$ in the sense of Definition \ref{singular nak}. We know that for any $k\ge 1$, there exists a sequence of hermitian metrics $\{h_{k,m}\}_{m=1}^{+\infty}$  of class $C^2$ convergent point-wisely to $h$ on $M_k$ which satisfies
\par
(1) for any $x\in \Omega:\ |e_x|_{h_{k,m}}\le |e_x|_{h_{k,m+1}},$ for any $m\ge 1$ and any $e_x\in E_x$;
\par
(2) $\Theta_{h_{k,m}}(E)\ge_{Nak} -\lambda_{k,m}\omega\otimes Id_E$ on $M_k$;
\par
(3)  $\lambda_{k,m}\to 0$ a.e. on $M_k$, where $\lambda_{k,m}$ is a sequence of continuous functions on $\overline{M_k}$;
\par
(4)  $0\le \lambda_{k,m}\le \lambda_k$ on $M_k$, for any $s\ge 1$, where $\lambda_k$ is a continuous function $\overline{M_k}$.

Since $k$ is fixed until last step, we simply denote $h_{k+1,m}$, $\lambda_{k+1,m}$ and $\lambda_{k+1}$ by $h_m$, $\lambda_{m}$ and $\lambda$  respectively. Denote $\tilde{h}_{m,m'}:=h_me^{-\Phi_{m'}}e^{-2\pi
n_k\delta_{m'}\Upsilon}e^{-\phi}$ on $M_k\backslash  \big(\sum\cup \sum_{m'}\big)$.

Note that, by Lemma \ref{completeness}, $M_k\backslash \big(\sum\cup \sum_{m'}\big)$ carries a complete K\"ahler metric.

\

\textbf{Step 5: some calculations.}\label{step 4 of main result}

\

We set
$B=[\eta \sqrt{-1}\Theta_{\tilde{h}_{m,m'}}-\sqrt{-1}\partial \bar{\partial}
\eta\otimes\text{Id}_E-\sqrt{-1}g\partial\eta \wedge\bar{\partial}\eta\otimes\text{Id}_E, \Lambda_{\omega}]$ on $M_k\backslash \big(\sum\cup \sum_{m'}\big)$. Direct calculation shows that
\begin{equation}\nonumber
\begin{split}
\partial\bar{\partial}\eta=&
-S_k'\big(-v_{t_0,\epsilon}(\psi)\big)\partial\bar{\partial}\big(v_{t_0,\epsilon}(\psi)\big)
+S_k''\big(-v_{t_0,\epsilon}(\psi)\big)\partial\big(v_{t_0,\epsilon}(\psi)\big)\wedge
\bar{\partial}\big(v_{t_0,\epsilon}(\psi)\big),\\
\eta\Theta_{\tilde{h}_{m,m'}}=&\eta\partial\bar{\partial}\phi\otimes\text{Id}_E+\eta\Theta_{h_{m}}+
\eta\partial\bar{\partial}\Phi_{m'}\otimes\text{Id}_E
+\eta(2\pi n_k\delta_{m'})\partial\bar{\partial}\Upsilon\otimes\text{Id}_E\\
=&S_ku_k''\big(-v_{t_0,\epsilon}(\psi)\big)\partial\big(v_{t_0,\epsilon}(\psi)\big)\wedge
\bar{\partial}\big(v_{t_0,\epsilon}(\psi)\big)\otimes\text{Id}_E
-S_ku_k'\big(-v_{t_0,\epsilon}(\psi)\big)\partial\bar{\partial}\big(v_{t_0,\epsilon}(\psi)\big)\otimes\text{Id}_E\\
+&S_k\Theta_{h_{m}}+
S_k(2\pi n_k\delta_{m'})\partial\bar{\partial}\Upsilon\otimes\text{Id}_E
+S_k\partial\bar{\partial}\Phi_{m'}\otimes\text{Id}_E.
\end{split}
\end{equation}
Hence
\begin{equation}\nonumber
\begin{split}
&\eta \sqrt{-1}\Theta_{\tilde{h}_{m,m'}}-\sqrt{-1}\partial \bar{\partial}
\eta\otimes\text{Id}_E-\sqrt{-1}g\partial\eta \wedge\bar{\partial}\eta\otimes\text{Id}_E\\
=&S_k\sqrt{-1}\Theta_{h_{m}}+
S_k(2\pi n_k\delta_{m'})\sqrt{-1}\partial\bar{\partial}\Upsilon\otimes\text{Id}_E
+S_k\sqrt{-1}\partial\bar{\partial}\Phi_{m'}\otimes\text{Id}_E\\
+&\big(S_k'-S_ku_k'\big)\big(v'_{t_0,\epsilon}(\psi)\sqrt{-1}\partial\bar{\partial}(\psi)+
v''_{t_0,\epsilon}(\psi)\sqrt{-1}\partial(\psi)\wedge\bar{\partial}(\psi)\big)\otimes\text{Id}_E\\
+&[\big(u_k''S_k-S_k''\big)-\tilde{g}_kS_k'^2]\sqrt{-1}\partial\big(v_{t_0,\epsilon}(\psi)\big)\wedge\bar{\partial}\big(v_{t_0,\epsilon}(\psi)\big)\otimes\text{Id}_E,
\end{split}
\end{equation}
where we omit the term $-v_{t_0,\epsilon}(\psi)$ in $(S_k'-S_ku_k')\big(-v_{t_0,\epsilon}(\psi)\big)$ and $[(u_k''S_k-S_k'')-\tilde{g}_kS_k'^2]\big(-v_{t_0,\epsilon}(\psi)\big)$ for simplicity.
Note that $S_k'(t)-S_k(t)u_k'(t)=1$, $\frac{u_k''(t)S_k(t)-S_k''(t)}{S_k'^2(t)}-\tilde{g}_k(t)=0$. We have
\begin{equation}\label{cal of curvature}
\begin{split}
&\eta \sqrt{-1}\Theta_{\tilde{h}_{m,m'}}-\sqrt{-1}\partial \bar{\partial}
\eta\otimes\text{Id}_E-\sqrt{-1}g\partial\eta \wedge\bar{\partial}\eta\otimes\text{Id}_E\\
=&S_k\sqrt{-1}\Theta_{h_{m}}+
S_k(2\pi n_k\delta_{m'})\sqrt{-1}\partial\bar{\partial}\Upsilon\otimes\text{Id}_E
+S_k\sqrt{-1}\partial\bar{\partial}\Phi_{m'}\otimes\text{Id}_E\\
+&\big(v'_{t_0,\epsilon}(\psi)\sqrt{-1}\partial\bar{\partial}(\psi)+
v''_{t_0,\epsilon}(\psi)\sqrt{-1}\partial(\psi)\wedge\bar{\partial}(\psi)\big)\otimes\text{Id}_E.
\end{split}
\end{equation}
We would like to discuss a property of $S_k(t)$.
\begin{Lemma}
\label{lem:main}
For large enough $t_{0}$, and for any $\varepsilon\in(0,1/4)$, the inequality
\begin{equation}
\label{equ:main}
S_k(-v_{t_{0},\varepsilon}(t))\geq S_k(-t)v'_{t_{0},\varepsilon}(t)
\end{equation}
holds any $t\in(-\infty,-T)$.
\end{Lemma}
\begin{proof}By the construction of $v_{t_{0},\varepsilon}(t)$ and $v'_{t_{0},\varepsilon}(t)$, we know that $S_k(-v_{t_{0},\varepsilon}(t))=S_k(-t)v'_{t_{0},\varepsilon}(t)$ holds for any $t\geq -t_{0}$. Note that $v'_{t_{0},\varepsilon}(t)=0$ for any $t\leq -t_{0}-1$. We know that $S_k(-v_{t_{0},\varepsilon}(t))>S_k(-t)v'_{t_{0},\varepsilon}(t)=0$ holds for any $t\leq -t_{0}-1$.
It suffices to consider the inequality \eqref{equ:main} for any $t\in(-t_{0}-1,t_{0})$.

Note that $0<S'_k(t)<1$  on $[T+\epsilon_k,+\infty)$. We know that $S_k(t)$ is increasing with respect to $t$ on $[T+\epsilon_k,+\infty)$. It follows from $\int_{T}^{+\infty}c(t)e^{-t}dt<+\infty$ that we have $\lim_{t \to +\infty}S_k(t)=+\infty$. We recall the following well-known lemma in mathematical analysis.
\begin{Lemma}
\label{lem:integ}
Let $f\geq 0$ be a continuous decreasing function on $[-t_{0}-1,-t_{0}]$.
Then $\int_{t}^{t_{0}}f(t_{1})dt_{1}\leq f(t)$ holds for any $t\in[-t_{0}-1,-t_{0}]$.
\end{Lemma}

Now we prove Lemma \ref{lem:main} by using Lemma \ref{lem:integ}.
It follows from $S_k'<1$, the differential mean value theorem $($implies the first $"\geq"$$)$, $S_k'>0$ (implies the second $"\geq"$), and Lemma \ref{lem:integ} $(f(t)\sim(-v'_{t_{0},\varepsilon}(t)+1)$, $\int_{t}^{t_{0}}f(t_{1})dt_{1}\sim v_{t_{0},\varepsilon}(t)-t)$ $($implies the third $"\geq"$$)$, that
\begin{equation}
\label{equ:deri_main}
\begin{split}
&S_k(-v_{t_{0},\varepsilon}(t))-S_k(-t)v'_{t_{0},\varepsilon}(t)
\\=&(S_k(-v_{t_{0},\varepsilon}(t))-S_k(-t))+(S_k(-t)-S_k(-t)v'_{t_{0},\varepsilon}(t))
\\\geq&
(-v_{t_{0},\varepsilon}(t)+t)+S_k(-t)(1-v'_{t_{0},\varepsilon}(t))
\\\geq&
(-v_{t_{0},\varepsilon}(t)+t)+S_k(t_{0})(1-v'_{t_{0},\varepsilon}(t))
\\\geq&
-(-v'_{t_{0},\varepsilon}(t)+1)+S_k(t_{0})(1-v'_{t_{0},\varepsilon}(t))
\end{split}
\end{equation}
holds for any $t\in(-t_{0}-1,-t_{0})$.

Then when $t_0$ is big enough (such that $S_k(t_{0})\geq1$), by inequality \eqref{equ:deri_main}, we know inequality \eqref{equ:main} holds for any $t< -T$.

Lemma \ref{lem:main} has been proved.
\end{proof}

It follows from inequalities \eqref{step3 curvature result 1}, \eqref{step3 curvature result 2} and Lemma \ref{lem:main} that, when $t_0$ is big enough, we have
\begin{align}
&\eta \sqrt{-1}\Theta_{\tilde{h}_{m,m'}}-\sqrt{-1}\partial \bar{\partial}
\eta\otimes\text{Id}_E-\sqrt{-1}g\partial\eta \wedge\bar{\partial}\eta\otimes\text{Id}_E\notag\\
=&S_k(\Theta_{h_{m}}+\lambda_m\omega\otimes \text{Id}_E)-S_k\lambda_m\omega\otimes \text{Id}_E \notag\\
+&S_k\bigg(2\pi n_k\delta_{m'}\partial\bar{\partial}\Upsilon\otimes\text{Id}_E
+\partial\bar{\partial}\Phi_{m'}\otimes\text{Id}_E
+2\pi n_ka_k \delta_{m'}\omega\otimes\text{Id}_E\bigg) \notag\\
-&2S_k\pi n_ka_k \delta_{m'}\omega\otimes\text{Id}_E+v'_{t_0,\epsilon}(\psi)\sqrt{-1}\partial\bar{\partial}\psi\otimes \text{Id}_E+
v''_{t_0,\epsilon}(\psi)\sqrt{-1}\big(\partial\psi\wedge\bar{\partial}\psi\big)\otimes\text{Id}_E \notag\\
\ge &S_k\big(-\psi\big) v'_{t_0,\epsilon}(\psi)\bigg(2\pi n_k\delta_{m'}\partial\bar{\partial}\Upsilon\otimes\text{Id}_E
+\partial\bar{\partial}\Phi_{m'}\otimes\text{Id}_E
+2\pi n_ka_k \delta_{m'}\omega\otimes\text{Id}_E\bigg)\notag\\
+&v'_{t_0,\epsilon}(\psi)S_k(-\psi)\frac{1}{S_k(-\psi)}\sqrt{-1}\partial\bar{\partial}\psi\otimes \text{Id}_E\notag\\
 -&(S_k\lambda_m+2S_k\pi n_ka_k \delta_{m'})\omega\otimes \text{Id}_E
+
v''_{t_0,\epsilon}(\psi)\sqrt{-1}\big(\partial\psi\wedge\bar{\partial}\psi\big)\otimes\text{Id}_E\notag\\
= &S_k\big(-\psi\big) v'_{t_0,\epsilon}(\psi)\bigg(2\pi n_k\delta_{m'}\partial\bar{\partial}\Upsilon\otimes\text{Id}_E
+\partial\bar{\partial}\Phi_{m'}\otimes\text{Id}_E
+2\pi n_ka_k \delta_{m'}\omega\otimes\text{Id}_E\notag\\
+&\frac{1}{S_k(-\psi)}\sqrt{-1}\partial\bar{\partial}\psi\otimes \text{Id}_E\bigg)
 -(S_k\lambda_m+2S_k\pi n_ka_k \delta_{m'})\omega\otimes \text{Id}_E
+
v''_{t_0,\epsilon}(\psi)\sqrt{-1}\big(\partial\psi\wedge\bar{\partial}\psi\big)\otimes\text{Id}_E\notag\\
\ge &-(S_k\lambda_m+2S_k\pi n_ka_k \delta_{m'})\omega\otimes \text{Id}_E
+
v''_{t_0,\epsilon}(\psi)\sqrt{-1}\big(\partial\psi\wedge\bar{\partial}\psi\big)\otimes\text{Id}_E.
\label{cal of curvature}
\end{align}

As $S_k(t)$ is increasing with respect to $t$ on $[T+2\epsilon_k,+\infty)$ and $v_{t_0,\epsilon}(\psi)>-t_0-1$, we know that $S_k(-v_{t_0,\epsilon}(\psi))\le S_k(-t_0-1)$. Denote $b_{t_0}:=S_k(-t_0-1)\pi n_ka_k$ for simplicity.
Then by \eqref{cal of curvature}, we have
\begin{equation}\label{step3 curvature}
\begin{split}
B+\big(S_k\lambda_m+2b_{t_0}\delta_{m'}\big)\text{Id}_E
\ge
v''_{t_0,\epsilon}(\psi)[\sqrt{-1}\big(\partial\psi\wedge\bar{\partial}\psi\big)\otimes\text{Id}_E, \Lambda_{\omega}]
\end{split}
\end{equation}
holds on $M_k\backslash (\sum\cup\sum_{m'})$.

Let $\lambda_{t_0}:=D''[(1-v'_{t_0,\epsilon}(\psi))\tilde{f}_{t_0}]$. Then we know that $\lambda_{t_0}$ is well defined on $M_k$, $D''\lambda_{t_0}=0$ and
\begin{equation}\nonumber
\begin{split}
\lambda_{t_0}&=-v''_{t_0,\epsilon}(\psi)\bar{\partial}\psi\wedge\tilde{f}_{t_0}
+\big(1-v'_{t_0,\epsilon}(\psi)\big)D''\tilde{f}_{t_0}\\
&=\lambda_{1,t_0}+\lambda_{2,t_0},
\end{split}
\end{equation}
where $\lambda_{1,t_0}:=-v''_{t_0,\epsilon}(\psi)\bar{\partial}\psi\wedge\tilde{f}_{t_0}$ and
$\lambda_{2,t_0}:=\big(1-v'_{t_0,\epsilon}(\psi)\big)D''\tilde{f}_{t_0}$. Note that
$$\text{supp}\lambda_{1,t_0}\subset \{-t_0-1+\epsilon<\psi<-t_0-\epsilon\}$$
and
$$\text{supp}\lambda_{2,t_0}\subset \{\psi<-t_0-\epsilon\}.$$
It follows from inequality \eqref{step3 curvature} that we have
\begin{equation}\nonumber
\begin{split}
&\langle \big(B+(S_k\lambda_m+2b_{t_0}\delta_{m'})\text{Id}_E\big)^{-1}\lambda_{1,t_0},\lambda_{1,t_0} \rangle_{\omega,\tilde{h}_{m,m'}}|_{M_k\backslash (\sum\cup\sum_{m'})}\\
\le &v''_{t_0,\epsilon}(\psi)|\tilde{f}_{t_0}|^2_{\omega,h_m}e^{-\Phi_{m'}}e^{-2\pi
n_k\delta_{m'}\Upsilon-\phi}.
\end{split}
\end{equation}
Then we know that
\begin{equation}\label{step3 I1 1}
\begin{split}
&\int_{M_k\backslash(\sum\cup \sum_{m'})}\langle \big(B+(S_k\lambda_m+2b_{t_0}\delta_{m'})\text{Id}_E\big)^{-1}\lambda_{1,t_0},\lambda_{1,t_0} \rangle_{\tilde{h}_{m,m'}}dV_{M,\omega}\\
\le &\int_{M_k\backslash(\sum\cup \sum_{m'})}v''_{t_0,\epsilon}(\psi)|\tilde{f}_{t_0}|^2_{\omega,h_m}e^{-\Phi_{m'}}e^{-(2\pi
n_k\delta_{m'})\Upsilon-\phi}dV_{M,\omega}\\
\le& I_{1,m',t_0,\epsilon}:=\sup_{M_k} e^{-\phi}\int_{M_k}v''_{t_0,\epsilon}(\psi)|\tilde{f}_{t_0}|^2_{\omega,h}e^{-\varphi-\psi}e^{-2\pi
n_k\delta_{m'}\Upsilon}dV_{M,\omega}.
\end{split}
\end{equation}
Note that
$$|\tilde{f}_{t_0}|^2_{\omega,h}|_{U}
=|\sum_{i=1}^{N}\sqrt{\xi_i}\sqrt{\xi_i}\tilde{f}_{i,t_0}|^2_{\omega,h}
\le (\sum_{i=1}^{N}\xi_i)(\sum_{i=1}^{N}\xi_i|\tilde{f}_{i,t_0}|^2_{\omega,h})
=\sum_{i=1}^{N}\xi_i|\tilde{f}_{i,t_0}|^2_{\omega,h}
$$
and when $t_0$ is big enough, we have $\big(U_i\cap\{-t_0-1+\epsilon<\psi<-t_0-\epsilon\}\big)\subset U$, for each $i=1,\ldots,N$.

Note that $\{\Upsilon=-\infty\}=H\subset \sum=\{\psi=-\infty\}$ and then $\Upsilon$ is smooth on $M\backslash \sum$. As $\psi$ has neat analytic singularities, we know that $e^{\psi}$ is smooth on $M$. Hence the set $\{e^{-t_0-1}\le e^\psi \le e^{-t_0}\}$ is closed subset of $M$. It follows from relatively compactness of $U_i$ and $\Upsilon$ is smooth on $M\backslash \sum$ that we know  $-\Upsilon$ is upper bounded by some real number $\gamma_{t_0,i}$ on $U_i\cap\{e^{-t_0-1}\le e^\psi \le e^{-t_0}\}$. Denote $\gamma_{t_0}=\sup_{i=1,\ldots,N}\gamma_{t_0,i}$.

For fixed $t_0$, we can always find $m_{t_0}$ big enough such that when $m'>m_{t_0}$, $e^{2\pi n_k\gamma_{t_0}\delta_{m'}}<(1+\tau)$ for any given $\tau>0$.
Then it follows from the definition of $v''_{t_0,\epsilon}(t)$, $\phi$, $h_m\le h$ for any $m\ge 1$, $\text{supp}\lambda_{1,t_0}\subset \{-t_0-1+\epsilon<\psi<-t_0-\epsilon\}$, inequality \eqref{step3 I1 1} that when $m'$ is big enough, we have
\begin{equation}\label{equ for I1m't0}
\begin{split}
I_{1,m',t_0,\epsilon}
\le& \frac{e^{2\pi
n_k\gamma_{t_0}\delta_{m'}}}{{1-4\epsilon}}(\sup_{t\ge t_0} e^{-u_k(t)}) \sum_{i=1}^{N}\int_{U_i\cap\{-t_0-1+\epsilon<\psi<-t_0-\epsilon\}}\xi_i|\tilde{f}_{i,t_0}|^2_{\omega,h}e^{-\varphi-\psi}dV_{M,\omega}\\
\le&  \frac{(1+\tau)}{{1-4\epsilon}}(\sup_{t\ge t_0} e^{-u_k(t)} ) \sum_{i=1}^{N}\int_{U_i\cap\{-t_0-1+\epsilon<\psi<-t_0-\epsilon\}}\xi_i|\tilde{f}_{i,t_0}|^2_{\omega,h}e^{-\varphi-\psi}dV_{M,\omega}.
\end{split}
\end{equation}
Denote $$I_{1,t_0}:=(\sup_{t\ge t_0} e^{-u_k(t)}) \sum_{i=1}^{N}\int_{U_i\cap\{-t_0-1+\epsilon<\psi<-t_0-\epsilon\}}\xi_i|\tilde{f}_{i,t_0}|^2_{\omega,h}e^{-\varphi-\psi}dV_{M,\omega}.$$
Then we have
\begin{equation}\label{I1m't0less Im't0}
\limsup_{m'\to +\infty}  I_{1,m',t_0,\epsilon}\le \frac{(1+\tau)}{{1-4\epsilon}} I_{1,t_0}.
\end{equation}
It follows from inequality \eqref{step1part3 6} that we know
\begin{equation}\nonumber
\begin{split}
&\limsup_{t_0\to +\infty} I_{1,t_0}\\
\le& (\sup_{t\ge t_0} e^{-u_k(t)} ) \limsup_{t_0\to +\infty} \sum_{i=1}^{N}\int_{U_i\cap\{-t_0-1+\epsilon<\psi<-t_0-\epsilon\}}\xi_i|\tilde{f}_{i,t_0}|^2_{\omega,h}e^{-\varphi-\psi}dV_{M,\omega}\\
\le & (\sup_{t\ge t_0} e^{-u_k(t)} )\sum_{i=1}^{N}
\int_{U_i\cap Y^0}\xi_i|f|^2_{\omega,h}e^{-\varphi}dV_{M,\omega}[\psi]\\
\le & (\sup_{t\ge t_0} e^{-u_k(t)} )
\int_{Y^0}|f|^2_{\omega,h}e^{-\varphi}dV_{M,\omega}[\psi].
\end{split}
\end{equation}
Then by the definition of $u_k$, when $t_0$ is big enough, we have
\begin{equation}\label{step3 I1 2}
\begin{split}
\limsup_{t_0\to +\infty}  I_{1,t_0}
\le  \big(\frac{1}{\delta}c_k(T)e^{-T}+\int^{+\infty}_T c_k(t_1)e^{-t_1}dt_1\big)
\int_{Y^0}|f|^2_{\omega,h}e^{-\varphi}dV_{M,\omega}[\psi].
\end{split}
\end{equation}
By inequality \eqref{equ for I1m't0}, we have
\begin{equation}\label{step3 I1m't0 2}
\begin{split}
&\limsup_{t_0\to +\infty}(\limsup_{m'\to +\infty}  I_{1,m',t_0,\epsilon})\\
\le &\frac{(1+\tau)}{1-4\epsilon}\limsup_{t_0\to +\infty}  I_{1,t_0}\\
\le  &\frac{(1+\tau)}{1-4\epsilon}\big(\frac{1}{\delta}c_k(T)e^{-T}+\int^{+\infty}_T c_k(t_1)e^{-t_1}dt_1\big)
\int_{Y^0}|f|^2_{\omega,h}e^{-\varphi}dV_{M,\omega}[\psi].
\end{split}
\end{equation}

Note that when $m'$ is big enough, we have $2\pi n_k \delta_{m'}<\beta_3$ for any given $\beta_3>0$. Denote
\begin{equation}\label{step3 I2 1}
\begin{split}
I_{2,m',t_0}:=&\int_{M_k\backslash (\sum\cup \sum_{m'})}\langle\lambda_{2,t_0},\lambda_{2,t_0} \rangle_{\tilde{h}_{m,m'}}dV_{M,\omega}\\
\le&\int_{M_k\cap\{\psi<-t_0-\epsilon\}}|D'' \tilde{f}_{t_0}|^2_{\omega,h_m}e^{-\Phi_{m'}}e^{-(2\pi
n_k\delta_{m'})\Upsilon-\phi}dV_{M,\omega}\\
\le & \big(\sup_{t\ge t_0} e^{-u_k(t)}\big)
\int_{M_k\cap\{\psi<-t_0\}}|D'' \tilde{f}_{t_0}|^2_{\omega,h}e^{-\varphi-\psi-\beta_3\Upsilon}dV_{M,\omega},
\end{split}
\end{equation}
the last inequality holds only for $m'$ is big enough.
It follows from equality \eqref{step1part4 1} and Cauchy-Schwarz inequality that when $t_0$ is big enough,
\begin{equation}\label{step3 I2 2}
\begin{split}
I_{2,m',t_0}\le& C_8\sum_{1\le i,j\le N}
\int_{U_i\cap U_j\cap\{\psi<-t_0\}}|\tilde{f}_{i,t_0}-\tilde{f}_{j,t_0}|^2_{\omega,h}e^{-\varphi-\psi-\beta_3\Upsilon}dV_{M,\omega},
\end{split}
\end{equation}
where $C_8>0$ is a real number independent of $t_0$.

For any $1\le i,j\le N$, we denote
\begin{equation}\label{step3 I2 3}
\begin{split}
I_{i,j,t_0}:=\int_{U_i\cap U_j\cap\{\psi<-t_0\}}|\tilde{f}_{i,t_0}-\tilde{f}_{j,t_0}|^2_{\omega,h}e^{-\varphi-\psi-\beta_3\Upsilon}dV_{M,\omega}.
\end{split}
\end{equation}
Next, we will show that $I_{i,j,t_0}\le C'e^{-2\beta_0 t_0}$ for some constant $C'>0$ independent of $t_0$.

It follows from inequality \eqref{step1part3 1}, the definition of $\hat{c}_1(t)$ and $\hat{c}_0(t)$ and  $\Gamma_i+m_i\le \psi$ on $\Omega_i$ that for any $1\le i \le N$, we have
\begin{equation}\nonumber
\begin{split}
\int_{\Omega_i\cap\{\psi<-t_0\}}\hat{c}_0(-\psi)|\tilde{f}_{i,t_0}|^2_{\omega,h(\text{det}h)^{\beta}}e^{-(1+\beta r)\varphi}dV_{M,\omega}\le C_9,
\end{split}
\end{equation}
where $C_9>0$ is a real number independent of $t_0$.
Recall that
$\hat{c}_0(t):=c(T_1)e^{(1-\beta_2)(t-T_1)}$, then we have
\begin{equation}\label{step3 I2 KEYESTIMATE}
\begin{split}
\int_{\Omega_i\cap\{\psi<-t_0\}}|\tilde{f}_{i,t_0}|^2_{\omega,h(\text{det}h)^{\beta}}e^{-(1+\beta r)\varphi}e^{-(1-\beta_2)\psi}dV_{M,\omega}\le \tilde{C}_9,
\end{split}
\end{equation}
where $\tilde{C}_9>0$ is a real number independent of $t_0$.

For fixed $i,j$, let $\{e_1,\cdots,e_r\}$ be a holomorphic frame on $ E|_{U_i\cap U_j}$. It follows from Lemma \ref{existence of bounded tra} that there exists a local frame $\{\zeta_1,\cdots,\zeta_r\}$ of $E|_{U_i\cap U_j}$ such that the local expression of $ he^{-\varphi}$ is a diagonal matrix with diagonal element $\det (h)e^{-r\varphi}$  and the transition matrix $B^{-1}$ from  $\{e_1,\cdots,e_r\}$ to  $\{
\zeta_1,\cdots,\zeta_r\}$ satisfies that each element $b_{i,j}(z)$ of $B$ is a bounded function on $U_i\cap U_j$. Let $dw$ be a local frame of $K_M|_{U_i\cap U_j}$. Then we can assume that

$$\tilde{f}_{i,t_0}=\sum_{p=1}^{r} F_{i,t_0,p}\zeta_p\otimes dw \text{ and } \tilde{f}_{j,t_0}=\sum_{p=1}^{r} F_{j,t_0,p}\zeta_p\otimes dw  \text{ on } U_i\cap U_j,$$
where $F_{i,t_0,p}$ and $F_{j,t_0,p}$ are measurable functions on $U_i\cap U_j$. Then \eqref{step3 I2 3} becomes
\begin{equation}\label{step3 I2 4}
\begin{split}
I_{i,j,t_0}=\int_{U_i\cap U_j\cap\{\psi<-t_0\}}\sum_{p=1}^{r}|F_{i,t_0,p}-F_{j,t_0,p}|^2(\det h)
e^{-r\varphi-\psi-\beta_3\Upsilon}dw\wedge d\bar{w}.
\end{split}
\end{equation}
Since the case is local and $he^{-\varphi}$ is a singular metric on $E|_{U_i\cap U_j}$ and  locally lower bounded, we can assume that all eigenvalues of $he^{-\varphi}$ are greater than $1$.
By the inequality \eqref{step3 I2 KEYESTIMATE} and the constructions of local frame $\{\zeta_1,\cdots,\zeta_r\}$, for any $1\le i\le N$ and any $1\le p\le r$, we have
\begin{equation}\label{step3 I2 KEYESTIMATE2}
\begin{split}
&\int_{U_i\cap\{\psi<-t_0\}}|F_{i,t_0,p}|^2(\det h)^{1+\beta}e^{-(1+\beta )r\varphi-(1-\beta_2)\psi}dw\wedge d\bar{w}\\
\le & \int_{\Omega_i\cap\{\psi<-t_0\}}|\tilde{f}_{i,t_0}|^2_{\omega,h}e^{-\varphi}(\text{det}h)^{\beta}e^{-\beta r\varphi-(1-\beta_2)\psi} dV_{M,\omega}
\le \tilde{C}_9.
\end{split}
\end{equation}
It follows from H\"older inequality that
\begin{equation}\label{step3 I2 5}
\begin{split}
I_{i,j,t_0}=&\sum_{p=1}^{r}\int_{U_i\cap U_j\cap\{\psi<-t_0\}}|F_{i,t_0,p}-F_{j,t_0,p}|^2(\det h)e^{-r\varphi-\psi-\beta_3\Upsilon}dw\wedge d\bar{w}\\
\le &\sum_{p=1}^{r}\big(\int_{U_i\cap U_j\cap\{\psi<-t_0\}}|F_{i,t_0,p}-F_{j,t_0,p}|^2(\det h)^{1+\beta}e^{-(1+\beta)r\varphi-(1-\beta_2)\psi}dw\wedge d\bar{w}\big)^{\frac{1}{1+\beta}}\times\\
&\big(\int_{U_i\cap U_j\cap\{\psi<-t_0\}}|F_{i,t_0,p}-F_{j,t_0,p}|^2e^{-(1+\beta_2\frac{1}{\beta})\psi-\beta_3\frac{1+\beta}{\beta}\Upsilon}dw\wedge d\bar{w}\big)^{\frac{\beta}{1+\beta}}\\
\le& C_{10}\sum_{p=1}^{r}
\big(\int_{U_i\cap U_j\cap\{\psi<-t_0\}}|F_{i,t_0,p}-F_{j,t_0,p}|^2e^{-(1+\beta_2\frac{1}{\beta})\psi-\beta_3\frac{1+\beta}{\beta}\Upsilon}dw\wedge d\bar{w}\big)^{\frac{\beta}{1+\beta}},
\end{split}
\end{equation}
when $t_0$ is big enough and $C_{10}>0$ is a real number independent of $t_0$.

 We would like to estimate the last integral by estimating its pull back under the morphism $\mu$. We cover $\mu^{-1}(U_i\cap U_j)\cap \{\psi\circ \mu <-t_0\}$ by a finite number of coordinate balls $W$ as we did in the Step 1 of the proof of Proposition \ref{p:inte}. Let $dw$ be a local frame of $K_{\tilde{M}}|_{\mu^{-1}(U_i\cap U_j)}$. Assume that under the local frame  $\{e_1\circ \mu,\cdots,e_r\circ \mu\}$, we can write
 $$\tilde{f}_{i,t_0}\circ \mu=\sum_{p=1}^{r} f_{i,t_0,p}(e_p\circ \mu)\otimes dw \text{ and } \tilde{f}_{j,t_0}\circ \mu=\sum_{p=1}^{r} f_{j,t_0,p}(e_p\circ \mu)\otimes dw\text{ on } W.$$
Then inequalities \eqref{step1part4 2} and \eqref{step1part4 3} show that for any $1\le i\le N$ and any $1\le p\le r$,
\begin{equation}\label{step3 I2 6}
  |f_{i,t_0,p}-f_{j,t_0,p}|^2|_{W_{i,j,t_0}}\le C_7e^{\beta_1 t_0 }\prod_{l \in \kappa}|w_l|^2
\end{equation}
when $\kappa \neq \emptyset$ and $t$ is big enough, and
\begin{equation}\label{step3 I2 7}
  |f_{i,t_0,p}-f_{j,t_0,p}|^2|_{W_{i,j,t_0}}\le C_7e^{\beta_1 t_0 }
\end{equation}
when $\kappa= \emptyset$ and $t_0$ is big enough, where $W_{i,j,t_0}=W\cap \mu^{-1}(U_i\cap U_j)\cap\{\psi\circ \mu <-t_0\}$ and  $C_7>0$ is a real number independent of $t_0$.

Note that for any $1\le i\le N$, we have
 $$\tilde{f}_{i,t_0}\circ \mu=(F_{i,t_0,1}\circ \mu,\cdots,F_{i,t_0,r}\circ \mu)^T=(B\circ \mu) (f_{i,t_0,1},\cdots,f_{i,t_0,r})^T,$$
 where $T$ means transposition.
Then it follows from inequalities \eqref{step3 I2 6}, \eqref{step3 I2 7} and $B\circ \mu$ is a bounded matrix on $W$ (shrink $W$ if necessary) that  we have
\begin{equation}\label{step3 I2 8}
  |F_{i,t,p}
  \circ\mu-F_{j,t_0,p} \circ\mu|^2|_{W_{i,j,t_0}}\le \hat{C}_7e^{\beta_1 t_0 }\prod_{l \in \kappa}|w_l|^2
\end{equation}
when $\kappa \neq \emptyset$ and $t_0$ is big enough, and
\begin{equation}\label{step3 I2 9}
  |F_{i,t_0,p}\circ\mu-F_{j,t_0,p}\circ\mu|^2|_{W_{i,j,t_0}}\le \hat{C}_7e^{\beta_1 t_0 }
\end{equation}
when $\kappa= \emptyset$ and $t_0$ is big enough, where  $\hat{C}_7>0$ is a real number independent of $t_0$.

By equality \eqref{local description of upsilon} in Remark \ref{Hsnc case}, we can assume that under the local  coordinate $(W;w_1,\ldots,w_n)$, we have
\begin{equation}\label{local description of upsilon 1}
\Upsilon=\log(\prod_{l=1}^{n}|w_l|^{2d_l})+v(w),
\end{equation}
where $d_l$ is nonnegative integer and $v(w)$ is a smooth function on $W$.
It follows from inequalities \eqref{step3 I2 8}, \eqref{step3 I2 9} and \eqref{local description of upsilon 1} that on each $W$,
$$\int_{W_{i,j,t_0}}|F_{i,t_0,p}\circ\mu-F_{j,t_0,p}\circ\mu|^2
e^{-(1+\beta_2\frac{1}{\beta})\psi\circ\mu-\beta_3\frac{1+\beta}{\beta}\Upsilon}|J_\mu|^2dw\wedge d\bar{w}\le C_{11}
\int_{W_{i,j,t_0}}\frac{d\lambda_w}{\prod_{l=1}^{n}|w_l|^{2\alpha_l}},$$
where $\alpha_l:=(\beta_1+\beta_2\frac{1}{\beta})ca_l+\beta_3\frac{1+\beta}{\beta}d_l+(ca_l-b_l)-\lfloor ca_l-b_l\rfloor_+$, $d\lambda_w$ is the Lebesgue measure on $W_{i,j,t_0}$ and
$C_{11}>0$ is a real number independent of $t_0$. Note that
$$(W\cap\{\psi\circ \mu<-t_0\})\subset \cup_{l=1}^n(\{|w_l|<e^{\frac{-t_0-m}{2c|a|}}\}\cap W),$$
where $m:=\inf_{W}\tilde{u}(w)$. Let $\beta_1$ satisfy that
\begin{equation}\label{condition of beta1}
\beta_1< \min_{1\le l \le n}\frac{1-(ca_l-b_l)+\lfloor ca_l-b_l\rfloor_+}{3ca_l}.
\end{equation}
Let $\beta_2=\beta_1\beta$ and $\beta_3<\min\limits_{1\le l \le n}\beta_1\frac{\beta}{1+\beta}\frac{1}{d_l}$. Then we know that $\alpha_l<1$ for any $1\le l \le n$.
Hence we have
\begin{equation}\label{step3 I2 10}
\begin{split}
\int_{W_{i,j,t_0}}\frac{d\lambda_w}{\prod_{l=1}^{n}|w_l|^{2\alpha_l}}
\le&\sum_{l=1}^{n}\int_{\{|w_l|<e^{\frac{-t_0-m}{2c|a|}}\}\cap W}\frac{d\lambda_w}{\prod_{l=1}^{n}|w_l|^{2\alpha_l}}\\
\le& C_{12}\sum_{l=1}^{n}e^{\frac{-(1-\alpha_l)t_0}{c|a|}},
\end{split}
\end{equation}
where  $C_{12}>0$ is a real number independent of $t_0$.
Denote $\beta_0:=\min_{1\le l \le n}\frac{\beta(1-\alpha_l)}{2(1+\beta)c|a|}$. Then
it follows from \eqref{step3 I2 5} and \eqref{step3 I2 10} that we have
\begin{equation}\nonumber
\begin{split}
I_{i,j,t_0}\le \tilde{C}_{13}e^{-2\beta_0 t_0},
\end{split}
\end{equation}
where  $\tilde{C}_{13}>0$ is a real number independent of $t_0$. Then it follows inequality \eqref{step3 I2 2} that we know, when $m'$ is big enough,
\begin{equation}\label{step3 I2 12}
\begin{split}
I_{2,m',t_0}\le C_{13}e^{-2\beta_0 t_0},
\end{split}
\end{equation}
where  $C_{13}>0$ is a real number independent of $t_0$.

\

\textbf{Step 6: solving $\bar{\partial}-$equation with error term.}

\

Given $\tau>0$,
note that $$\langle a_1+a_2,a_1+a_2\rangle\le (1+\tau)\langle a_1,a_1\rangle+(1+\frac{1}{\tau})\langle a_2,a_2\rangle$$
holds for any $a_1,a_2$ in an inner product space $(H,\langle \cdot,\cdot \rangle)$.
It follows from inequality \eqref{step3 curvature} that on $M_k\backslash (\sum\cup\sum_{m'})$, for any $\tau>0$, we have
\begin{equation}\label{Step 4 1}
\begin{split}
&\int_{M_k\backslash (\sum\cup\sum_{m'})} \langle \big(B+(S_k\lambda_m+2b_{t_0}\delta_{m'}+e^{-\beta_0t_0})\text{Id}_E\big)^{-1}\lambda_{t_0},\lambda_{t_0} \rangle_{\omega,\tilde{h}_{m,m'}}dV_{M,\omega}\\
\le & \int_{M_k\backslash(\sum\cup\sum_{m'})}(1+\tau) \langle \big(B+(S_k\lambda_m+2b_{t_0}\delta_{m'}+e^{-\beta_0t_0})\text{Id}_E\big)^{-1}\lambda_{1,t_0},\lambda_{1,t_0} \rangle_{\omega,\tilde{h}_{m,m'}}dV_{M,\omega}\\
&+\int_{M_k\backslash(\sum\cup\sum_{m'})}(1+\frac{1}{\tau}) \langle \big(B+(S_k\lambda_m+2b_{t_0}\delta_{m'}+e^{-\beta_0t_0})\text{Id}_E\big)^{-1}\lambda_{2,t_0},\lambda_{2,t_0} \rangle_{\omega,\tilde{h}_{m,m'}}dV_{M,\omega}\\
\le &(1+\tau)\int_{M_k\backslash(\sum\cup\sum_{m'})} \langle \big(B+(S_k\lambda_m+2b_{t_0}\delta_{m'})\text{Id}_E\big)^{-1}\lambda_{1,t_0},\lambda_{1,t_0} \rangle_{\omega,\tilde{h}_{m,m'}}dV_{M,\omega}\\
+&(1+\frac{1}{\tau})\int_{M_k\backslash(\sum\cup\sum_{m'})} \langle e^{\beta_0t_0}\lambda_{2,t_0},\lambda_{2,t_0} \rangle_{\omega,\tilde{h}_{m,m'}}dV_{M,\omega}\\
= & (1+\tau) I_{1,m',t_0,\epsilon}+(1+\frac{1}{\tau})e^{\beta_0t_0}I_{2,m',t_0}\\
\le &(1+\tau) I_{1,m',t_0,\epsilon}+(1+\frac{1}{\tau})e^{-\beta_0t_0},
\end{split}
\end{equation}
where the last inequality holds because of inequality  \eqref{step3 I2 12}.
By inequalities \eqref{I1m't0less Im't0} and \eqref{step3 I1m't0 2}, we know that for fixed $t_0$, when $m'$ is big, $(1+\tau) I_{1,m',t_0,\epsilon}+(1+\frac{1}{\tau})e^{-\beta_0t_0}$ is finite.

From now on, we fix some $\epsilon \in (0,\frac{1}{8})$.
Recall that $\tilde{h}_{m,m'}=h_me^{-\Phi_{m'}}e^{-2\pi n_k \delta_{m'}\Upsilon}e^{-u_k\big(-v_{t_0,\epsilon}(\psi)\big)}$ on $M_k\backslash (\sum\cup\sum_{m'})$ and let $P_{m,m'}:L^2(M_k\backslash (\sum\cup\sum_{m'}), \wedge^{n,1}T^*M\otimes E, \omega\otimes\tilde{h}_{m,m'})\to \text{Ker} D''$ be the orthogonal projection. Then by Lemma \ref{d-bar equation with error term}, there exist $u_{k,t_0,m,m',\epsilon}\in L^2(M_k\backslash (\sum\cup\sum_{m'}), K_m\otimes E, \omega\otimes\tilde{h}_{m,m'})$ and $\eta_{k,t_0,m,m',\epsilon}\in L^2(M_k\backslash (\sum\cup\sum_{m'}), \wedge^{n,1}T^*M\otimes E, \omega\otimes\tilde{h}_{m,m'})$ such that
\begin{equation}\label{step 4 d-bar equation}
D''u_{k,t_0,m,m',\epsilon}+P_{m,m'}(\sqrt{s\lambda_m+2b_{t_0}\delta_{m'}+e^{-\beta_0t_0}}\eta_{k,t_0,m,m',\epsilon})=\lambda_{t_0}
\end{equation}
and
\begin{equation}\label{step 4 estimate}
\begin{split}
&\int_{M_k\backslash (\sum\cup \sum_{m'})}(\eta+g^{-1})^{-1}|u_{k,t_0,m,m',\epsilon}|^2_{\omega,\tilde{h}_{m,m'}}dV_{M,\omega}
+\int_{M_k\backslash (\sum\cup \sum_{m'})}|\eta_{k,t_0,m,m',\epsilon}|^2_{\omega,\tilde{h}_{m,m'}}dV_{M,\omega}\\
\le & (1+\tau) I_{1,m',t_0,\epsilon}+(1+\frac{1}{\tau})e^{-\beta_0t_0}<+\infty.
\end{split}
\end{equation}
By definition, $(\eta+g^{-1})^{-1}=c_k(-v_{t_0,\epsilon}(\psi))e^{v_{t_0,\epsilon}(\psi)}e^{\phi}$. It follows from inequality \eqref{step 4 estimate} that
\begin{equation}\label{step 4 estimate 1}
\begin{split}
&\int_{M_k\backslash (\sum\cup \sum_{m'})}c_k\big(-v_{t_0,\epsilon}(\psi)\big)e^{v_{t_0,\epsilon}(\psi)-2\pi n_k \delta_{m'}\Upsilon}|u_{k,t_0,m,m',\epsilon}|^2_{\omega,h_m}e^{-\Phi_{m'}}dV_{M,\omega}\\
\le & (1+\tau) I_{1,m',t_0,\epsilon}+(1+\frac{1}{\tau})e^{-\beta_0t_0}.
\end{split}
\end{equation}
and
\begin{equation}\label{step 4 estimate 2}
\begin{split}
&\int_{M_k\backslash (\sum\cup \sum_{m'})}|\eta_{k,t_0,m,m',\epsilon}|^2_{\omega,h_m}e^{-\Phi_{m'}-2\pi n_k \delta_{m'}\Upsilon-\phi}dV_{M,\omega}\\
\le  &(1+\tau) I_{1,m',t_0,\epsilon}+(1+\frac{1}{\tau})e^{-\beta_0t_0}.
\end{split}
\end{equation}

Note that $v_{t_0,\epsilon}(\psi)$ is bounded on $M_k$ and $c_k(t)e^{-t}$ is decreasing near $+\infty$, we know that $c(-v_{t_0,\epsilon}(\psi))e^{v_{t_0,\epsilon}(\psi)}$ has positive lower bound on $M_k$. We also have $e^{-\phi}=e^{-u_k(-v_{t_0,\epsilon}(\psi))}$ has positive lower bound on $M_k$. As $\Upsilon$ and $\Phi_{m'}$ are upper-bounded on $M_k$, $e^{-\Upsilon}$ and $e^{-\Phi_{m'}}$ also have positive lower bound on $M_k$. By inequalities \eqref{step 4 estimate 1} and \eqref{step 4 estimate 2}, we know that
$$u_{k,t_0,m,m',\epsilon}\in L^2(M_k\backslash (\sum\cup\sum_{m'}), K_M\otimes E, \omega\otimes h_m)$$
and $$\eta_{k,t_0,m,m',\epsilon}\in L^2(M_k\backslash (\sum\cup\sum_{m'}), \wedge^{n,1}T^*M\otimes E, \omega\otimes h_m).$$
By Lemma \ref{extension of equality} and equality \eqref{step 4 d-bar equation}, we know that
\begin{equation}\label{step 4 d-bar equation 2}
D''u_{k,t_0,m,m',\epsilon}+P_{m,m'}(\sqrt{s\lambda_m+2b_{t_0}\delta_{m'}+e^{-\beta_0t_0}}\eta_{k,t_0,m,m',\epsilon})=\lambda_{t_0}
\end{equation}
holds on $M_k$. Inequalities \eqref{step 4 estimate 1}, \eqref{step3 I1 2}, \eqref{step3 I2 12} and $\Upsilon$ is upper bounded on $M_k$ imply that for fixed $t_0$, when $m'$ is big, we have
\begin{equation}\label{step 4 estimate 5}
\begin{split}
&\int_{M_k}c_k\big(-v_{t_0,\epsilon}(\psi)\big)e^{v_{t_0,\epsilon}(\psi)}|u_{k,t_0,m,m',\epsilon}|^2_{\omega,h_m}e^{-\Phi_{m'}}dV_{M,\omega}\\
\le  &e^{2\pi n_k \delta_{m'}M_{\Upsilon}}\bigg((1+\tau) I_{1,m',t_0,\epsilon}+(1+\frac{1}{\tau})e^{-\beta_0t_0}\bigg)\\
= &\tilde{M}_{m'}\bigg((1+\tau) I_{1,m',t_0,\epsilon}+(1+\frac{1}{\tau})e^{-\beta_0t_0}\bigg)\\
< & +\infty,
\end{split}
\end{equation}
where $M_{\Upsilon}:=\sup_{M_k}\Upsilon$ and we denote $e^{2\pi n_k \delta_{m'}M_{\Upsilon}}$ by $\tilde{M}_{m'}$ for simplicity. We note that $\tilde{M}_{m'}\to 1$ as $m'\to +\infty$.
For fixed $t_0$, when $m'$ is big, we also have
\begin{equation}\label{step 4 estimate 4}
\begin{split}
&\int_{M_k}|\eta_{k,t_0,m,m',\epsilon}|^2_{\omega,h_m}e^{-\Phi_{m'}}e^{-2\pi n_k \delta_{m'}\Upsilon-\phi}dV_{M,\omega}\\
\le  &(1+\tau) I_{1,m',t_0,\epsilon}+(1+\frac{1}{\tau})e^{-\beta_0t_0}\\
< & +\infty.
\end{split}
\end{equation}

\

\textbf{Step 7: when $m\to +\infty$.}

\

In Step 7, note that $t_0$ is fixed and $m'$ is fixed and big enough.

By the construction of $v_{t_0,\epsilon}(\psi)$ and $c_k(t)\in\mathcal{G}_{T,\delta}$, we know that $c_k\big(-v_{t_0,\epsilon}(\psi)\big)e^{v_{t_0,\epsilon}(\psi)}$ has positive upper and lower bound on $M_k$.
It follows from inequality \eqref{step 4 estimate 5} and $c_k\big(-v_{t_0,\epsilon}(\psi)\big)e^{v_{t_0,\epsilon}(\psi)}>0$ on $M_k$ that we have
$$\sup_m\int_{M_k}|u_{k,t_0,m,m',\epsilon}|^2_{\omega,h_m}e^{-\Phi_{m'}}dV_{M,\omega}<+\infty.$$
As $h_1\le h_m$, we have
\begin{equation}\label{Step 5 estimate for uh1}
\sup_m\int_{M_k}|u_{k,t_0,m,m',\epsilon}|^2_{\omega,h_1}e^{-\Phi_{m'}}dV_{M,\omega}<+\infty.
\end{equation}
Since the closed unit ball of Hilbert space is
weakly compact, we can extract a subsequence of $\{u_{k,t_0,m,m',\epsilon}\}$ (also denoted by $\{u_{k,t_0,m,m',\epsilon}\}$) weakly
convergent to $u_{k,t_0,m',\epsilon}$ in $L^2(M_k, K_M\otimes E, \omega\otimes h_1e^{-\Phi_{m'}})$ as $m\to+\infty$. As $c_k\big(-v_{t_0,\epsilon}(\psi)\big)e^{v_{t_0,\epsilon}(\psi)}$ is upper bounded on $M_k$, hence we know that $\sqrt{c_k\big(-v_{t_0,\epsilon}(\psi)\big)e^{v_{t_0,\epsilon}(\psi)}}u_{k,t_0,m,m',\epsilon}$ weakly
converges to $\sqrt{c_k\big(-v_{t_0,\epsilon}(\psi)\big)e^{v_{t_0,\epsilon}(\psi)}}u_{k,t_0,m',\epsilon}$ in $L^2(M_k, K_M\otimes E, \omega\otimes h_1e^{-\Phi_{m'}})$ as $m\to+\infty$.

For fixed $i\in\mathbb{Z}_{\ge 1}$, as $h_1$ and $h_i$ are both $C^2$ smooth hermitian metrics on $M_{k+1}$ and $M_k\Subset M_{k+1}\Subset X$, we know $h_i\le C_i h_1$ for some $C_i\ge 1$ on $\overline{M_k}$. It follows from Lemma \ref{equiv of weak convergence} that we know $\sqrt{c_k\big(-v_{t_0,\epsilon}(\psi)\big)e^{v_{t_0,\epsilon}(\psi)}}u_{k,t_0,m,m',\epsilon}$ weakly
converges to $\sqrt{c_k\big(-v_{t_0,\epsilon}(\psi)\big)e^{v_{t_0,\epsilon}(\psi)}}u_{k,t_0,m',\epsilon}$ in $L^2(M_k, K_M\otimes E, \omega\otimes h_ie^{-\Phi_{m'}})$ as $m\to+\infty$.
Then we have
\begin{equation}\label{step 5 estimate 1}
\begin{split}
&\int_{M_k}c_k\big(-v_{t_0,\epsilon}(\psi)\big)e^{v_{t_0,\epsilon}(\psi)}|u_{k,t_0,m',\epsilon}|^2_{\omega,h_i}e^{-\Phi_{m'}}dV_{M,\omega}\\
\le
&\liminf_{m\to +\infty}\int_{M_k}c_k\big(-v_{t_0,\epsilon}(\psi)\big)e^{v_{t_0,\epsilon}(\psi)}|u_{k,t_0,m,m',\epsilon}|^2_{\omega,h_i}e^{-\Phi_{m'}}dV_{M,\omega}\\
\le
&\liminf_{m\to +\infty}\int_{M_k}c_k\big(-v_{t_0,\epsilon}(\psi)\big)e^{v_{t_0,\epsilon}(\psi)}|u_{k,t_0,m,m',\epsilon}|^2_{\omega,h_m}e^{-\Phi_{m'}}dV_{M,\omega}\\
\le  &\liminf_{m\to +\infty}\tilde{M}_{m'}\bigg((1+\tau) I_{1,m',t_0,\epsilon}+(1+\frac{1}{\tau})e^{-\beta_0t_0}\bigg)\\
=& \tilde{M}_{m'}\bigg((1+\tau) I_{1,m',t_0,\epsilon}+(1+\frac{1}{\tau})e^{-\beta_0t_0}\bigg)\\
< & +\infty.
\end{split}
\end{equation}
Let $i\to +\infty$ in inequality \eqref{step 5 estimate 1}, by monotone convergence theorem, we have
\begin{equation}\label{step 5 estimate 2}
\begin{split}
&\int_{M_k}c_k\big(-v_{t_0,\epsilon}(\psi)\big)e^{v_{t_0,\epsilon}(\psi)}|u_{k,t_0,m',\epsilon}|^2_{\omega,h}e^{-\Phi_{m'}}dV_{M,\omega}\\
\le
&\tilde{M}_{m'}\bigg[(1+\tau) I_{1,m',t_0,\epsilon}+(1+\frac{1}{\tau})e^{-\beta_0t_0}\bigg]\\
< & +\infty.
\end{split}
\end{equation}

Recall that $\tilde{h}_{m,m'}=h_me^{-\Phi_{m'}}e^{-2\pi n_k \delta_{m'}\Upsilon}e^{-u_k\big(-v_{t_0,\epsilon}(\psi)\big)}$.
It follows from inequality \eqref{step 4 estimate 4} that we have
$$\sup_m\int_{M_k}|\eta_{k,t_0,m,m',\epsilon}|^2_{\omega,\tilde{h}_{m,m'}}dV_{M,\omega}<+\infty.$$
As $h_1\le h_m$, we have
$$\sup_m\int_{M_k}|\eta_{k,t_0,m,m',\epsilon}|^2_{\omega,\tilde{h}_{1,m'}}dV_{M,\omega}<+\infty.$$
Since the closed unit ball of Hilbert space is
weakly compact, we can extract a subsequence of $\{\eta_{k,t_0,m,m',\epsilon}\}$ (also denoted by $\{\eta_{k,t_0,m,m',\epsilon}\}_{m}$) weakly convergent to $\eta_{k,t_0,m',\epsilon}$ in
$L^2(M_k, \wedge^{n,1}T^*M\otimes E, \omega\otimes \tilde{h}_{1,m'})$ as $m\to+\infty$.

For fixed $i\in\mathbb{Z}_{\ge 1}$, as $h_1$ and $h_i$ are both $C^2$ smooth hermitian metrics on $M_{k+1}$ and $M_k\Subset X$. It follows from Lemma \ref{equiv of weak convergence} that we know  for any $i\ge1$, $\{\eta_{k,t_0,m,m',\epsilon}\}$ also weakly converges to $\eta_{k,t_0,m',\epsilon}$ in
$L^2(M_k, \wedge^{n,1}T^*M\otimes E, \omega\otimes \tilde{h}_{i,m'})$ as $m\to+\infty$. It follows from inequality \eqref{step 4 estimate 4} that for any fixed $i$, we have
\begin{equation}\label{step 5 eta 1}
\begin{split}
&\int_{M_k}|\eta_{k,t_0,m',\epsilon}|^2_{\omega,h_i}e^{-\Phi_{m'}}e^{-2\pi n_k \delta_{m'}\Upsilon-\phi}dV_{M,\omega}\\
\le &\liminf_{m\to+\infty} \int_{M_k}|\eta_{k,t_0,m,m',\epsilon}|^2_{\omega,h_i}e^{-\Phi_{m'}}e^{-2\pi n_k \delta_{m'}\Upsilon-\phi}dV_{M,\omega}\\
\le & \liminf_{m\to+\infty} \int_{M_k}|\eta_{k,t_0,m,m',\epsilon}|^2_{\omega,h_m}e^{-\Phi_{m'}}e^{-2\pi n_k \delta_{m'}\Upsilon-\phi}dV_{M,\omega}\\
\le  &(1+\tau) I_{1,m',t_0,\epsilon}+(1+\frac{1}{\tau})e^{-\beta_0t_0}\\
< & +\infty.
\end{split}
\end{equation}
Letting $i\to +\infty$ in inequality \eqref{step 5 eta 1}, by monotone convergence theorem, we have
\begin{equation}\label{step 5 eta 2}
\begin{split}
&\int_{M_k}|\eta_{k,t_0,m',\epsilon}|^2_{\omega,h}e^{-\Phi_{m'}}e^{-2\pi n_k \delta_{m'}\Upsilon-\phi}dV_{M,\omega}\\
\le &\lim_{i\to+\infty}\int_{M_k}|\eta_{k,t_0,m',\epsilon}|^2_{\omega,h_i}e^{-\Phi_{m'}}e^{-2\pi n_k \delta_{m'}\Upsilon-\phi}dV_{M,\omega}\\
\le  &(1+\tau) I_{1,m',t_0,\epsilon}+(1+\frac{1}{\tau})e^{-\beta_0t_0}\\
< & +\infty.
\end{split}
\end{equation}
Note that $S_k\lambda_m+2b_{t_0}\delta_{m'}+e^{-\beta_0t_0} \le C_{14}\lambda+\tilde{C}_{14}$ on $M_K$ and $\lambda$ is continuous on $\overline{M_k}$. It follows from Lemma \ref{weakly convergence} that we know $\sqrt{S_k\lambda_m+2b_{t_0}\delta_{m'}+e^{-\beta_0t_0}}\eta_{k,t_0,m,m',\epsilon}$ weakly converges to $\sqrt{2b_{t_0}\delta_{m'}+e^{-\beta_0t_0}}\eta_{k,t_0,m',\epsilon}$ as $m\to+\infty$ in  $L^2(M_k, \wedge^{n,1}T^*M\otimes E, \omega\otimes \tilde{h}_{1,m'})$. And we also have
$$\sup_m\int_{M_k}(S_k\lambda_m+2b_{t_0}\delta_{m'}+e^{-\beta_0t_0})|\eta_{k,t_0,m,m',\epsilon}|^2_{\omega,\tilde{h}_{m,m'}}dV_{M,\omega}<+\infty.$$

Denote $P_{m'}:L^2(M_k\backslash (\sum\cup \sum_{m'}), \wedge^{n,1}T^*M\otimes E, \omega\otimes\tilde{h}_{m'})\to \text{Ker} D''$ be the orthogonal projection where $\tilde{h}_{m'}=he^{-\Phi_{m'}}e^{-2\pi n_k \delta_{m'}\Upsilon}e^{-u_k\big(-v_{t_0,\epsilon}(\psi)\big)}$. It follows from Lemma \ref{weakly converge lemma} that we know that there exists a subsequence of $\sqrt{S_k\lambda_m+2b_{t_0}\delta_{m'}+e^{-\beta_0t_0} }\eta_{k,t_0,m,m',\epsilon}$ (also denoted by $\{\sqrt{S_k\lambda_m+2b_{t_0}\delta_{m'}+e^{-\beta_0t_0} }\eta_{k,t_0,m,m',\epsilon}\}_m$) weakly converges to some $\tilde{\eta}_{k,t_0,m',\epsilon}$ as $m\to+\infty$ and $P_{m,m'}(\sqrt{S_k\lambda_m+2b_{t_0}\delta_{m'}+e^{-\beta_0t_0}}\eta_{k,t_0,m,m',\epsilon})$ weakly converges to $P_{m'}(\tilde{\eta}_{k,t_0,m',\epsilon})$ in $L^2(M_k, \wedge^{n,1}T^*M\otimes E, \omega\otimes \tilde{h}_{1,m'})$ as $m\to+\infty$. By the uniqueness of weak limit, we know that
$\tilde{\eta}_{k,t_0,m',\epsilon}=\sqrt{2b_{t_0}\delta_{m'}+e^{-\beta_0t_0} }\eta_{k,t_0,m',\epsilon}$  and then $P_{m'}(\tilde{\eta}_{k,t_0,m',\epsilon})=P_{m'}(\sqrt{2b_{t_0}\delta_{m'}+e^{-\beta_0t_0} })\eta_{k,t_0,m',\epsilon})$.

Let $m\to +\infty$ in equality \eqref{step 4 d-bar equation 2}, we have
\begin{equation}\label{step 5 d-bar equation}
D''u_{k,t_0,m',\epsilon}+P_{m'}(\sqrt{2b_{t_0}\delta_{m'}+e^{-\beta_0t_0} }\eta_{k,t_0,m',\epsilon})=\lambda_{t_0}.
\end{equation}

 \

\textbf{Step 8: when $m'\to +\infty$.}

\

In Step 8, note that $t_0$ is fixed.

By the construction of $v_{t_0,\epsilon}(\psi)$ and $c_k(t)\in\mathcal{G}_{T,\delta}$, we know that $c_k\big(-v_{t_0,\epsilon}(\psi)\big)e^{v_{t_0,\epsilon}(\psi)}$ has positive upper and lower bound on $M_k$.
It follows from inequalities \eqref{step 5 estimate 2}, \eqref{step3 I1 2}, \eqref{step3 I2 12} and $c_k\big(-v_{t_0,\epsilon}(\psi)\big)e^{v_{t_0,\epsilon}(\psi)}>0$ on $M_k$ that we have
$$\sup_{m'}\int_{M_k}|u_{k,t_0,m',\epsilon}|^2_{\omega,h}e^{-\Phi_{m'}}dV_{M,\omega}<+\infty.$$
Note that $\Phi_{m'}$ is a locally upper-bounded function which is decreasing with respect to $m'$ and converges to $\Phi$ as $m'\to +\infty$. Then $\Phi_{m'}$ is uniformly bounded above with respect to $m'$,  we have
\begin{equation}\label{Step 5 estimate for uh1}
\sup_{m'}\int_{M_k}|u_{k,t_0,m',\epsilon}|^2_{\omega,h}dV_{M,\omega}<+\infty.
\end{equation}
Since the closed unit ball of Hilbert space is
weakly compact, by \eqref{Step 5 estimate for uh1}, we know that there exists a subsequence of $\{u_{k,t_0,m',\epsilon}\}$ (also denoted by $\{u_{k,t_0,m',\epsilon}\}$) weakly
convergent to $u_{k,t_0,\epsilon}$ in $L^2(M_k, K_M\otimes E, \omega\otimes h)$ as $m'\to+\infty$.

Denote $B_{m'',l}=\min\{e^{-\Phi_{m''}},l\}$ for any $m'',l\in\mathbb{Z}_+$.
As $c_k\big(-v_{t_0,\epsilon}(\psi)\big)e^{v_{t_0,\epsilon}(\psi)}$ is upper bounded on $M_k$, hence we know that $\sqrt{c_k\big(-v_{t_0,\epsilon}(\psi)\big)e^{v_{t_0,\epsilon}(\psi)}B_{m'',l}}u_{k,t_0,m',\epsilon}$ weakly
converges to $\sqrt{c_k\big(-v_{t_0,\epsilon}(\psi)\big)e^{v_{t_0,\epsilon}(\psi)}B_{m'',l}}u_{k,t_0,\epsilon}$ in $L^2(M_k, K_M\otimes E, \omega\otimes h)$ as $m'\to+\infty$.

Hence by \eqref{I1m't0less Im't0}, \eqref{step3 I1m't0 2} and \eqref{step 5 estimate 2}, we have
\begin{equation}\label{estimate um' 1}
\begin{split}
&\int_{M_k}c_k\big(-v_{t_0,\epsilon}(\psi)\big)e^{v_{t_0,\epsilon}(\psi)}B_{m'',l}|u_{k,t_0,\epsilon}|^2_{\omega,h}dV_{M,\omega}\\
\le &\liminf_{m'\to +\infty}\int_{M_k}c_k\big(-v_{t_0,\epsilon}(\psi)\big)e^{v_{t_0,\epsilon}(\psi)}B_{m'',l}|u_{k,t_0,m',\epsilon}|^2_{\omega,h}dV_{M,\omega}\\
\le&\liminf_{m'\to +\infty}\int_{M_k}c_k\big(-v_{t_0,\epsilon}(\psi)\big)e^{v_{t_0,\epsilon}(\psi)}|u_{k,t_0,m',\epsilon}|^2_{\omega,h}e^{-\Phi_{m'}}dV_{M,\omega}\\
\le &\limsup_{m'\to +\infty}\tilde{M}_{m'}\bigg[(1+\tau) I_{1,m',t_0,\epsilon}+(1+\frac{1}{\tau})e^{-\beta_0 t_0}\bigg]\\
\le& \frac{(1+\tau)^2}{1-4\epsilon} I_{1,t_0}+(1+\frac{1}{\tau})C_{13}e^{-\beta_0 t_0}\\
< & +\infty.
\end{split}
\end{equation}
Letting $l\to+\infty$ and then $m''\to +\infty$ in inequality \eqref{estimate um' 1}, by monotone convergence theorem, we have (note that $\Phi=\varphi+\psi$)
\begin{equation}\label{estimate um' 2}
\begin{split}
&\int_{M_k}c_k\big(-v_{t_0,\epsilon}(\psi)\big)e^{v_{t_0,\epsilon}(\psi)}|u_{k,t_0,\epsilon}|^2_{\omega,h}e^{-\varphi-\psi}dV_{M,\omega}\\
\le& \frac{(1+\tau)^2}{1-4\epsilon} I_{1,t_0}+(1+\frac{1}{\tau})C_{13}e^{-\beta_0 t_0}\\
< & +\infty.
\end{split}
\end{equation}

As $\sqrt{2b_{t_0}\delta_{m'}+e^{-\beta_0 t_0} }$ is a real number, we know that $$P_{m'}(\sqrt{2b_{t_0}\delta_{m'}+e^{-\beta_0 t_0} }\eta_{k,t_0,m',\epsilon})=\sqrt{2b_{t_0}\delta_{m'}+e^{-\beta_0 t_0} }P_{m'}(\eta_{k,t_0,m',\epsilon}).$$
Denote $v_{k,t_0,m',\epsilon}=P_{m'}(\eta_{k,t_0,m',\epsilon})$. Then it follows from estimates \eqref{step 5 eta 2}, \eqref{I1m't0less Im't0} and \eqref{step3 I1m't0 2}  that
\begin{equation}\label{estiamte v 1}
\begin{split}
&\int_{M_k}|v_{k,t_0,m',\epsilon}|^2_{\omega,h}e^{-\Phi_{m'}}e^{-2\pi n_k \delta_{m'}\Upsilon-\phi}dV_{M,\omega}\\
\le& \int_{M_k}|\eta_{k,t_0,m',\epsilon}|^2_{\omega,h}e^{-\Phi_{m'}}e^{-2\pi n_k \delta_{m'}\Upsilon-\phi}dV_{M,\omega}\\
\le  &(1+\tau) I_{1,m',t_0,\epsilon}+(1+\frac{1}{\tau})e^{-\beta_0 t_0}\\
\le & \frac{(1+\tau)^2}{1-4\epsilon} \big(\frac{1}{\delta}c(T)e^{-T}+\int^{+\infty}_T c(t_1)e^{-t_1}dt_1\big)
\int_{Y^0}|f|^2_{\omega,h}e^{-\varphi}dV_{M,\omega}[\psi]+(1+\frac{1}{\tau})e^{-\beta_0 t_0}\\
\le &\hat{C}_{14}< +\infty,
\end{split}
\end{equation}
where $\hat{C}_{14}$ is a positive constant independent of $m'$ and $t_0$.
Equality \eqref{step 5 d-bar equation} becomes
\begin{equation}\label{step 7 d-bar equation}
D''u_{k,t_0,m',\epsilon}+\sqrt{2b_{t_0}\delta_{m'}+e^{-\beta_0 t_0} }v_{k,t_0,m',\epsilon}=\lambda_{t_0}.
\end{equation}

Note that $\Phi_{m'}$, $\Upsilon$ and $\phi=u_k(-v_{t_0,\epsilon}(\psi))$ is uniformly upper bounded on $M_k$ with respect to $m'$, $\lim_{m'\to+\infty}\delta_{m'}=0$. Then it follows from inequality \eqref{estiamte v 1} that we have
\begin{equation}\label{unifor estimate of vm't0}
\sup_{m'}\int_{M_k}|v_{k,t_0,m',\epsilon}|^2_{\omega,h}dV_{M,\omega}<+\infty,
\end{equation}
Since the closed unit ball of Hilbert space is
weakly compact, we can extract a subsequence of $\{v_{k,t_0,m',\epsilon}\}_{m'}$ (also denoted by $\{v_{k,t_0,m',\epsilon}\}_{m'}$) weakly convergent to $v_{k,t_0,\epsilon}$ in
$L^2(M_k, \wedge^{n,1}T^*M\otimes E, \omega\otimes h)$ as $m'\to+\infty$.
 For fixed integers $m''>0$ and $l>0$, denote $$W_{m'',l}=\min\{e^{-\Phi_{m''}},l\}.$$
Then $W_{m'',l}$ is a bounded function on $M_k$.  Then we know that $\sqrt{W_{m'',l}}v_{k,t_0,m',\epsilon}$ weakly converges to $\sqrt{W_{m'',l}}v_{k,t_0,\epsilon}$ in
$L^2(M_k, \wedge^{n,1}T^*M\otimes E, \omega\otimes h)$ as $m'\to+\infty$.
It follows from inequality \eqref{estiamte v 1}, $0\le\delta_{m'}\le \delta_{1}$, $\Upsilon$  is upper-bounded on $M_k$ and $\phi=u_k(-v_{t_0,\epsilon}(\psi))$ is bounded on $M_k$ that we  have
\begin{equation}\label{estiamte vt0}
\begin{split}
&\int_{M_k}|v_{k,t_0,\epsilon}|^2_{\omega,h}W_{m'',l}dV_{M,\omega}\\
\le& \liminf_{m'\to +\infty}\int_{M_k}|v_{k,t_0,m',\epsilon}|^2_{\omega,h}W_{m'',l}dV_{M,\omega}\\
\le& \liminf_{m'\to +\infty}\int_{M_k}|v_{k,t_0,m',\epsilon}|^2_{\omega,h}e^{-\Phi_{m''}}dV_{M,\omega}\\
\le &\liminf_{m'\to +\infty}\int_{M_k}|v_{k,t_0,m',\epsilon}|^2_{\omega,h}e^{-\Phi_{m'}}dV_{M,\omega}\\
\le &C_{14}\liminf_{m'\to +\infty}\int_{M_k}|v_{k,t_0,m',\epsilon}|^2_{\omega,h}e^{-\Phi_{m'}}e^{-2\pi n_k \delta_{m'}\Upsilon-\phi}dV_{M,\omega}\\
\le & C_{14}\hat{C}_{14}\\
< &+\infty,
\end{split}
\end{equation}
where $C_{14}$ is a positive constant independent of $t_0$, $m'$, $\epsilon$, $l$ and $m''$. Let $l\to +\infty$ and $m''\to +\infty$ in \eqref{estiamte vt0}, by monotone convergence theorem, we have
\begin{equation}\label{estiamte vt0_1}
\begin{split}
\int_{M_k}|v_{k,t_0,\epsilon}|^2_{\omega,h}e^{-\varphi-\psi}dV_{M,\omega}
\le  C_{14}\hat{C}_{14}
< +\infty,
\end{split}
\end{equation}

It follows from $\lim\limits_{m'\to +\infty}\delta_{m'}=0$ and Lemma \ref{weakly convergence} that we know  $\sqrt{2b_{t_0}\delta_{m'}+e^{-\beta_0 t_0}}v_{k,t_0,m',\epsilon}$ weakly converges to $\sqrt{e^{-\beta_0 t_0}}v_{k,t_0,\epsilon}$ in $L^2(M_k, \wedge^{n,1}T^*M\otimes E, \omega\otimes h_1)$ as $m'\to+\infty$.
Note that $\lambda_{t_0}:=D''[\big(1-v'_{t_0,\epsilon}(\psi)\big)\tilde{f}_{t_0}]$ and denote
$F_{k,t_0,\epsilon}=-u_{k,t_0,\epsilon}+\big(1-v'_{t_0,\epsilon}(\psi)\big)\tilde{f}_{t_0}$. Letting $m'\to +\infty$ in \eqref{step 7 d-bar equation}, we have
\begin{equation}\label{step 7 d-bar equation 2}
D''F_{k,t_0,\epsilon}=\sqrt{e^{-\beta_0 t_0}}v_{k,t_0,\epsilon}.
\end{equation}

It follows from estimate \eqref{estimate um' 2} that we have
\begin{equation}\label{estimate Fktet 1}
\begin{split}
&\int_{M_k}c_k\big(-v_{t_0,\epsilon}(\psi)\big)e^{v_{t_0,\epsilon}(\psi)}|F_{k,t_0,\epsilon}-\big(1-v'_{t_0,\epsilon}(\psi)\big)\tilde{f}_{t_0}|^2_{\omega,h}e^{-\varphi-\psi}dV_{M,\omega}\\
\le& \frac{(1+\tau)^2}{1-4\epsilon} I_{1,t_0}+(1+\frac{1}{\tau})C_{13}e^{-\beta_0 t_0}\\
< & +\infty.
\end{split}
\end{equation}

\

\textbf{Step 9: when $t_0\to +\infty$.}

\

Note that $v_{t_0,\epsilon}(\psi)\ge \psi$ and $c_k(t)e^{-t}$ is decreasing with respect to $t$ near $+\infty$.
It follows from inequality \eqref{estimate Fktet 1} that we have
\begin{equation}\label{step 6 estimate 1}
\begin{split}
&\int_{M_k} c_k(-\psi)
|F_{k,t_0,\epsilon}|^2_{\omega,h}e^{-\varphi}dV_{M,\omega}\\
\le&(1+\tau)\int_{M_k} c_k(-\psi)
|F_{k,t_0,\epsilon}-(1-v'_{t_0,\epsilon}(\psi))\tilde{f}_{t_0}|^2_{\omega,h}e^{-\varphi}dV_{M,\omega}\\
+&(1+\frac{1}{\tau})\int_{M_k} c_k(-\psi)
|(1-v'_{t_0,\epsilon}(\psi))\tilde{f}_{t_0}|^2_{\omega,h}e^{-\varphi}dV_{M,\omega}\\
\le&(1+\tau)
\int_{M_k}c_k\big(-v_{t_0,\epsilon}(\psi)\big)e^{v_{t_0,\epsilon}(\psi)-\psi}
|F_{k,t_0,\epsilon}-(1-v'_{t_0,\epsilon}(\psi))\tilde{f}_{t_0}|^2_{\omega,h}e^{-\varphi}dV_{M,\omega}\\
+&(1+\frac{1}{\tau})\int_{M_k} c_k(-\psi)
|(1-v'_{t_0,\epsilon}(\psi))\tilde{f}_{t_0}|^2_{\omega,h}e^{-\varphi}dV_{M,\omega}\\
\le&(1+\tau) \bigg[\frac{(1+\tau)^2}{1-4\epsilon} I_{1,t_0}+(1+\frac{1}{\tau})C_{13}e^{-\beta_0 t_0}\bigg]
+(1+\frac{1}{\tau})S_{t_0},
\end{split}
\end{equation}
where $S_{t_0}:=\int_{M_k} c_k(-\psi)
|(1-v'_{t_0,\epsilon}(\psi))\tilde{f}_{t_0}|^2_{\omega,h}e^{-\varphi}dV_{M,\omega}$. Now we prove that $$\lim_{t_0\to +\infty}S_{t_0}=0.$$
By the consturction of $\tilde{f}_{t_0}$, when $t_0$ is big enough, we have
\begin{equation}\label{step 6 estimate 2}
\begin{split}
S_{t_0}=&\int_{M_k} c_k(-\psi)
|(1-v'_{t_0,\epsilon}(\psi))\tilde{f}_{t_0}|^2_{\omega,h}e^{-\varphi}dV_{M,\omega}\\
=&\int_{M_k\cap \{\psi<-t_0\}}c_k(-\psi)|\tilde{f}_{t_0}|^2_{\omega,h}e^{-\varphi}dV_{M,\omega}\\
\le &\sum_{i=1}^{N}\int_{U_i\cap \{\psi<-t_0\}}c_k(-\psi)|\tilde{f}_{i,t_0}|^2_{\omega,h}e^{-\varphi}dV_{M,\omega}.
\end{split}
\end{equation}
Denote $$S_{i,t_0}:=\int_{U_i\cap \{\psi<-t_0\}}c_k(-\psi)|\tilde{f}_{i,t_0}|^2_{\omega,h}e^{-\varphi}dV_{M,\omega}.$$ It suffices to prove $\lim_{t_0\to +\infty}S_{i,t_0}=0$.

The following notations can be referred to the proof of Proposition \ref{p:inte} and Step 1.
On $U_i\Subset \Omega_i$, note that when $t_0$ is big enough, $\Gamma_i+m_i+t_0\ge \psi+T_1-T$. As $\hat{c}_1(t)e^{-t}$ is decreasing with respect to $t$, we have
\begin{equation}\nonumber
\begin{split}
\hat{c}_1(-\Gamma_i-m_i-t_0)e^{\Gamma_i+m_i+t_0}\ge \hat{c}_1\big(-\psi-(T_1-T)\big)e^{\psi+(T_1-T)},
\end{split}
\end{equation}
which implies that (note that $\psi\ge \Gamma_i+m_i$ on  $\Omega_i$)
\begin{equation}\nonumber
\begin{split}
\hat{c}_1(-\Gamma_i-m_i-t_0)\ge& \hat{c}_1\big(-\psi-(T_1-T)\big)e^{\psi-\Gamma_i-m_i-t_0+(T_1-T)}\\
\ge& \hat{c}_1\big(-\psi-(T_1-T)\big)e^{-t_0+(T_1-T)}\\
\ge& e^{-T_1}c_k(-\psi)e^{-t_0+(T_1-T)}\\
=&c_k(-\psi)e^{-t_0-T}
\end{split}
\end{equation}
Hence it follows from inequality \eqref{step1part3 1} that we have
\begin{equation}\label{step 6 estimate 3}
\begin{split}
    \int_{U_i\cap \{\psi<-t_0\}}c_k(-\psi)|\tilde{f}_{i,t_0}|^2_{\omega,h(\text{det}h)^{\beta}}e^{-(1+\beta r)\varphi}dV_{M,\omega}
    \le  C_{15},
\end{split}
\end{equation}
where $C_{15}>0$ is a real number independent of $t_0$.

The following discussion is similar to the discussion we did in Step 5.
For fixed $i$,  let $\{e_1,\cdots,e_r\}$ be a holomorphic frame on $ E|_{U_i}$. It follows from Lemma \ref{existence of bounded tra} that there exists a local frame $\{\zeta_1,\cdots,\zeta_r\}$ of $E|_{U_i}$ such that the local expression of $ he^{-\varphi}$ is a diagonal matrix with diagonal element $\det (h)e^{-r\varphi}$ and the transition matrix $B^{-1}$ from  $\{e_1,\cdots,e_r\}$ to  $\{
\zeta_1,\cdots,\zeta_r\}$ satisfies that each element $b_{i,j}(z)$ of $B$ is a bounded function on $U_i\cap U_j$. Let $dw$ be a local frame of $K_M|_{U_i}$.

 Then we can assume that

$$\tilde{f}_{i,t_0}=\sum_{p=1}^{r} F_{i,t_0,p}\zeta_p\otimes dw,$$
where $F_{i,t_0,p}$ are measurable functions on $U_i$. Then $S_{i,t_0}$ becomes
\begin{equation}\label{step 6 estimate 4}
\begin{split}
S_{i,t_0}=\sum_{p=1}^{r}\int_{U_i\cap \{\psi<-t_0\}}c_k(-\psi)|\tilde{F}_{i,t_0,p}|^2(\det h)e^{-r\varphi}dw\wedge d\bar{w}.
\end{split}
\end{equation}
Since the case is local and $he^{-\varphi}$ is a singular metric on $E|_{U_i\cap U_j}$ and  locally lower bounded, we can assume that all eigenvalues of $he^{-\varphi}$ are greater than $1$.
By the inequality \eqref{step 6 estimate 3} and the constructions of local frame $\{\zeta_1,\cdots,\zeta_r\}$, for any $1\le i\le N$ and any $1\le p\le r$, we have
\begin{equation}\label{step 6 estimate 5}
\begin{split}
&\int_{U_i\cap\{\psi<-t_0\}}c_k(-\psi)|F_{i,t_0,p}|^2(\det h)^{1+\beta}e^{-(1+\beta)r\varphi}dw\wedge d\bar{w}\\
=&\int_{U_i\cap\{\psi<-t_0\}}c_k(-\psi)|F_{i,t_0,p}|^2(\det h)e^{-r\varphi}(\text{det}h)^{\beta}e^{-\beta r\varphi}dw\wedge d\bar{w}\\
\le & \int_{U_i\cap\{\psi<-t_0\}}c_k(-\psi)|\tilde{f}_{i,t_0}|^2_{\omega,h}e^{-\varphi} (\text{det}h)^{\beta}e^{-\beta r\varphi}dw\wedge d\bar{w}
\le \tilde{C}_{15},
\end{split}
\end{equation}
where $\tilde{C}_{15}>0$ is a real number independent of $t_0$.
It follows from H\"older inequality that
\begin{equation}\nonumber
\begin{split}
S_{i,t_0}=&\sum_{p=1}^{r}\int_{U_i\cap\{\psi<-t_0\}}c_k(-\psi)|F_{i,t_0,p}|^2(\det h)e^{-r\varphi}dw\wedge d\bar{w}\\
\le &\sum_{p=1}^{r}\big(\int_{U_i\cap\{\psi<-t_0\}}c_k(-\psi)|F_{i,t_0,p}|^2(\det h)^{1+\beta}e^{-(1+\beta)r\varphi}dw\wedge d\bar{w}\big)^{\frac{1}{1+\beta}}\times\\
&\big(\int_{U_i\cap\{\psi<-t_0\}}c_k(-\psi)|F_{i,t_0,p}|^2dw\wedge d\bar{w}\big)^{\frac{\beta}{1+\beta}}\\
\le& C_{16}\sum_{p=1}^{r}
\big(\int_{U_i\cap\{\psi<-t_0\}}c_k(-\psi)|F_{i,t_0,p}|^2dw\wedge d\bar{w}\big)^{\frac{\beta}{1+\beta}},
\end{split}
\end{equation}
when $t_0$ is big enough and $C_{16}>0$ is a real number independent of $t_0$. It suffices to prove that
\begin{equation}\nonumber
\begin{split}
\lim_{t_0\to +\infty}\int_{U_i\cap\{\psi<-t_0\}}c_k(-\psi)|F_{i,t_0,p}|^2dw\wedge d\bar{w}=0.
 \end{split}
\end{equation}

 We cover $\mu^{-1}(U_i)\cap \{\psi\circ \mu <-t_0\}$ by a finite number of coordinate balls $W$ as we did in the Step 1 of the proof of Proposition \ref{p:inte}. Denote $W_{i,t_0}:=W\cap\mu^{-1}(U_i)\cap \{\psi\circ \mu <-t_0\}$ and $d\lambda_w$ be the Lebesgue measure on $W_{i,t_0}$. It suffices to prove that
\begin{equation}\label{step 6  goal 2}
\begin{split}
\lim_{t_0\to +\infty}\int_{W_{i,t_0}}c_k(-\psi\circ \mu)|F_{i,t_0,p}\circ \mu|^2|J_{\mu}|^2d\lambda(w)=0.
 \end{split}
\end{equation}

 Note that $\{e_1\circ \mu,\cdots,e_r\circ \mu\}$ is a local frame of $E|_W$. Let $dw$ be the local frame of $K_{\tilde{M}}|_{W}$. Assume that under the local frame  $\{e_1\circ \mu,\cdots,e_r\circ \mu\}$ and $dw$, we can write
 $$\tilde{f}_{i,t_0}\circ \mu=\sum_{p=1}^{r} f_{i,t_0,p}e_p\otimes dw \text{ on } W.$$
Then inequalities \eqref{eq:220627d}, \eqref{eq:220627e}, \eqref{eq:220627g} and \eqref{eq:220627h} show that for  any $1\le p\le r$, the following inequalities hold
\begin{equation}\label{step 6 estimate 6}
\begin{split}
 &|f_{i,t_0,p}\circ\mu(w',w_{p_0})- f_{i,t_0,p}\circ\mu(w',0)|^2\le C_{17}e^{\beta_1t_0}\prod_{l\in\kappa}|w_l|^2,\\
		&|f_{i,t_0,p}\circ\mu(w',0)|^2=|f_p\circ\mu(w',0)|^2\le C_{17}\prod_{l\in\kappa\backslash\{p_0\}}|w_l|^2
 \end{split}
\end{equation}
in case (A) and
\begin{equation}\label{step 6 estimate 7}
\begin{split}
 &|f_{i,t_0,p}\circ\mu(w)|^2\le C_{17}e^{\beta_1t_0} \text{\ when $\kappa=\emptyset$\ and},\\
		&|f_{i,t_0,p}\circ\mu(w)|^2\le C_{17}e^{\beta_1t_0}\prod_{l\in\kappa}|w_l|^2 \text{\ when $\kappa\neq\emptyset$.}
 \end{split}
\end{equation}
in case (B), where $C_{17}>0$ is a real number independent of $t_0$.

Note that for any $1\le i\le N$, we have
 $$\tilde{f}_{i,t_0}\circ \mu=(F_{i,t_0,1}\circ \mu,\cdots,F_{i,t_0,r}\circ \mu)^T=B(f_{i,t_0,1},\cdots,f_{i,t_0,r})^T,$$
 where $T$ means transposition.
Then it follows from inequalities \eqref{step 6 estimate 6}, \eqref{step 6 estimate 7} and $B$ is bounded on $W$ (shrink $W$ if necessary) that we have
\begin{equation}\label{step 6 estimate 8}
\begin{split}
 &|F_{i,t_0,p}\circ\mu(w',w_{p_0})- F_{i,t_0,p}\circ\mu(w',0)|^2\le C_{18}e^{\beta_1t_0}\prod_{l\in\kappa}|w_l|^2,\\
		&|F_{i,t_0,p}\circ\mu(w',0)|^2=|f_p\circ\mu(w',0)|^2\le C_{18}\prod_{l\in\kappa\backslash\{l_0\}}|w_l|^2
 \end{split}
\end{equation}
in case (A) and
\begin{equation}\label{step 6 estimate 9}
\begin{split}
 &|F_{i,t_0,p}\circ\mu(w)|^2\le C_{18}e^{\beta_1t_0} \text{\ when $\kappa=\emptyset$\ and},\\
		&|F_{i,t_0,p}\circ\mu(w)|^2\le C_{18}e^{\beta_1t_0}\prod_{l\in\kappa}|w_l|^2 \text{\ when $\kappa\neq\emptyset$.}
 \end{split}
\end{equation}
in case (B), where $C_{18}>0$ is a real number independent of $t_0$.

By inequalities \eqref{step 6 estimate 8} and \eqref{step 6 estimate 9}, to prove \eqref{step 6  goal 2}, we only need to prove
\begin{equation}\label{step 6  goal 3}
\begin{split}
\lim_{t_0\to +\infty}\int_{W_{i,t_0}}c_k(-\psi\circ \mu)\big(\prod_{l\in\kappa\backslash\{l_0\}}|w_l|^{2}\big)
\big(\prod_{l=1}^n|w_l|^{2b_l}\big)d\lambda(w)=0.
 \end{split}
\end{equation}
in case (A) and (note that on $\{\psi<-t_0\}$, $e^{\beta_1t_0}\le e^{-\beta_1 \psi}$)
\begin{equation}\label{step 6  goal 3}
\begin{split}
\lim_{t_0\to +\infty}\int_{W_{i,t}}c_k(-\psi\circ \mu)\big(\prod_{l\in\kappa}|w_l|^{2}\big)\big(\prod_{l=1}^n|w_l|^{2b_l-2\beta_1ca_l}\big)d\lambda(w)=0.
 \end{split}
\end{equation}
in case (A) and case (B).

 By using Fubini's Theorem and the change of variables, direct calculation shows that
\begin{equation}\nonumber
\begin{split}
&\lim_{t_0\to +\infty}\int_{W_{i,t}}c_k(-\psi\circ \mu)\big(\prod_{l\in\kappa\backslash\{l_0\}}|w_l|^{2}\big)
\big(\prod_{l=1}^n|w_l|^{2b_l}\big)d\lambda(w)\\
\le &C_{19}\lim_{t_0\to +\infty}\int_{t_0}^{+\infty}c(t-M)e^{-t+M}
=0
 \end{split}
\end{equation}
in case (A), where $M:=\sup_{W} \tilde{u}(w)$. Hence \eqref{step 6 estimate 8} holds. By similar calculation and $\beta_1$ satisfies \eqref{condition of beta1}, we also know that \eqref{step 6 estimate 9} holds in case (B). By the above discussion, we know that
\begin{equation}\label{step 6 estimate 10}
\lim_{t_0\to +\infty}S_{t_0}=0.
\end{equation}
Note that $$\limsup_{t_0\to +\infty}  I_{1,t_0}\le\big(\frac{1}{\delta}c_k(T)e^{-T}+\int^{+\infty}_T c_k(t_1)e^{-t_1}dt_1\big)
\int_{Y^0}|f|^2_{\omega,h}e^{-\varphi}dV_{M,\omega}[\psi]$$
Combining equalities \eqref{step 6 estimate 1} and \eqref{step 6 estimate 10} we have
\begin{equation}\label{step 6 estimate 11}
\sup_{t_0}\int_{M_k}c_k(-\psi)
|F_{k,t_0,\epsilon}|^2_{\omega,h}e^{-\varphi}dV_{M,\omega}<+\infty.
\end{equation}

Hence we know that there exists a subsequence of $\{F_{k,t_0,\epsilon}\}_{t_0}$ (also denoted by $\{F_{k,t_0,\epsilon}\}_{t_0}$) weakly convergent to
$\{F_{k,\epsilon}\}$ in $L^2(M_k, K_M\otimes E, \omega\otimes he^{-\varphi}c_k(-\psi))$ as $t_0\to +\infty$.

It follows from inequality \eqref{step 6 estimate 1} and \eqref{step 6 estimate 10} that we have
\begin{equation}\label{estimate FKET0}
\begin{split}
&\int_{M_k} c_k(-\psi)
|F_{k,\epsilon}|^2_{\omega,h}e^{-\varphi}dV_{M,\omega}\\
\le& \liminf_{t_0\to +\infty}\int_{M_k} c_k(-\psi)
|F_{k,t_0,\epsilon}|^2_{\omega,h}e^{-\varphi}dV_{M,\omega}\\
\le&\limsup_{t_0\to +\infty}\bigg((1+\tau) \bigg[\frac{(1+\tau)^2}{1-4\epsilon} I_{1,t_0}+(1+\frac{1}{\tau})C_{13}e^{-\beta_0 t_0}\bigg]
+(1+\frac{1}{\tau})S_{t_0}\bigg)\\
\le & \frac{(1+\tau)^3}{1-4\epsilon}\big(\frac{1}{\delta}c_k(T)e^{-T}+\int^{+\infty}_T c_k(t_1)e^{-t_1}dt_1\big)
\int_{Y^0}|f|^2_{\omega,h}e^{-\varphi}dV_{M,\omega}[\psi].
\end{split}
\end{equation}

It follows from $\psi$ is upper bounded on $M_k$, $he^{-\varphi}$ is locally lower bounded and inequality \eqref{estiamte vt0_1} that
\begin{equation}\label{estiamte vt0 step 8}
\begin{split}
\sup_{t_0}\int_{M_k}|v_{k,t_0,\epsilon}|^2_{\omega,\tilde{h}}dV_{M,\omega}
\le  C_{14}\hat{C}_{14}
< +\infty,
\end{split}
\end{equation}
where $\tilde{h}$ is a smooth metric on $E$ such that $he^{-\varphi}\ge \tilde{h}$ on $M_k$.

Since the closed unit ball of Hilbert space is
weakly compact, we can extract a subsequence of $\{v_{k,t_0,\epsilon}\}$ (also denoted by $\{v_{k,t_0,\epsilon}\}_{t_0}$) weakly convergent to $v_{k,\epsilon}$ in
$L^2(M_k, \wedge^{n,1}T^*M\otimes E, \omega\otimes \tilde{h})$ as $t_0\to+\infty$. It follows from Lemma \ref{weakly convergence} that we know  $\sqrt{e^{-\beta_0 t_0}}v_{k,t_0,\epsilon}$ weakly converges to $0$ in $L^2(M_k, \wedge^{n,1}T^*M\otimes E, \omega\otimes \tilde{h})$ as $t_0\to+\infty$. Hence
$\sqrt{e^{-\beta_0 t_0}}v_{k,t_0,\epsilon}$ weakly converges to $0$ in $L^2_{\text{loc}}(M_k, \wedge^{n,1}T^*M\otimes E, \omega\otimes \tilde{h})$ as $t_0\to+\infty$.

It follows from $\psi$ is smooth on $M_k\backslash \sum$, $c_k(t)$ is smooth function on $[T+4\epsilon_k,+\infty)$ and $\{F_{k,t_0,\epsilon}\}$  weakly converges to
$\{F_{k,\epsilon}\}$ in $L^2(M_k, K_M\otimes E, \omega\otimes he^{-\varphi}c_k(-\psi))$ as $t_0\to +\infty$ that we have $\{F_{k,t_0,\epsilon}\}$ also weakly converges to
$\{F_{k,\epsilon}\}$ in $L^2_{\text{loc}}(M_k\backslash \sum, K_M\otimes E, \omega\otimes he^{-\varphi})$ as $t_0\to +\infty$. It follows from Lemma \ref{equiv of weak convergence} that  $\{F_{k,t_0,\epsilon}\}$ also weakly converges to
$\{F_{k,\epsilon}\}$ in $L^2_{\text{loc}}(M_k\backslash \sum, K_M\otimes E, \omega\otimes \tilde{h})$.

Let $t_0\to +\infty$ in equality \eqref{step 7 d-bar equation 2}, we have
\begin{equation}\label{step 8 d-bar equation}
D''F_{k,\epsilon}=0 \text{\ holds on } M_k\backslash \sum.
\end{equation}
Hence $F_{k,\epsilon}$ is an $E$-valued holomorphic $(n,0)$-form on $M_k\backslash \sum$, which satisfies
\begin{equation}\label{estimate FKET step8}
\begin{split}
&\int_{M_k} c_k(-\psi)
|F_{k,\epsilon}|^2_{\omega,h}e^{-\varphi}dV_{M,\omega}\\
\le & \frac{(1+\tau)^3}{1-4\epsilon}\big(\frac{1}{\delta}c(T)e^{-T}+\int^{+\infty}_T c(t_1)e^{-t_1}dt_1\big)
\int_{Y^0}|f|^2_{\omega,h}e^{-\varphi}dV_{M,\omega}[\psi].
\end{split}
\end{equation}
\

\textbf{Step 10: solving $\bar{\partial}$-equation locally.}

\

In this step, we prove that $F_{k,\epsilon}$ is actually  a holomorphic extension of $f$ from $Y^0\cap M_k$ to $M_k$.

Let $x\in M_k\cap Y^0$ be any point. Let $\Omega_x$ be as in Step 1. Let $\tilde{U}_x\Subset \Omega_x\cap M_k$ be a local coordinate ball which is centered at $x$. Note that $E|_{\tilde{U}_x}$ is trivial vector bundle.

Note that by equality \eqref{step 7 d-bar equation 2} and the definition of $F_{k,t_0}$, we have
$$D''u_{k,t_0,\epsilon}+\sqrt{e^{-\beta_0t_0}}v_{k,t_0,\epsilon}=
D''[\big(1-v'_{t_0,\epsilon}(\psi)\big)\tilde{f}_{t_0}].$$
It follows from inequality \eqref{estiamte vt0_1} and $he^{-\varphi}$ is locally lower bounded that we have
\begin{equation}\label{Step 8 formula 1}
\int_{\tilde{U}_x}|v_{k,t_0,\epsilon}|^2_{\tilde{h}}e^{-\psi}\le C_{20},
\end{equation}
where $\tilde{h}$ is a smooth metric on $E$ such that $he^{-\varphi}\ge \tilde{h}$ on $\tilde{U}_x$ and $C_{20}>0$ is a positive number independent of $t_0$.

Note that $D''\big(\sqrt{e^{-\beta_0t_0}}v_{k,t_0,\epsilon}\big)=0$. It follows from Lemma \ref{hormander} that there exists an $E$-valued $(n,0)$-form $s_{k,t_0,\epsilon}\in L^2(\tilde{U}_x, K_M\otimes E,\tilde{h}e^{-\psi})$ such that $D''s_{k,t_0,\epsilon}=\sqrt{e^{-\beta_0t_0}}v_{k,t_0,\epsilon}$ and
\begin{equation}\label{Step 8 formula 2}
\int_{\tilde{U}_x}|s_{k,t_0,\epsilon}|^2_{\tilde{h}}e^{-\psi}\le C_{21}\int_{\tilde{U}_x}|\sqrt{e^{-\beta_0t_0}}v_{k,t_0,\epsilon}|^2_{\tilde{h}}e^{-\psi}
\le C_{21}C_{20}e^{-\beta_0t_0},
\end{equation}
where $C_{21}>0$ is a positive number independent of $t_0$. Hence we have
\begin{equation}\label{Step 8 formula 3}
\int_{\tilde{U}_x}|s_{k,t_0,\epsilon}|^2_{\tilde{h}}\le C_{22}e^{-\beta_0t_0},
\end{equation}
where $C_{22}>0$ is a positive number independent of $t_0$.

Now define $G_{k,t_0,\epsilon}:=-u_{k,t_0,\epsilon}-s_{k,t_0,\epsilon}
+\big(1-v'_{t_0,\epsilon}(\psi)\big)\tilde{f}_{t_0}$ on $\tilde{U}_x$. Then we know that
$G_{k,t_0,\epsilon}=F_{k,t_0,\epsilon}-s_{k,t_0,\epsilon}$ and $D''G_{k,t_0,\epsilon}=0$. Hence $G_{k,t_0,\epsilon}$ is holomorphic on $\tilde{U}_x$ and we know that $u_{k,t_0,\epsilon}+s_{k,t_0,\epsilon}$ is smooth on $\tilde{U}_x$ .

It follows from $\tilde{h}\le he^{-\varphi}$ on $\tilde{U}_x$, $c_k(t)e^{-t}$ is decreasing with respect to $t$, $\psi$ is upper bounded on $\tilde{U}_x$, inequalities \eqref{step 6 estimate 11} and \eqref{Step 8 formula 2} that we have
\begin{equation}\label{Step 8 formula 4}
\int_{\tilde{U}_x}c_k(-\psi)|G_{k,t_0,\epsilon}|^2_{\tilde{h}}
\le 2\int_{\tilde{U}_x}c_k(-\psi)|F_{k,t_0,\epsilon}|^2_{\tilde{h}}
+2\int_{\tilde{U}_x}|s_{k,t_0,\epsilon}|^2_{\tilde{h}}e^{-\psi}\le C_{23},
\end{equation}
where $C_{23}>0$ is a positive number independent of $t_0$.

It follows from inequality \eqref{estimate um' 2}, the construction of $v_{t_0,\epsilon}(t)$ and $\tilde{h}\le he^{-\varphi}$ on $\tilde{U}_x$ that we have
\begin{equation}\label{Step 8 formula 5}
\int_{\tilde{U}_x}|u_{k,t_0,\epsilon}|^2_{\tilde{h}}e^{-\psi}
\le C_{t_0},
\end{equation}
where $C_{t_0}>0$ is a sequence of positive number depends on $t_0$. Then, by inequalities \eqref{Step 8 formula 2} and \eqref{Step 8 formula 5}, we have
\begin{equation}\label{Step 8 formula 6}
\int_{\tilde{U}_x}|u_{k,t_0,\epsilon}+s_{k,t_0,\epsilon}|^2_{\tilde{h}}e^{-\psi}
\le 2C_{t_0}+2C_{21}C_{20}e^{-\beta_0t_0}.
\end{equation}
Note that $e^{-\psi}$ is not integrable along $Y$ and $u_{k,t_0,\epsilon}+s_{k,t_0,\epsilon}$ is smooth on $\tilde{U}_x$. By \eqref{Step 8 formula 6}, we know that $u_{k,t_0,\epsilon}+s_{k,t_0,\epsilon}=0$ on $\tilde{U}_x\cap Y$ for any $t_0$. Hence $G_{k,t_0,\epsilon}
=\tilde{f}_{t_0}=f$ on $\tilde{U}_x\cap Y_0$ for any $t_0$.

It follows from inequality \eqref{Step 8 formula 3} that there exists a subsequence of $\{s_{k,t_0,\epsilon}\}$ \big(also denoted by $\{s_{k,t_0,\epsilon}\}$\big) weakly converges to $0$ in  $L^2(\tilde{U}_x, K_M\otimes E,\tilde{h})$ as $t_0\to +\infty$. Note that $\{F_{k,t_0,\epsilon}\}$ weakly converges to $F_{k,\epsilon}$ in  $L^2_{\text{loc}}(\tilde{U}_x\backslash \sum, K_M\otimes E,\tilde{h})$ as $t_0\to +\infty$.  Hence we know that $\{G_{k,t_0,\epsilon}\}$ weakly converges to $F_{k,\epsilon}$ in  $L^2_{\text{loc}}(\tilde{U}_x\backslash \sum, K_M\otimes E,\tilde{h})$ as $t_0\to +\infty$.

It follows from inequality \eqref{Step 8 formula 4} and Lemma \ref{l:converge} that we know there exists a subsequence of $\{G_{k,t_0,\epsilon}\}$ \big(also denoted by $\{G_{k,t_0,\epsilon}\}$\big) compactly converges to an $E$-valued holomorphic $(n,0)$-form $G_{k,\epsilon}$ on $\tilde{U}_x$ as $t_0\to +\infty$. As $G_{k,t_0,\epsilon}
=f$ on $\tilde{U}_x\cap Y_0$ for any $t_0$, we know that $G_{k,\epsilon}=f$ on  $\tilde{U}_x\cap Y_0$.

As $\{G_{k,t_0,\epsilon}\}$  compactly converges to $G_{k,\epsilon}$ on $\tilde{U}_x$ as $t_0\to +\infty$ and  $\{G_{k,t_0,\epsilon}\}$ weakly converges to $F_{k,\epsilon}$ in  $L^2_{\text{loc}}(\tilde{U}_x\backslash \sum, K_M\otimes E,\tilde{h})$ as $t_0\to +\infty$, by the uniqueness of weak limit, we know that $G_{k,\epsilon}=F_{k,\epsilon}$ on any relatively compact open subset of $\tilde{U}_x$. Note that $G_{k,,\epsilon}$ is holomorphic on $\tilde{U}_x$ and $F_{k,\epsilon}$ is holomorphic on $\tilde{U}_x\backslash \sum$, we have $F_{k,\epsilon}\equiv G_{k,\epsilon}$ on $\tilde{U}_x\backslash \sum$, and we know that $F_{k,\epsilon}$ can extended to an $E$-valued holomorphic $(n,0)$-form on $\tilde{U}_x$ which equals to $G_{k,\epsilon}$. As $G_{k,\epsilon}=f$ on  $\tilde{U}_x\cap Y_0$, we know that $F_{k,\epsilon}=f$ on  $\tilde{U}_x\cap Y_0$. Since $x$ is arbitrarily chosen, we know that $F_{k,\epsilon}$ is  holomorphic on $M_k$ and $F_{k,\epsilon}=f$ on  $M_k\cap Y_0$.

\

\textbf{Step 11: end of the proof.}

\

Now we have a family of $E$-valued holomorphic $(n,0)$-forms $F_{k,\epsilon}$ on $M_k$ such that $F_{k,\epsilon}=f$ on $M_k\cap Y_0$ and
\begin{equation}\label{step 8 estimate 0}
\begin{split}
&\int_{M_k} c_k(-\psi)|F_{k,\epsilon}|^2_{\omega,h}e^{-\varphi}dV_{M,\omega}\\
\le&\frac{(1+\tau)^3}{1-4\epsilon}\big(\frac{1}{\delta}c_k(T)e^{-T}+\int^{+\infty}_T c_k(t_1)e^{-t_1}dt_1\big)
\int_{Y^0}|f|^2_{\omega,h}e^{-\varphi}dV_{M,\omega}[\psi].
\end{split}
\end{equation}
Recall that $\epsilon\in (0,\frac{1}{8})$. By inequality \eqref{step 8 estimate 0}, we have
\begin{equation}\label{step 11 estimate 0}
\sup_{\epsilon\in (0,\frac{1}{8})}\int_{M_k} c_k(-\psi)|F_{k,\epsilon}|^2_{\omega,h}e^{-\varphi}dV_{M,\omega}<+\infty.
\end{equation}

For any compact subset $K\subset M_k\backslash \sum=\{\psi=-\infty\}$, as $\psi$ is smooth on $M_k\backslash \sum$, we know that $\psi$ is upper and lower bounded on $K$. As $c_k(t)$ is continuous on $[T,+\infty)$, we have $c_{k}(-\psi)$ is uniformly lower bounded on $K$.
Note that $he^{-\varphi}$ is locally lower bounded. It follows from Lemma \ref{l:converge} and inequality \eqref{step 11 estimate 0}, we know that there exists a subsequence of $\{F_{k,\epsilon}\}_{\epsilon}$ (also denoted by $\{F_{k,\epsilon}\}_{\epsilon}$) compactly convergent to an $E$-valued holomorphic $(n,0)$-form  $F_k$ on $M_k$ as $\epsilon \to 0$. It follows from Fatou's lemma (let $\epsilon \to 0$) and inequality \eqref{step 8 estimate 0} that we have
\begin{equation}\nonumber
\begin{split}
&\int_{M_k} c_k(-\psi)|F_{k}|^2_{\omega,h}e^{-\varphi}dV_{M,\omega}\\
\le&\liminf_{\epsilon \to 0}\int_{M_k} c_k(-\psi)|F_{k,\epsilon}|^2_{\omega,h}e^{-\varphi}dV_{M,\omega}\\
\le&\liminf_{\epsilon \to 0}\frac{(1+\tau)^3}{1-4\epsilon}\big(\frac{1}{\delta}c_k(T)e^{-T}+\int^{+\infty}_T c_k(t_1)e^{-t_1}dt_1\big)
\int_{Y^0}|f|^2_{\omega,h}e^{-\varphi}dV_{M,\omega}[\psi]\\
\le &(1+\tau)^3\big(\frac{1}{\delta}c_k(T)e^{-T}+\int^{+\infty}_T c_k(t_1)e^{-t_1}dt_1\big)
\int_{Y^0}|f|^2_{\omega,h}e^{-\varphi}dV_{M,\omega}[\psi].
\end{split}
\end{equation}
Hence there exists a family of $E$-valued holomorphic $(n,0)$-forms $F_{k}$ on $M_k$ such that $F_{k}=f$ on $M_k\cap Y_0$ and
\begin{equation}\nonumber
\begin{split}
&\int_{M_k} c_k(-\psi)|F_{k}|^2_{\omega,h}e^{-\varphi}dV_{M,\omega}\\
\le&(1+\tau)^3\big(\frac{1}{\delta}c_k(T)e^{-T}+\int^{+\infty}_T c_k(t_1)e^{-t_1}dt_1\big)
\int_{Y^0}|f|^2_{\omega,h}e^{-\varphi}dV_{M,\omega}[\psi].
\end{split}
\end{equation}

Since $\tau>0$ is arbitrarily chosen and $c_k(T)e^{-T}=c(T)e^{-T}$, we have

\begin{equation}\label{continuous 1}
\begin{split}
&\int_{M_k} c_k(-\psi)|F_k|^2_{\omega,h}e^{-\varphi}dV_{M,\omega}\\
\le& \big(\frac{1}{\delta}c(T)e^{-T}+\int^{+\infty}_T c_k(t_1)e^{-t_1}dt_1\big)
\int_{Y^0}|f|^2_{\omega,h}e^{-\varphi}dV_{M,\omega}[\psi].
\end{split}
\end{equation}

Let $k_1> k$ be big enough. It follows from inequality \eqref{continuous 1}, $M_k\Subset M_{k_1}$ and $\frac{1}{\delta}c(T)e^{-T}+\int_{T}^{+\infty}c_k(t)e^{-t}dt$ converges to $\frac{1}{\delta}c(T)e^{-T}+\int_{T}^{+\infty}c(t)e^{-t}dt<+\infty$ as $k\to +\infty$ that we have
\begin{equation}\label{continuous uni estimate}
\begin{split}
\sup_{k_1}\int_{M_k} c_{k_1}(-\psi)|F_{k_1}|^2_{\omega,h}e^{-\varphi}dV_{M,\omega}<+\infty.
\end{split}
\end{equation}

For any compact subset $K\subset M_k\backslash \sum=\{\psi=-\infty\}$, as $\psi$ is smooth on $M_k\backslash \sum$, we know that $\psi$ is upper and lower bounded on $K$. It follows from $c_{k_1}(t)$ are uniformly convergent to $c(t)$ on any compact subset of $(T,+\infty)$ and $c(t)$ is a positive continuous function on $[T,+\infty)$ that we know
$c_{k_1}(-\psi)$ is uniformly lower bounded on $K$. Note that $he^{-\varphi}$ is locally lower bounded. By Lemma \ref{l:converge} and inequality \eqref{continuous uni estimate}, we know that there exists a subsequence of $\{F_{k_1}\}_{{k_1}\in\mathbb{Z}^+}$ (also denoted by $\{F_{k_1}\}_{{k_1}\in\mathbb{Z}^+}$) compactly convergent to an $E$-valued holomorphic $(n,0)$-form  $\tilde{F}_k$ on $M_k$. It follows from Fatou's lemma (let $k_1\to +\infty$) and inequality \eqref{continuous 1} that we have
\begin{equation}\label{continuous 2}
\begin{split}
&\int_{M_k} c(-\psi)|\tilde{F}_k|^2_{\omega,h}e^{-\varphi}dV_{M,\omega}\\
\le& \big(\frac{1}{\delta}c(T)e^{-T}+\int^{+\infty}_T c(t_1)e^{-t_1}dt_1\big)
\int_{Y^0}|f|^2_{\omega,h}e^{-\varphi}dV_{M,\omega}[\psi].
\end{split}
\end{equation}
As $\{F_{k_1}\}_{{k_1}\in\mathbb{Z}^+}$  compactly convergent to an $E$-valued holomorphic $(n,0)$-form  $\tilde{F}_k$ on $M$, we know that $\tilde{F}_k=f$ on $M_k\cap Y_0$.

Again for any compact subset $K\subset M\backslash Y$, as $\psi$ is smooth on $M\backslash Y$, we know that $\psi$ is upper and lower bounded on $K$. It follows  $c(t)$ is a positive continuous function on $[T,+\infty)$ that we know
$c(-\psi)$ is uniformly lower bounded on $K$. By Lemma \ref{l:converge} and inequality \eqref{continuous 1}, we know that there exists a subsequence of $\{\tilde{F}_k\}_{k\in\mathbb{Z}^+}$ (also denoted by $\{\tilde{F}_k\}_{k\in\mathbb{Z}^+}$) compactly convergent to an $E$-valued holomorphic $(n,0)$-form  $F$ on $M$. It follows from Fatou's lemma (let $k\to +\infty$) and inequality \eqref{continuous 2} that we have
\begin{equation}\nonumber
\begin{split}
&\int_{M_k} c(-\psi)|F|^2_{\omega,h}e^{-\varphi}dV_{M,\omega}\\
\le& \big(\frac{1}{\delta}c(T)e^{-T}+\int^{+\infty}_T c(t_1)e^{-t_1}dt_1\big)
\int_{Y^0}|f|^2_{\omega,h}e^{-\varphi}dV_{M,\omega}[\psi].
\end{split}
\end{equation}
Letting $k\to +\infty$, by monotone convergence theorem, we have
\begin{equation}\nonumber
\begin{split}
&\int_{M} c(-\psi)|F|^2_{\omega,h}e^{-\varphi}dV_{M,\omega}\\
\le& \big(\frac{1}{\delta}c(T)e^{-T}+\int^{+\infty}_T c(t_1)e^{-t_1}dt_1\big)
\int_{Y^0}|f|^2_{\omega,h}e^{-\varphi}dV_{M,\omega}[\psi].
\end{split}
\end{equation}
As $\{\tilde{F}_k\}_{k\in\mathbb{Z}^+}$  compactly convergent to an $E$-valued holomorphic $(n,0)$-form  $F$ on $M$, we know that $F=f$ on $M\cap Y_0$.

Theorem \ref{main result} has been proved.
\end{proof}

\begin{Remark}\label{rem:main result for line bdl}Let $M,Y,E,h$ be as in Theorem \ref{main result}. In the setting of Theorem \ref{main result}, $e^{-\varphi}$ can be viewed as a Lebesgue measurable metric on the trivial line bundle $L=M\times \mathbb{C}$. Note that Theorem \ref{main result} still holds for the case $L$ is nontrivial.

Specifically, let $(L,h_L)$ be any line bundle with singular hermitian metric $h_L$ on $M$. Assume that $h\otimes h_{L}$ is locally lower bounded (see Definition \ref{locally lower bound}) and the curvature of $h_L$ satisfies that  \\
(1') $\sqrt{-1}\Theta_{L}+\sqrt{-1}\partial\bar{\partial}\psi\ge 0$ on $M\backslash\{\psi=-\infty\}$ in the sense of currents;\\
(2')  $\big(\sqrt{-1}\Theta_{L}+\sqrt{-1}\partial\bar{\partial}\psi\big)
+\frac{1}{s(-\psi)}\sqrt{-1}\partial\bar{\partial}\psi\ge 0$ on $M\backslash\{\psi=-\infty\}$ in the sense of currents.\\
Since we can choose a smooth metric $\tilde{h}_L$ of $L$ such that $h_L=\tilde{h}_Le^{-\varphi}$, where $\varphi$ is a Lebesgue measurable function on $M$. Using almost the same proof as Theorem \ref{main result}, we know that Theorem \ref{main result} still holds for the case $L$ is a general line bundle with singular hermitian metric $h_L$.
\end{Remark}

Now we prove Theorem \ref{main result 1} by using Theorem \ref{main result} and Lemma \ref{class GT}.

\begin{proof}[Proof of Theorem  \ref{main result 1}] Since $M$ is weakly pseudoconvex, there exists a smooth plurisubharmonic
exhaustion function $P$ on $M$. Let $M_k:=\{P<k\}$ $(k=1,2,...,) $. We choose $P$ such that
$M_1\ne \emptyset$.\par
Then $M_k$ satisfies $
M_k\Subset  M_{k+1}\Subset  ...M$ and $\cup_{k=1}^n M_k=M$. Each $M_k$ is weakly
pseudoconvex K\"ahler manifold with exhaustion plurisubharmonic function
$P_k=1/(k-P)$.

For fixed $k$, as $\psi$ is plurisubharmonic function on $M$, we know that
$$\sup_{M_k}\psi<-T_k,$$
where $T_k>T$ is a real number depending on $k$.
It follows from Lemma \ref{class GT} that for any given $T_1>T$, there exist $c_{T_2}(t)\in \mathcal{G}_{T_2,\delta_2}$ and $\delta_2>0$ satisfying the conditions in Lemma \ref{class GT}, where $T<T_2<T_1$ and $T_1<T_k$.

It follows from $\psi<-T$ is a plurisubharmonic function on $M$, conditions in Theorem \ref{main result 1} that we know  the curvature conditions in Theorem \ref{main result} is satisfied. Note that $c_{T_2}(t)$ satisfies $$\frac{1}{\delta_2}c_{T_2}(T_2)e^{-T_2}+\int_{T_2}^{+\infty}c_{T_2}(t_1)e^{-t_1}dt_1
=\int_{T}^{+\infty}c(t_1)e^{-t_1}dt_1.$$
Then it follows from Theorem \ref{main result} that there exists a family $E$-valued holomorphic $(n,0)$-form
$F_{k,T_2}$ such that $F_{k,T_2}|_{Y_0}=f$ and
\begin{equation}\nonumber
 \int_{M_k} c_{T_2}(-\psi)|F_{k,T_2}|^2_{\omega,h}e^{-\varphi}dV_{M,\omega}\le \left(\int_{T}^{+\infty}c(t_1)e^{-t_1}dt_1\right)\int_{Y^0}|f|^2_{\omega,h}e^{-\varphi}dV_{M,\omega}[\psi].
\end{equation}
By the construction of $c_{T_2}$ (condition (1) in Lemma \ref{class GT}), we know that
$$c_{T_2}(-\psi)=c(-\psi)$$
on $M_k$. Denote $F_{k,T_2}$ by $F_k$, then we have

\begin{equation}\nonumber
\begin{split}
&\int_{M_k} c(-\psi)|F_{k}|^2_{\omega,h}e^{-\varphi}dV_{M,\omega}\\
 \le &\left(\int_{T}^{+\infty}c(t_1)e^{-t_1}dt_1\right)\int_{Y^0}|f|^2_{\omega,h}e^{-\varphi}dV_{M,\omega}[\psi].
\end{split}
\end{equation}

It follows from $he^{-\varphi}$ is locally lower bounded, Lemma \ref{l:converge} and
diagonal method, there exists a subsequence of $\{F_{k'}\}$, which is denoted by $\{F_{k''}\}$, such that $\{F_{k''}\}$ is uniformly convergent on any
$\overline{M_k}$ to an $E$-valued holomorphic $(n,0)$-form  $F$ on $M$.
Then by Fatou's Lemma, we have

\begin{equation}\nonumber
\begin{split}
&\int_{M} c(-\psi)|F|^2_{\omega,h}e^{-\varphi}dV_{M,\omega}\\
\le&\liminf_{k\to+\infty}\int_{M_k} c(-\psi)|F_{k}|^2_{\omega,h}e^{-\varphi}dV_{M,\omega}\\
\le& \big(\int^{+\infty}_T c(t_1)e^{-t_1}dt_1\big)
\int_{Y^0}|f|^2_{\omega,h}e^{-\varphi}dV_{M,\omega}[\psi].
\end{split}
\end{equation}

We also note that $F|_{Y_0}=f$. Theorem \ref{main result 1} has been proved.

\end{proof}

\section{Proof of Theorem \ref{nec in equ of L2 ext} and Corollary \ref{ohsawa3 sin ver}}

In this section, we prove Theorem \ref{nec in equ of L2 ext} and Corollary \ref{ohsawa3 sin ver}.

We firstly prove Theorem \ref{nec in equ of L2 ext} by using the Theorem \ref{main result 1} and the concavity property of minimal $L^2$ integrals
\begin{proof}[Proof of Theorem \ref{nec in equ of L2 ext}] The notations in the proof can be referred to Section \ref{minimal l2 integrals}.

 Note that by the inequality \eqref{f_1 is finite}, we have $G(0)<+\infty$. It follows from Theorem \ref{concavity of min l2 int} that we know that $G(h^{-1}(r))$ is concave with respect to  $r\in (0,\int_{0}^{+\infty}c(t)e^{-t}dt)$, where $h(t)=\int_{t}^{+\infty}c(l)e^{-l}dl$.

 Note that $M$ is weakly pseudoconvex K\"{a}hler manifold, we have a smooth plurisubharmonic exhaustion  function $\Phi$ of $M$.
 As $\psi<0$ is a plurisubharmonic function on $M$ with neat analytic singularities, we know that $e^{\psi}$ is a smooth plurisubharmonic function on $M$. Hence, for any $t\ge 0$, we know that $\Phi+\frac{1}{e^{-t}-e^{\psi}}$ is a smooth  plurisubharmonic exhaustion  function of $\{\psi<-t\}$. Note that $\{\psi<-t\}$ is an open complex sub-manifold of $M$, which implies $\{\psi<-t\}$ is K\"{a}hler. Then we know that $\{\psi<-t\}$ is a weakly pseudoconvex K\"{a}hler manifold for any $t\ge 0$.

 It follows from Theorem \ref{main result 1} that, for any $t\ge 0$, there exists an $E$-valued holomorphic $(n,0)$-form $F_t$ on $\{\psi<-t\}$ such that $F_t|_{Y_0}=f$ and
\begin{equation}\label{necessary formula 1}
\int_{\{\psi<-t\}} c(-\psi)|F_t|^2_{\omega,h}e^{-\varphi}dV_{M,\omega}\le \left(\int_{t}^{+\infty}c(t_1)e^{-t_1}dt_1\right)
\int_{Y^0}|f|^2_{\omega,h}e^{-\varphi}dV_{M,\omega}[\psi].
\end{equation}

Recall that $\mathcal{F}_{z_0}=\mathcal{E}(e^{-\psi})_{z_0}$ for any $z_0\in Z_0$ and $Z_0=Y_0$. It follows from that $\psi$ has log canonical singularities along $Y$ that $(\tilde{f}-f_1)_{z_0}\in
\mathcal{O} (K_M)_{z_0} \otimes \mathcal{F}_{z_0},\text{for any }  z_0\in Z_0$ is equivalent to $\tilde{f}=f_1=f$ on $Z_0=Y$.

Then by inequality \eqref{necessary formula 1} and the definition of $G(t)$, for any $t\ge 0$, we have
\begin{equation}\label{necessary formula 2}
\frac{G(t)}{\int_{t}^{+\infty}c(t_1)e^{-t_1}dt_1}\le
\int_{Y^0}|f|^2_{\omega,h}e^{-\varphi}dV_{M,\omega}[\psi].
\end{equation}

Since the equality $\|f\|_{L^2}=\inf\{\|F\|_{L^2}: F$ is a holomorphic extension of $f$ from $Y$ to $M\}$ holds, we know that
\begin{equation}\label{necessary formula 3}
\frac{G(0)}{\int_{0}^{+\infty}c(t_1)e^{-t_1}dt_1}=
\int_{Y^0}|f|^2_{\omega,h}e^{-\varphi}dV_{M,\omega}[\psi].
\end{equation}
Combining with formulas \eqref{necessary formula 2}, \eqref{necessary formula 3} and $G(h^{-1}(r))$ is concave with respect to  $r\in (0,\int_{0}^{+\infty}c(t)e^{-t}dt)$, we know that
$G(h^{-1}(r))$ is actually linear with respect to  $r\in (0,\int_{0}^{+\infty}c(t)e^{-t}dt)$.

Then by Corollary \ref{necessary condition for linear of G}, we know that there exists a unique $E$-valued holomorphic $(n,0)$-form $F$ on
$M$ such that $F|_{Y_0}=f$ and for any $t\ge 0$
\begin{equation}\nonumber
G(t)=\int_{\{\psi<-t\}} c(-\psi)|F|^2_{\omega,h}e^{-\varphi}dV_{M,\omega}=(\int_t^{+\infty}c(t_1)e^{-t_1}dt_1) \int_{Y^0}|f|^2_{\omega,h}e^{-\varphi}dV_{M,\omega}[\psi].
\end{equation}

Theorem \ref{nec in equ of L2 ext} has been proved.
\end{proof}

Now we prove Corollary \ref{ohsawa3 sin ver}.
\begin{proof}
Since $M$ is Stein, there exists a sequence of Stein manifolds $M_k$ satisfies $
M_k\Subset  M_{k+1}\Subset  ...M$ and $\cup_{k=1}^n M_k=M$.

As $M$ is a Stein manifold and $\psi$ is plurisubharmonic function on $M$, there exists a sequence of smooth plurisubharmonic function $\{\psi_m\}_{m\ge 1}$ on $M$ decreasingly converges to $\psi$ as $m\to +\infty$. For fixed $k$, we may assume that $\sup_m\sup_{M_k}\log|s|^2+\psi_m<0$.

It follows from Theorem \ref{main result 1} $\big( M\sim M_k$, $c(t)\sim 1$, $h\sim h_Ee^{-\psi}$, $\psi\sim \log|s|^2+\psi_m$, $\varphi\sim -\psi_m\big)$ that there exists a holomorphic section $F_{m,k}$ of $K_M\otimes E$ on $M_k$ satisfying $F_{m,k}=f\wedge ds$ on $S_{reg}\cap M_k$ and
\begin{equation}\label{cor 1.23 estimate}
  c_{n}\int_{M_k}\{F_{m,k},F_{m,k}\}_{h_E}e^{-\psi+\psi_m}
\le 2\pi  c_{n-1}\int_{S_{reg}\cap M_k}\{f,f\}_{h_E}e^{-\psi}.
\end{equation}

Let $K$ be any compact subset of $M_k$. As $-\psi+\psi_m\ge 0$ for any $m\ge 1$, by inequality \eqref{cor 1.23 estimate}, we know that
$$\sup_m \int_K \{F_{m,k},F_{m,k}\}_{h_E}<+\infty.$$
Let $\omega$ be a hermitian metric on $M$. By H\"older inequality, when $a>0$ is small enough, we have
\begin{equation}\label{holder inq estimate}
\sup_m \int_K (|F_{m,k}|^2_{\omega,h_E}e^{-\psi})^adV_{\omega}\le \bigg(\sup_m \int_K \{F_{m,k},F_{m,k}\}_{h_E}\bigg)^a\bigg(\int_K e^{-\frac{a}{1-a}\psi}\bigg)^{1-a}<+\infty.
\end{equation}
It follows from $h_Ee^{-\psi}$  is singular Nakano semi-positive in the sense of Definition \ref{singular nak} and inequality \eqref{holder inq estimate} that we know $\sup_{m}\int_K|F_{m,k}|^2_{\tilde{h}}<+\infty$, where $\tilde{h}$ is a smooth metric of $E$ .
Hence $F_{m,k}$ compactly converges to a holomorphic $E$-valued $(n,0)$-form $F_k$ on $M_k$ and we know that $F_k=f\wedge ds$ on $S_{reg}\cap M_k$. It follows from Fatou's lemma and inequality \eqref{cor 1.23 estimate} that we have
\begin{equation}\label{cor ohsawa mtoinfty}
  c_{n}\int_{M_k}\{F_k,F_k\}_{h_E}
\le 2\pi  c_{n-1}\int_{S_{reg}\cap M_k}\{f,f\}_{h_E}e^{-\psi}<+\infty.
\end{equation}

Let $\tilde{K}$ be any compact subset of $M$. By H\"older inequality, when $\tilde{a}>0$ is small enough, we have
\begin{equation}\label{holder inq estimate 2}
\sup_k \int_{\tilde{K}} (|F_{k}|^2_{\omega,h_E}e^{-\psi})^{\tilde{a}}dV_{\omega}\le \bigg(\sup_k \int_{\tilde{K}} \{F_{k},F_{k}\}_{h_E}\bigg)^{\tilde{a}}\bigg(\int_{\tilde{K}} e^{-\frac{\tilde{a}}{1-\tilde{a}}\psi}\bigg)^{1-\tilde{a}}<+\infty.
\end{equation}
It follows from $h_Ee^{-\psi}$  is singular Nakano semi-positive in the sense of Definition \ref{singular nak} and inequality \eqref{holder inq estimate 2} that we know $\sup_{k}\int_{\tilde{K}}|F_{k}|^2_{\tilde{h}}<+\infty$.
Hence $F_{k}$ compactly converges to a holomorphic $E$-valued $(n,0)$-form $F$ on $M$ and we know that $F=f\wedge ds$ on $S_{reg}$. It follows from Fatou's lemma and inequality
\eqref{cor ohsawa mtoinfty} that we have
\begin{equation}\nonumber
  c_{n}\int_{M}\{F,F\}_{h_E}
\le 2\pi  c_{n-1}\int_{S_{reg}}\{f,f\}_{h_E}e^{-\psi}<+\infty.
\end{equation}

Corollary \ref{ohsawa3 sin ver} has been proved.
\end{proof}

\vspace{.1in} {\em Acknowledgements}. The authors would like to thank Professor Xiangyu Zhou for his encouragement.
The authors would like to thank Dr. Shijie Bao for checking the manuscript.
The first author and the second author were supported by National Key R\&D Program of China 2021YFA1003100.
The first author was supported by NSFC-11825101, NSFC-11522101 and NSFC-11431013. The second author was supported by China Postdoctoral Science Foundation 2022T150687.

\bibliographystyle{references}
\bibliography{xbib}

\begin{thebibliography}{100}

\bibitem{BG}S.J. Bao and Q.A. Guan, $L^2$ extension and effectiveness of strong openness property. Acta Mathematica Sinica, English Series, submitted, 2021, see also
https://www.researchgate.net/publication/353802916.

 \bibitem{BG2}S.J. Bao and Q.A. Guan, $L^2$
extension and effectiveness of $L^p$, strong openness property.
Acta Mathematica Sinica, English Series, submitted, 2021.

 \bibitem{BGBoundary1}S.J. Bao and Q.A. Guan, Modules at boundary points, fiberwise Bergman kernels, and log-subharmonicity, arXiv:2204.01413.


 \bibitem{BGBoundary2}S.J. Bao and Q.A. Guan, Modules at boundary points, fiberwise Bergman kernels, and log-subharmonicity  II--on Stein manifolds, arXiv:2205.08044.

				\bibitem{BGY-concavity5}S.J. Bao, Q.A. Guan and Z. Yuan, Concavity property of minimal $L^2$ integrals with Lebesgue measurable gain V-----fibrations over open Riemann surfaces. https://www.researchgate.net/publication/357506625.

				\bibitem{BGY-concavity6}S.J. Bao, Q.A. Guan and Z. Yuan, Concavity property of minimal $L^2$ integrals with Lebesgue measurable gain VI-----fibrations over products of open Riemann surfaces. https://www.researchgate.net/publication/357621727.


\bibitem{BGMY7}S.J. Bao, Q.A. Guan, Z.T. Mi and Z. Yuan, Concavity property of minimal $L^2$ integrals with Lebesgue measurable gain \uppercase\expandafter{\romannumeral7}-negligible weights, arXiv:2204.07266,  see also
https://www.researchgate.net/publication/358215153.

 \bibitem{berndtsson1996}B. Berndtsson, The extension theorem of Ohsawa-Takegoshi and the theorem of Donnelly-Fefferman, Ann. L'Inst. Fourier (Grenoble) 46 (1996), no. 4, 1083–1094.



\bibitem{berndtsson annals}B. Berndtsson, Curvature of vector bundles associated to holomorphic fibrations, Annals of
Math, 169 (2009), 531–560.



\bibitem{BL}B. Berndtsson and L. Lempert, A proof of the Ohsawa–Takegoshi theorem with
sharp estimates, J. Math. Soc. Japan 68 (2016), no. 4, 1461-1472.
\bibitem{berndtsson paun}Bo Berndtsson and Mihai P\u{a}un. Bergman kernels and the pseudoeffectivity of relative canonical bundles. Duke
Math. J., 145(2):341–378, 2008.

\bibitem{Blocki2013}Z. Blocki, On the Ohsawa–Takegoshi extension theorem, Univ. Iagel. Acta Math.
No. 50 (2012), 53-61.

\bibitem{Blocki-inv}Z. Blocki,  Suita conjecture and the Ohsawa-Takegoshi extension theorem, Invent Math., 193 (2013), 149-158.

\bibitem{BM-1991}  E. Bierstone and P. D. Milman, A simple constructive proof of canonical resolution of singularities, pp. 11–30 in Effective methods in algebraic geometry (Castiglioncello,1990), edited by T. Mora and C. Traverso, Progr. Math.94, Birkhäuser, Boston, 1991.

\bibitem{Boucksom note}S\'ebastien Boucksom, Singularities of plurisubharmonic functions and multiplier ideals, electronically accessible at http://sebastien.boucksom.perso.math.cnrs.fr/notes/L2.pdf.

    \bibitem{CA}de Cataldo, Mark Andrea A., Singular Hermitian metrics on vector bundles,
   J. Reine Angew. Math. 502 (1998), 93-122.


\bibitem{Demailly82}J.-P Demailly, Estimations $L^2$ pour l'op\'erateur $\bar\partial$ d'un fibr\'e vectoriel holomorphe semi-positif au-dessus d'une vari\'et\'e k\"ahl\'erienne compl\`ete.(French) \textit{$L^2$ estimates for the $\bar\partial$-operator of a semipositive
holomorphic vector bundle over a complete K\"ahler manifold}, Ann. Sci. \'Ecole Norm. Sup. (4) 15 (1982), no. 3, 457-511.

\bibitem{DemaillyReg}J.-P Demailly, Regularization of closed positive currents of type (1,1) by the flow of a Chern connection, Actes du Colloque en l'honneur de P.Dolbeault (Juin 1992),\'edit\'e par H.Skoda et J.-M Tr\'epreau, Aspect of Mathematics, Vol.E26, Vieweg,1994,105-126.

     \bibitem{Demaillyshm}   Damailly J P. Singular Hermitian metrics on positive line bundles. Complex algebraic varieties. Springer, Berlin, Heidelberg, 1992: 87-104.

\bibitem{DemaillyManivel}         Demailly, J.-P.: On the Ohsawa-Takegoshi-Manivel $L^2$ extension theorem. In: Proceedings of
the Conference in Honour of the 85th Birthday of Pierre Lelong, Paris, Sept 1997. Progress
in Mathematics. Birkhäuser, Basel (2000)

\bibitem{DemaillyAG}J.-P Demailly, Analytic Methods in Algebraic Geometry, Higher Education Press, Beijing, 2010.



   		\bibitem{Demaillybook}J.-P Demailly, Complex analytic and differential geometry, electronically accessible at
https://www-fourier.ujf-grenoble.fr/\textasciitilde demailly/manuscripts/agbook.pdf.
		
		
		\bibitem{D2016}J.-P. Demailly, Extension of holomorphic functions defined on non reduced analytic subvarieties, pp. 191–222 in The legacy of Bernhard Riemann after one hundred and fifty years, Vol. I, edited by L. Ji et al., Adv. Lect. Math. (ALM) 35, International Press, Somerville, MA, 2016.

\bibitem{DHP}Demailly J P, Hacon C D, P\u{a}un M. Extension theorems, non-vanishing and the existence of good minimal models. Acta mathematica, 2013, 210(2): 203-259.

        \bibitem{DNWZ}Fusheng Deng, Jiafu Ning, Zhiwei Wang, and Xiangyu Zhou. Positivity of holomorphic vector bundles in terms
of $L^p$-conditions of $\bar{\partial}$, 2020. arXiv:2001.01762v1.

\bibitem{GMPEK} Q.A. Guan and Z.T. Mi, Concavity of minimal $L^2$ integrals related to multiplier ideal sheaves, Peking Mathematical Journal, published online, https://doi.org/10.1007/s42543-021-00047-5.

        \bibitem{GMY-boundary2}Q.A. Guan, Z.T. Mi and Z. Yuan, Boundary points, minimal $L^2$ integrals and concavity property \uppercase\expandafter{\romannumeral2}: on weakly pseudoconvex K\"ahler manifolds, arXiv:2203.07723v2.

   \bibitem{GMY-boundary5}Q.A. Guan, Z.T. Mi and Z. Yuan, Boundary points, minimal $L^2$ integrals and concavity property \uppercase\expandafter{\romannumeral5}---vector bundles, arXiv:2206.00443v2.


\bibitem{GY-concavity1}Q.A. Guan and Z. Yuan, Concavity property of minimal $L^2$ integrals with Lebesgue measurable gain, preprint. https://www.researchgate.net/publication/353794984.

    		
		\bibitem{GY-concavity3}Q.A. Guan and Z. Yuan, Concavity property of minimal $L^2$ integrals with Lebesgue measurable gain III-----open Riemann surfaces, https://www.researchgate.net/publication/356171464.
		
		
		\bibitem{GY-concavity4}Q.A. Guan and Z. Yuan, Concavity property of minimal $L^2$ integrals with Lebesgue measurable gain IV-----product of open Riemann surfaces, https://www.researchgate.net/publication/356786874.

\bibitem{guan-zhou CRMATH2012}Q.A. Guan and X.Y. Zhou. Optimal constant problem in the $L^2$ extension theorem. C. R. Math. Acad. Sci. Paris Ser I, 2012,
350: 753–756.

\bibitem{guan-zhou CRMATH2014}
Q.A. Guan, X.Y. Zhou, An $L^2$
extension theorem with optimal estimate, C. R. Acad. Sci.
Paris, Ser. I 352 (2014), no. 2, 137–141.

       \bibitem{guan-zhou13ap}Q.A. Guan and X.Y. Zhou, A solution of an $L^{2}$ extension problem with an optimal estimate and applications, Ann. of Math. (2) 181 (2015), no. 3, 1139--1208.

           \bibitem{GZsci}Q.A. Guan and X.Y Zhou,
Optimal constant in an $L^2$ extension problem and a proof of a conjecture of Ohsawa, Sci. China Math., 2015, 58(1):35-59.

 \bibitem{GZZCRMATH}Guan, Q.A., Zhou, X.Y., Zhu, L.F.: On the Ohsawa-Takegoshi $L^2$ extension theorem and the
twisted Bochner-Kodaira identity. C. R. Math. Acad. Sci. Paris 349(13–14), 797–800 (2011).

          \bibitem{HPS} Hacon C, Popa M, Schnell C. Algebraic fiber spaces over abelian varieties: around a recent theorem by Cao and P\u{a}un, Local and global methods in algebraic geometry, 143–195, Contemp. Math., 712, Amer. Math. Soc. [2018].


\bibitem{Hironaka} H. Hironaka, Resolution of singularities of an algebraic variety over a field ofcharacteristic zero: I, II,  Ann. of Math. (2) 79 (1964), 109-203, 205-326.

\bibitem{Hormander} L.H\"{o}rmander, An introduction to complex analysis in several complex variables, third edition, NorthHolland Mathematical Library, 7, North-Holland Publishing Co., Amsterdam, 1990.

    \bibitem{inayama}Takahiro Inayama. Nakano positivity of singular hermitian metrics and vanishing theorems of Demailly-Nadel-Nakano type, arXiv:2004.05798v3.

    \bibitem{OT87}Ohsawa, Takeo; Takegoshi, Kensho,
On the extension of $L^2$ holomorphic functions. Math. Z. 195 (1987), no. 2, 197–204.

\bibitem{Ohsawa2}T. Ohsawa, On the extension of $L^2$ holomorphic functions. II. Publ. Res. Inst. Math. Sci.
24 (1988), no. 2, 265–275.
\bibitem{Ohsawa3}T. Ohsawa,
On the extension of $L^2$ holomorphic functions. III. negligible weights, Math. Z.
219 (1995), no. 2, 215–225.

 \bibitem{Ohsawa4}Ohsawa, T.: On the extension of $L^2$ holomorphic functions. IV. A new density concept. In:
Geometry and Analysis on Complex Manifolds, pp. 157-170. World Scientific, River Edge
(1994)

\bibitem{Ohsawa5}T. Ohsawa,
On the extension of $L^2$ holomorphic functions. V. Effects of generalization, Nagoya Math. J. 161 (2001), 1-21. Erratum to: ``On the extension of $L^2$ holomorphic functions. V. Effects of generalization" [Nagoya Math. J. 161 (2001), 1-21]. Nagoya Math.J. 163 (2001), 229.

\bibitem{Ohsawabook}T. Ohsawa,
$L^2$ Approaches in Several Complex variables: Towards the Oka-Cartan Theory with Precise Bounds. Springer Monographs in Mathematics. Tokyo: Springer, 2018.

\bibitem{paun07}Mihai P\u{a}un, Siu’s invariance of plurigenera: a one-tower proof. J. Differential Geom. 76 (2007),
no. 3, 485–493.

\bibitem{paun takayama}
Mihai P\u{a}un and Shigeharu Takayama. Positivity of twisted relative pluricanonical bundles and their direct images.
J. Algebraic Geom., 27(2):211–272, 2018.



   \bibitem{Raufi1}Hossein Raufi. Singular hermitian metrics on holomorphic vector bundles. Ark. Mat., 53(2):359–382, 2015.

       \bibitem{siuINV}Y.-T. Siu, Invariance of plurigenera, Invent. Math. 134 (1998), 661–673.

       \bibitem{siu96}Y.T. Siu, The Fujita conjecture and the extension theorem of Ohsawa-Takegoshi, Geometric Complex Analysis, World Scientific, Hayama, 1996, pp.223-277.

     \bibitem{siu02}       Y.-T. Siu, Extension of twisted pluricanonical sections with plurisubharmonic weight and
invariance of semipositively twisted plurigenera for manifolds not necessarily of general type.
Complex geometry (G\"{o}ttingen, 2000), 223–277, Springer, Berlin, (2002).

\bibitem{siu04}Y.T. Siu, Invariance of plurigenera and torsion-freeness of direct image sheaves of pluricanonical bundles, in Finite or Infinite Dimensional Complex Analysis and Applications (Kluwer, Boston, MA, 2004), Adv. Complex Anal. Appl. 2, pp. 45-83.

   \bibitem{Suita1972}  N. Suita, Capacities and kernels on Riemann surfaces, Arch. Ration. Mech. Anal., 46 (1972), 212–217.

    \bibitem{voisin}Voisin, Claire, Hodge theory and complex algebraic geometry. I., Cambridge Studies in Advanced Mathematics, 76. Cambridge University Press, Cambridge, 2007. x+322 pp. ISBN: 978-0-521-71801-1.

        \bibitem{zhou-abel}Xiangyu Zhou,  A survey on $L^2$ extension problem. Complex geometry and dynamics, 291-309, Abel Symp., 10, Springer, Cham, 2015.

   \bibitem{ZhouZhu-jdg}X.Y. Zhou and L.F. Zhu, An optimal $L^2$ extension theorem on weakly pseudoconvex  K\"ahler manifolds. (English summary)
J. Differential Geom. 110 (2018), no. 1, 135-186.


\bibitem{ZZ2019}X.Y Zhou and L.F.Zhu, Optimal $L^2$ extension of sections from subvarieties in weakly pseudoconvex manifolds. Pacific J. Math. 309 (2020), no. 2, 475-510.

    \bibitem{ZGZ}L.F. Zhu, Q.A. Guan, X.Y. Zhou, On the Ohsawa–Takegoshi $L^2$
extension theorem and the Bochner–Kodaira identity with non-smooth twist factor, J. Math.
Pures Appl. 97 (2012), 579–601, MR 2921602, Zbl 1244.32005.

\end{thebibliography}

\end{document}